\documentclass[11pt]{article}
\usepackage[numbers,square]{natbib}
\usepackage{amsmath}
\usepackage{amsthm}
\usepackage{amsfonts}
\usepackage{amssymb}
\usepackage{siunitx}
\usepackage{commath}
\usepackage{graphicx}
\usepackage{xspace}
\usepackage{color}
\usepackage[lofdepth,lotdepth]{subfig}
\usepackage{stmaryrd}
\usepackage{url}
\usepackage{amsthm}
\usepackage{graphbox}
\usepackage{footnote}
\makesavenoteenv{tabular}
\makesavenoteenv{table}
\usepackage[hidelinks]{hyperref}
\hypersetup{colorlinks = true, urlcolor = blue,
  linkcolor = blue, citecolor = blue}
\usepackage{cleveref}

\usepackage[margin=0.7in]{geometry}
\makeatletter
\newcommand{\tnorm}{\@ifstar\@tnorms\@tnorm}
\newcommand{\@tnorms}[1]{%
  \left|\mkern-1.5mu\left|\mkern-1.5mu\left|
   #1
  \right|\mkern-1.5mu\right|\mkern-1.5mu\right|
}
\newcommand{\@tnorm}[2][]{%
  \mathopen{#1|\mkern-1.5mu#1|\mkern-1.5mu#1|}
  #2
  \mathclose{#1|\mkern-1.5mu#1|\mkern-1.5mu#1|}
}
\makeatother

\newcommand{\jump}[1]{\llbracket #1 \rrbracket}
\newcommand{\av}[1]{\{\!\!\{#1\}\!\!\}}
\newtheorem{theorem}{Theorem}
\newtheorem{lemma}{Lemma}

\newtheorem{remark}{Remark}
\newtheorem{definition}{Definition}
\title{Robust preconditioning for an HDG discretization of the
  time-dependent Stokes equations}
\author{
  E. Henr\'iquez\thanks{Department of Applied Mathematics, University of
    Waterloo, ON, Canada (\url{ehenriqu@uwaterloo.ca}),
    \url{http://orcid.org/0000-0002-0243-0368}}
  \and
  J. J. Lee\thanks{Department of Mathematics, Baylor University,
    TX, USA (\url{jeonghun_lee@baylor.edu}),
    \url{https://orcid.org/0000-0001-5201-8526}}
  \and
  S. Rhebergen\thanks{Department of Applied Mathematics, University of
    Waterloo, ON, Canada (\url{srheberg@uwaterloo.ca}),
    \url{http://orcid.org/0000-0001-6036-0356}}}
\begin{document}
\maketitle
\begin{abstract}
  We present parameter-robust preconditioners for linear systems that
  arise after applying static condensation to a hybridizable
  discontinuous Galerkin (HDG) discretization of the time-dependent
  Stokes problem. Building upon the theoretical framework introduced
  in our previous work [SIAM Journal on Scientific Computing,
  47(6):A3212-A3238, 2025], we extend the analysis to derive new
  preconditioners that remain robust with respect to all physical and
  discretization parameters. The construction relies on first
  establishing uniform well-posedness of the HDG formulation (before
  static condensation) through appropriately defined norms. Based on
  this result, we identify sufficient conditions that a norm on the
  face space must satisfy to guarantee parameter-robustness of the
  resulting preconditioner for the statically condensed HDG
  system. Numerical experiments in two and three dimensions verify our
  theoretical results.
\end{abstract}
\section{Introduction}
\label{s:introduction}

The time-dependent Stokes equations play a key role in the modelling
of viscous flows, for example, in semi-implicit time-stepping schemes
for the numerical approximation of the Navier--Stokes equations. Fast
solvers for this problem are essential in large-scale simulations in
which preconditioners play an essential role in the design of
efficient iterative methods. Preconditioning the time-dependent Stokes
equations has been extensively studied for non-hybridized
formulations. For example, the classical work \cite{cahouet1988some}
introduced one of the first preconditioners for the time-dependent
Stokes equations. Other approaches include those proposed in
\cite{bramble1997iterative,olshanskii2006uniform} and
\cite{mardal2004uniform} in which the latter developed a
preconditioner within the framework of norm-equivalent
parameter-preconditioning, as reviewed by Mardal and Winther in
\cite{mardal2011preconditioning}. The main difficulty in designing
preconditioners for the time-dependent Stokes equations is related to
the ratio $\nu/\tau$, where $\tau$ is the inverse of a discrete time
step and $\nu$ is the viscosity parameter. As this ratio approaches
zero, the resulting system becomes equivalent to the Darcy problem;
therefore, a preconditioner that is effective for the steady Stokes
system is not parameter-robust when this ratio is small.

Hybridizable discontinuous Galerkin (HDG) methods were introduced by
Cockburn et al. \cite{cockburn2009unified} with the aim of reducing
the high computational cost associated with solving the linear systems
arising from classical discontinuous Galerkin (DG) methods. This is
achieved by introducing additional unknowns defined on cell faces and
applying static condensation in which cell unknowns are eliminated
from the linear system.

The development of fast and robust solvers for the reduced system
obtained by applying static condensation to an HDG discretization
remains an active area of research. Various approaches have been
investigated, including multigrid methods
\cite{cockburn2014multigrid,he2021local,lu2022analysis,lu2022homogeneous,lu2024homogeneous},
domain decomposition techniques
\cite{tu2020analysis,tu2021bddc,zhang2022robust}, Schwarz methods
\cite{lu2023two,yu2024nonoverlapping}, auxiliary space preconditioners
\cite{fu2021uniform}, and AIR algebraic multigrid for space-time
problems \cite{sivas2021air}. Regarding the development of block
preconditioners, a variety of strategies have been proposed for
different problems, including the Stokes problem
\cite{henriquez2025parameter,rhebergen2018preconditioning,rhebergen2022preconditioning},
Darcy flow \cite{henriquez2025parameter}, Biot's equation
\cite{henriquez2025preconditioning,kraus2021uniformly}, quasi-static
multiple-network poroelastic theory model (MPET)
\cite{kraus2023hybridized}, linear elasticity and generalized Stokes
problems \cite{fu2023uniform}, and the stationary Navier--Stokes
problem
\cite{lindsay2025preconditioning,sivas2021preconditioning,southworth2020fixed}.

In our previous work \cite{henriquez2025parameter}, we presented an
extension of the Mardal--Winther framework
\cite{mardal2011preconditioning} to design parameter-robust
preconditioners for the reduced system arising from symmetric
hybridizable discretizations. The technique consists of first
determining parameter dependent inner products and their induced norms
in which the non-condensed linear system is both uniformly bounded and
inf-sup stable. This inner product then defines a parameter-robust
preconditioner for the non-condensed system (see
\cref{s:mardalwinther}). Then, by eliminating the cell
degrees-of-freedom from this preconditioner, we obtain a
preconditioner for the reduced problem. This reduced preconditioner is
therefore the Schur complement of the matrix representation associated
with the inner products used to define the non-condensed
preconditioner. This Schur complement defines an inner product on the
face-space (see \cref{eq:innerprodXh}) which in turn induces a
``face-norm''. In \cite{henriquez2025parameter} we identified a
condition for this face-norm in relation to the norm on the
non-condensed space (see \cref{eq:henriquez2025}) -- we will refer to
this condition as the \emph{face-norm condition} -- that if satisfied
then the reduced preconditioner will be a parameter-robust
preconditioner for the reduced HDG discretization.

This manuscript consists of three main results. The first main result
is a generalization of the face-norm condition. In particular, we now
show that \emph{any} inner product that induces a norm on the faces
that satisfies a generalized face-norm condition defines a
parameter-robust preconditioner for a reduced system resulting from a
symmetric hybridizable discretization (see
\cref{thm:cond-precon}). The advantage of this generalization is that
we are no longer restricted by using the Schur complement of a
preconditioner originally derived for a hybridizable system before
static condensation as preconditioner for the reduced problem.  This
generalization is especially useful when dealing with intersections
and sums of Hilbert spaces. The second main result of this manuscript
is proving uniform well-posedness of the non-condensed HDG
discretization of the time-dependent Stokes equations. This is a key
step in the construction of parameter-robust preconditioners for the
reduced HDG discretization. The third main result of this manuscript
is the application and verification of the aforementioned general
preconditioning framework for symmetric hybridizable discretizations
to derive new parameter-robust preconditioners for the reduced HDG
discretization of the time-dependent Stokes equations.

This manuscript is organized as follows. In \cref{s:precongen} we
present our first main result, i.e., a general preconditioning
framework for symmetric hybridizable discretizations; this section
presents a generalization of our work in
\cite{henriquez2025parameter}. In \cref{s:t-stokes} we present the
time-dependent Stokes problem and its HDG discretization, while our
second main result, i.e., a proof of uniform well-posedness of this
HDG method before static condensation, is given in
\cref{s:uniform-wp}. The third main result, parameter-robust
preconditioners for the reduced form of the hybridizable
discretization of the time-dependent Stokes equations, is presented in
\cref{s:precon}. Our theoretical findings are verified by numerical
experiments in \cref{s:numex} while conclusions are drawn in
\cref{s:conclusions}.

\section{General preconditioning framework}
\label{s:precongen}

In this section we present a general preconditioning framework for
symmetric hybridizable discretizations. In \cref{s:mardalwinther} we
first briefly summarize the Mardal--Winter framework as presented in
\cite{mardal2011preconditioning}. We then summarize its extension to
hybridizable discretizations, as presented in
\cite{henriquez2025parameter}, in \cref{ss:preconframeworkscold}. A
new generalization of this extension is presented in
\cref{ss:preconframeworkscnew}. We start by introducing some notation.

Denote by $\mathcal{T}_h$ a mesh of mesh size $h$ and denote by
$\boldsymbol{X}_h$ and $\boldsymbol{X}_h^*$ a finite element space
defined on $\mathcal{T}_h$ and its dual, respectively. The pairing of
$\boldsymbol{X}_h^*$ and $\boldsymbol{X}_h$ is denoted by
$\langle \cdot, \cdot
\rangle_{\boldsymbol{X}_h^*,\boldsymbol{X}_h}$. An inner product
defined on $\boldsymbol{X}_h$ is denoted by
$(\cdot, \cdot)_{\boldsymbol{X}_h}$. The norm induced by this inner
product is denoted by $\norm[0]{\cdot}_{\boldsymbol{X}_h}$.

Let $\mathcal{L}(\boldsymbol{X}_h,\boldsymbol{Y}_h)$ be the set of
bounded linear operators mapping $\boldsymbol{X}_h$ to
$\boldsymbol{Y}_h$. Let
$A \in \mathcal{L}(\boldsymbol{X}_h, \boldsymbol{X}_h)$. We have the
following definitions:
\begin{align*}
  \norm[0]{A}_{\mathcal{L}(\boldsymbol{X}_h, \boldsymbol{X}_h)}
  &= \sup_{\boldsymbol{x}_h,\boldsymbol{y}_h \in \boldsymbol{X}_h}
    \frac{(A\boldsymbol{x}_h, \boldsymbol{y}_h)_{\boldsymbol{X}_h}}{\norm[0]{\boldsymbol{x}_h}_{\boldsymbol{X}_h} \norm[0]{\boldsymbol{y}_h}_{\boldsymbol{X}_h}},
  \\
  \norm[0]{A^{-1}}^{-1}_{\mathcal{L}(\boldsymbol{X}_h, \boldsymbol{X}_h)}
  &= \inf_{\boldsymbol{x}_h\in\boldsymbol{X}_h} \sup_{\boldsymbol{y}_h \in \boldsymbol{X}_h}
    \frac{(A\boldsymbol{x}_h, \boldsymbol{y}_h)_{\boldsymbol{X}_h}}{\norm[0]{\boldsymbol{x}_h}_{\boldsymbol{X}_h} \norm[0]{\boldsymbol{y}_h}_{\boldsymbol{X}_h}}.
\end{align*}
The condition number of $A$ is given by
$\kappa(A) := \norm[0]{A}_{\mathcal{L}(\boldsymbol{X}_h,
  \boldsymbol{X}_h)} \norm[0]{A^{-1}}_{\mathcal{L}(\boldsymbol{X}_h,
  \boldsymbol{X}_h)}$. It is known that the convergence rate of a
Krylov subspace method applied to a symmetric problem of the form
$A\boldsymbol{x}_h = \boldsymbol{b}_h$ can be bounded in terms of
$\kappa(A)$.

\subsection{The Mardal--Winther framework}
\label{s:mardalwinther}

Let $a_h(\cdot, \cdot)$ be a symmetric bilinear form
on $\boldsymbol{X}_h \times \boldsymbol{X}_h$ and consider the
problem: Given $\boldsymbol{f}_h \in \boldsymbol{X}_h^*$, find
$\boldsymbol{x}_h \in \boldsymbol{X}_h$ such that
\begin{equation}
  \label{eq:genproblem}
  a_h(\boldsymbol{x}_h, \boldsymbol{y}_h)
  = \langle \boldsymbol{f}_h, \boldsymbol{y}_h \rangle_{\boldsymbol{X}_h^*,\boldsymbol{X}_h}
  \qquad \forall \boldsymbol{y}_h \in \boldsymbol{X}_h.
\end{equation}
The discrete problem \cref{eq:genproblem} can equivalently be written
as
\begin{equation}
  \label{eq:genproblemequiv}
  A \boldsymbol{x}_h = \boldsymbol{f}_h \quad \text{in } \boldsymbol{X}_h^*,
\end{equation}
for the unknown $\boldsymbol{x}_h \in \boldsymbol{X}_h$ in which $A$
is the operator defined by
$\langle A\boldsymbol{x}_h, \boldsymbol{y}_h
\rangle_{\boldsymbol{X}_h^*,\boldsymbol{X}_h} =
a_h(\boldsymbol{x}_h,\boldsymbol{y}_h)$ for all
$\boldsymbol{x}_h,\boldsymbol{y}_h \in \boldsymbol{X}_h$.

Assume $a_h(\cdot, \cdot)$ is uniformly bounded and inf-sup stable in
$\norm[0]{\cdot}_{\boldsymbol{X}_h}$, i.e., assume there exist uniform
constants (constants independent of the mesh-size and problem
parameters) $c_1, c_2 > 0$ such that
\begin{subequations}
  \label{eq:uniformstabinfsup}
  \begin{align}
    \label{eq:uniformstabinfsup_a}
    a_h(\boldsymbol{x}_h, \boldsymbol{y}_h)
    &\le c_1 \norm[0]{\boldsymbol{x}_h}_{\boldsymbol{X}_h} \norm[0]{\boldsymbol{y}_h}_{\boldsymbol{X}_h} && \forall \boldsymbol{x}_h, \boldsymbol{y}_h \in \boldsymbol{X}_h,
    \\
    \label{eq:uniformstabinfsup_b}
    \inf_{\boldsymbol{x}_h\in \boldsymbol{X}_h} \sup_{\boldsymbol{y}_h \in \boldsymbol{X}_h}
    \frac{a_h(\boldsymbol{x}_h, \boldsymbol{y}_h)}{\norm[0]{\boldsymbol{x}_h}_{\boldsymbol{X}_h} \norm[0]{\boldsymbol{y}_h}_{\boldsymbol{X}_h}} &\ge c_2.
  \end{align}
\end{subequations}
The Mardal--Winther framework \cite{mardal2011preconditioning} shows
that if a preconditioner
$P^{-1} : \boldsymbol{X}_h^* \to \boldsymbol{X}_h$ is defined by
\begin{equation}
  \label{eq:preconMWP}
  (P^{-1}\boldsymbol{f}_h, \boldsymbol{y}_h)_{\boldsymbol{X}_h} =
  \langle \boldsymbol{f}_h, \boldsymbol{y}_h
  \rangle_{\boldsymbol{X}_h^*,\boldsymbol{X}_h}
  \qquad \forall \boldsymbol{y}_h \in \boldsymbol{X}_h,
  \quad \forall \boldsymbol{f}_h \in \boldsymbol{X}_h^*,
\end{equation}
then the condition number of $P^{-1}A$ is bounded by $c_1/c_2$. Since
$c_1$ and $c_2$ are uniform constants, the condition number is
independent of discretization and problem parameters and so $P^{-1}$
is a parameter-robust preconditioner for problems in which $A$ is
symmetric.

\subsection{An extension of the Mardal--Winther framework to
  hybridizable systems}
\label{ss:preconframeworkscold}

Assume $\boldsymbol{X}_h := X_h \times \bar{X}_h$ so that any
$\boldsymbol{x}_h \in \boldsymbol{X}_h$ can be written as
$\boldsymbol{x}_h = (x_h, \bar{x}_h)$ with $x_h \in X_h$ and
$\bar{x}_h \in \bar{X}_h$. We can then write \cref{eq:genproblemequiv}
as
\begin{equation}
  \label{eq:hybridsystem}
  \begin{bmatrix}
    A_{11} & A_{21}^T \\ A_{21} & A_{22}
  \end{bmatrix}
  \begin{bmatrix}
    x_h \\ \bar{x}_h
  \end{bmatrix}
  =
  \begin{bmatrix}
    f_h \\ \bar{f}_h
  \end{bmatrix},
\end{equation}
with $A_{11} : X_h \to X_h^*$, $A_{21} : X_h \to \bar{X}_h^*$, and
$A_{22} : \bar{X}_h \to \bar{X}_h^*$. If \cref{eq:hybridsystem} is
obtained from a hybridizable discretization, and assuming $x_h$ are
the local degrees of freedom, then $A_{11}$ is block
diagonal. Eliminating $x_h$ from \cref{eq:hybridsystem} we obtain the
following reduced problem for $\bar{x}_h$:
\begin{equation}
  \label{eq:reducedprob}
  S_A \bar{x}_h = \bar{b}_h
\end{equation}
where $S_A := A_{22} - A_{21}A_{11}^{-1}A_{21}^T$ is the Schur
complement of the matrix in \cref{eq:hybridsystem} and
$\bar{b}_h := \bar{f}_h - A_{21}A_{11}^{-1}f_h$.

Let $P : \boldsymbol{X}_h \to \boldsymbol{X}_h^*$, defined in
\cref{eq:preconMWP}, have the same block structure as $A$, i.e.,
\begin{equation}
  \label{eq:preconPgen}
  P =
  \begin{bmatrix}
    P_{11} & P_{21}^T \\ P_{21} & P_{22}
  \end{bmatrix},
\end{equation}
and let $S_P := P_{22} - P_{21}P_{11}^{-1}P_{21}^T$ be the Schur
complement of $P$. Assume $S_P:\bar{X}_h \to \bar{X}_h^*$ is a
positive operator, i.e., $S_P$ is symmetric and
$\langle S_P\bar{x}_h, \bar{x}_h \rangle_{\bar{X}_h^*,\bar{X}_h} > 0$
for all $\bar{x}_h \in \bar{X}_h \backslash \cbr[0]{0}$. Then $S_P$
defines an inner product on $\bar{X}_h$:
\begin{equation}
  \label{eq:innerprodXh}
  (\bar{x}_h, \bar{y}_h)_{\bar{X}_h}
  := \langle S_P \bar{x}_h, \bar{y}_h \rangle_{\bar{X}_h^*,\bar{X}_h}
  \quad \forall \bar{x}_h, \bar{y}_h \in \bar{X}_h.
\end{equation}
In \cite[Theorem 2.3]{henriquez2025parameter} we proved that if there
exists a uniform constant $c_l > 0$ such that
\begin{equation}
  \label{eq:henriquez2025}
  \norm[0]{(-A_{11}^{-1}A_{21}^T\bar{x}_h, \bar{x}_h)}_{\boldsymbol{X}_h}
  \le c_l \norm[0]{\bar{x}_h}_{\bar{X}_h} \quad \forall \bar{x}_h \in \bar{X}_h,
\end{equation}
where $\norm[0]{\cdot}_{\bar{X}_h}$ is the norm induced by the inner
product $(\cdot, \cdot)_{\bar{X}_h}$, then $S_P$ is a parameter-robust
preconditioner for the reduced problem \cref{eq:reducedprob}. We will
refer to \cref{eq:henriquez2025} as the \emph{face-norm condition}.

\subsection{A new generalization of the Mardal--Winther framework for
  hybridizable systems}
\label{ss:preconframeworkscnew}

The following theorem presents a generalization to the face-norm
condition \cref{eq:henriquez2025} to obtain parameter-robust
preconditioners for \cref{eq:reducedprob}.

\begin{theorem}
  \label{thm:cond-precon}
  Let $A$, $S_A$, and $P$ be the operators defined in
  \cref{ss:preconframeworkscold}. Furthermore, let
  $\bar{P} : \bar{X}_h \to \bar{X}_h^*$ be any operator that defines
  an inner product $(\cdot, \cdot)_{\bar{X}_h}$ on $\bar{X}_h$ in the
  sense that
  \begin{equation*}
    \langle \bar{P} \bar{x}_h, \bar{y}_h\rangle_{\bar{X}_h^*, \bar{X}_h}
    = (\bar{x}_h, \bar{y}_h)_{\bar{X}_h}
    \quad \forall \bar{x}_h, \bar{y}_h \in \bar{X}_h.
  \end{equation*}
  Assume that \cref{eq:hybridsystem} is uniformly well-posed in the
  $\norm[0]{\cdot}_{\boldsymbol{X}_h}$-norm and that $A_{11}$ is
  invertible. If there exist uniform constants $c_l, c_u > 0$ such
  that
  \begin{subequations}
    \label{eq:cond-precon-1}
    \begin{align}
      \label{eq:cond-precon-1a}
      \norm[0]{\boldsymbol{x}_h}_{\boldsymbol{X}_h}
      & \geq c_l \norm[0]{\bar{x}_h}_{\bar{X}_h},
      \\
      \label{eq:cond-precon-1b}
      \norm[0]{(- A_{11} A_{21}^T \bar{x}_h, \bar{x}_h)}_{\boldsymbol{X}_h}
      & \leq c_u \norm[0]{\bar{x}_h}_{\bar{X}_h},
    \end{align}
  \end{subequations}
  for all $\boldsymbol{x}_h \in \boldsymbol{X}_h$, then
  \begin{equation}
    \label{eq:cond-precon-2}
    \norm[0]{\bar{P}^{-1}S_A}_{\mathcal{L}(\bar{X}_h, \bar{X}_h)},
    \quad \norm[0]{(\bar{P}^{-1}S_A)^{-1}}_{\mathcal{L}(\bar{X}_h , \bar{X}_h)},
  \end{equation}
  are uniformly bounded. 
\end{theorem}
\begin{proof}
  The proof to show that
  $\norm[0]{\bar{P}^{-1}S_A}_{\mathcal{L}(\bar{X}_h , \bar{X}_h)}$ is
  uniformly bounded follows identical steps as used in the proof of
  \cite[Theorem 2.3]{henriquez2025parameter} and is therefore
  omitted. We therefore only prove that
  $\norm[0]{(\bar{P}^{-1}S_A)^{-1}}_{\mathcal{L}(\bar{X}_h ,
    \bar{X}_h)}$ is uniformly bounded.

  Since $A$ is well-posed \cref{eq:uniformstabinfsup}
  holds. Therefore,
  \begin{align}
    \label{eq:cond-precon-aux2}
    \begin{split}
      c_2 & \leq \inf_{\boldsymbol{x}_h \in \boldsymbol{X}_h}
            \sup_{\boldsymbol{y}_h \in \boldsymbol{X}_h}
            \dfrac{\langle A\boldsymbol{x}_h, \boldsymbol{y}_h\rangle_{\boldsymbol{X}_h^*, \boldsymbol{X}_h}}{\norm[0]{\boldsymbol{x}_h}_{\boldsymbol{X}_h} \norm[0]{\boldsymbol{y}_h}_{\boldsymbol{X}_h}}
      \\
          &= \inf_{\boldsymbol{x}_h \in \boldsymbol{X}_h}
            \sup_{\boldsymbol{y}_h \in \boldsymbol{X}_h}
            \dfrac{\langle A_{11}(x_h + A_{11}^{-1}A_{21}^T \bar{x}_h), y_h + A_{11}^{-1} A_{21}^T \bar{y}_h\rangle_{X_h^*, X_h}
            + \langle S_A \bar{x}_h, \bar{y}_h\rangle_{\bar{X}_h^*, \bar{X}_h}}{\norm[0]{\boldsymbol{x}_h}_{\boldsymbol{X}_h} \norm[0]{\boldsymbol{y}_h}_{\boldsymbol{X}_h}}
      \\
          &\leq \inf_{\bar{x}_h \in \bar{X}_h}
            \sup_{(y_h, \bar{y}_h) \in X_h \times \bar{X}_h}
            \dfrac{\langle S_A \bar{x}_h, \bar{y}_h\rangle_{\bar{X}_h^*, \bar{X}_h}}{\norm[0]{(- A_{11}^{-1} A_{21}^T \bar{x}_h, \bar{x}_h)}_{\boldsymbol{X}_h} \norm[0]{(y_h, \bar{y}_h)}_{\boldsymbol{X}_h}},
    \end{split}
  \end{align}
  where the last inequality is a consequence of choosing
  $x_h = - A_{11}^{-1} A_{21}^T \bar{x}_h$. Using
  \cref{eq:cond-precon-1a} we find
  \begin{align*}
    c_l^2 c_2
    &\leq \inf_{\bar{x}_h \in \bar{X}_h}
      \sup_{\bar{y}_h \in \bar{X}_h}
      \frac{\langle S_A \bar{x}_h, \bar{y}_h\rangle_{\bar{X}_h^*, \bar{X}_h}}{\norm[0]{\bar{x}_h}_{\bar{X}_h} \norm[0]{\bar{y}_h}_{\bar{X}_h}}
    \\
    &= \inf_{\bar{x}_h \in \bar{X}_h} 
      \sup_{\bar{y}_h \in \bar{X}_h}
      \frac{(\bar{P}^{-1} S_A \bar{x}_h, \bar{y}_h)_{\bar{X}_h}}{\norm[0]{\bar{x}_h}_{\bar{X}_h} \norm[0]{\bar{y}_h}_{\bar{X}_h}}
      = \norm[0]{(\bar{P}^{-1} S_A)^{-1}}_{\mathcal{L}(\bar{X}_h, \bar{X}_h)}^{-1},
  \end{align*}
  so that the result follows.
\end{proof}

\begin{remark}
  Note that if the conditions of \cref{thm:cond-precon} are satisfied,
  then the condition number $\kappa(\bar{P}^{-1}S_A)$ is uniformly
  bounded, i.e., $\bar{P}^{-1}$ is a parameter-robust
  preconditioner. We further emphasize that \cref{thm:cond-precon} is
  a generalization of \cite[Theorem
  2.3]{henriquez2025parameter}. Indeed, let $S_P$ be as defined in
  \cref{ss:preconframeworkscold}. Then \cite[Theorem
  2.3]{henriquez2025parameter} follows by choosing $\bar{P} = S_P$ in
  \cref{thm:cond-precon}.
\end{remark}

\section{The time-dependent Stokes equations and its discretization}
\label{s:t-stokes}

The time-dependent Stokes equations are given by
\begin{subequations}
  \label{eq:tstokes-model-1}
  \begin{align}
    \label{eq:tstokes-model-1a}
    \partial_t u - \nabla \cdot (\nu \nabla u) + \nabla p  & = \tilde{f} && \text{in }\Omega\times I,
    \\
    \label{eq:tstokes-model-1b}
    \nabla \cdot u & = 0 && \text{in }\Omega\times I,
    \\
    \label{eq:tstokes-model-1c}
    u & = 0 && \text{on }\partial \Omega\times I,
    \\
    \label{eq:tstokes-model-1d}
    u & = u_0 && \text{on } \Omega,
  \end{align}
\end{subequations}
where $\Omega \subset \mathbb{R}^d$, with $d=2$ or $d=3$, is a bounded
polygonal domain, $I = (0, T]$ is the time-interval of interest in
which $T>0$ is the final time, $u$ is the fluid velocity, $p$ is the
pressure (which is assumed to have zero mean), $\tilde{f}$ is a given
external force, $\nu>0$ is the constant viscosity parameter, and $u_0$
is a prescribed divergence-free initial velocity.

Discretizing the time-dependent Stokes equations by backward Euler
results in the following system of equations that needs to be solved
at each time-step:
\begin{subequations}
  \label{eq:tstokes-model-be}
  \begin{align}
    \label{eq:tstokes-model-bea}
    \tau u - \nabla \cdot (\nu \nabla u) + \nabla p  & = f && \text{in }\Omega,
    \\
    \label{eq:tstokes-model-beb}
    \nabla \cdot u & = 0 && \text{in }\Omega,
    \\
    \label{eq:tstokes-model-bec}
    u & = 0 && \text{on }\partial \Omega,
  \end{align}
\end{subequations}
where $\tau = 1/\Delta t$, with $\Delta t$ the time-step, and
$f := \tilde{f} + \tau \tilde{u}$ in which $\tilde{u}$ is the solution
from the previous time-step.

We discretize \cref{eq:tstokes-model-be} by the pressure-robust HDG
method presented in
\cite{rhebergen2017analysis,rhebergen2018hybridizable}. To describe
this method, let $\mathcal{T}_h$ denote a quasi-uniform mesh for the
domain consisting of simplicial cells $K$, denote by $h$ the global
mesh size, let $\mathcal{F}_h$, $\mathcal{F}_h^{int}$ and
$\mathcal{F}_h^{bnd}$ denote the sets of all faces, interior faces,
and boundary faces, respectively, and let $\Gamma_0$ denote the union
of all faces. Consider the following velocity and pressure cell and
face finite element spaces:
\begin{align*}
  V_h &:= \cbr[0]{ v_h \in [L^2(\Omega)]^d:\ v_h \in [\mathbb{P}_{k}(K)]^d,
        \ \forall K \in \mathcal{T}_h },
  \\
  \bar{V}_h &:= \cbr[0]{ \bar{v}_h \in [L^2(\Gamma_0)]^d: \
              \bar{v}_h \in [\mathbb{P}_{k}(F)]^d, \ \forall F \in \mathcal{F}_h,\
              \bar{v}_h = 0 \text{ on } \partial \Omega },
  \\
  Q_h &:= \cbr[0]{ q_h \in L^2(\Omega):\ q_h \in \mathbb{P}_{k-1}(K),
        \ \forall K \in \mathcal{T}_h } \cap L^2_0(\Omega),
  \\
  \bar{Q}_h &:= \cbr[0]{ \bar{q}_h \in L^2(\Gamma_0): \
              \bar{v}_h \in \mathbb{P}_{k}(F), \ \forall F \in \mathcal{F}_h },
\end{align*}
where $\mathbb{P}_{k}(K)$ and $\mathbb{P}_{k}(F)$ denote the sets of
polynomials of degree at most $k$ on a cell $K$ and face $F$ and
$L_0^2(\Omega)$ is the space of functions in $L^2(\Omega)$ with zero
mean. For ease of notation we write
$\boldsymbol{V}_h := V_h \times \bar{V}_h$,
$\boldsymbol{Q}_h := Q_h \times \bar{Q}_h$, and
$\boldsymbol{X}_h := \boldsymbol{V}_h \times \boldsymbol{Q}_h$, with
elements $\boldsymbol{v}_h := (v_h,\bar{v}_h) \in \boldsymbol{V}_h$,
$\boldsymbol{q}_h := (q_h, \bar{q}_h) \in \boldsymbol{Q}_h$, and
$\boldsymbol{y}_h := (\boldsymbol{v}_h, \boldsymbol{q}_h) \in
\boldsymbol{X}_h$.

We define $(u, v)_K := \int_K u \odot v \dif x$ and
$\langle u, v \rangle_{\partial K} := \int_{\partial K} u \odot v \dif
s$ where $\odot$ is multiplication if $u,v$ are scalar functions, the
dot product if $u,v$ are vector functions, and the Frobenius inner product if
$u,v$ are matrix functions. We then define
$(u,v)_{\mathcal{T}_h} := \sum_{K\in \mathcal{T}_h}(u,v)_K$,
$\langle u,v \rangle_{\partial \mathcal{T}_h} := \sum_{K \in
  \mathcal{T}_h} \langle u,v \rangle_{\partial K}$.

To define the HDG method we require the following bilinear forms for
$\boldsymbol{u}_h,\boldsymbol{v}_h \in \boldsymbol{V}_h$ and
$\boldsymbol{p}_h,\boldsymbol{q}_h \in \boldsymbol{Q}_h$:
\begin{subequations}
  \label{eq:bilinear-forms}
  \begin{align}
    \label{eq:bilinear-forms-a}
    d_h(\boldsymbol{u}_h, \boldsymbol{v}_h)
    :=& (\nu \nabla u_h, \nabla v_h)_{\mathcal{T}_h}
        + \langle \nu \eta h_K^{-1} (u_h - \bar{u}_h),
        v_h - \bar{v}_h \rangle_{\partial \mathcal{T}_h}
    \\ \nonumber
    & - \langle \nu \nabla u_hn, v_h - \bar{v}_h\rangle_{\partial\mathcal{T}_h}
      - \langle \nu \nabla v_hn, u_h - \bar{u}_h\rangle_{\partial\mathcal{T}_h},
    \\
    \label{eq:bilinear-forms-b}
    b_h(v_h, \boldsymbol{q}_h) 
    :=& - (q_h, \nabla \cdot v_h)_{\mathcal{T}_h}
        + \langle \bar{q}_h, v_h \cdot n \rangle_{\partial \mathcal{T}_h}
    \\ \nonumber
    =& (\nabla q_h, v_h)_{\mathcal{T}_h} - \langle q_h-\bar{q}_h, v_h \cdot n \rangle_{\partial \mathcal{T}_h}, 
    \\
    \label{eq:bilinear-forms-c}
    a_h((\boldsymbol{u}_h,\boldsymbol{p}_h), (\boldsymbol{v}_h,\boldsymbol{q}_h))
    :=& \tau (u_h, v_h)_{\mathcal{T}_h} 
        + d_h(\boldsymbol{u}_h, \boldsymbol{v}_h)
        + b_h(v_h, \boldsymbol{p}_h)
        + b_h(u_h, \boldsymbol{q}_h),
  \end{align}
\end{subequations}
where $h_K$ is the diameter of a cell $K \in \mathcal{T}_h$, $n$ is
the outward unit normal vector on $\partial K$, and $\eta > 1$ is the
interior penalty parameter.

\begin{definition}[The HDG method]
  The HDG method for \cref{eq:tstokes-model-be} is defined as: Find
  $\boldsymbol{x}_h := (\boldsymbol{u}_h, \boldsymbol{p}_h) \in
  \boldsymbol{X}_h$ such that
  \begin{equation}
    \label{eq:hdgmethod}
    a_h(\boldsymbol{x}_h, \boldsymbol{y}_h)
    = (f, v_h)_{\mathcal{T}_h}
    \qquad \forall \boldsymbol{y}_h = (\boldsymbol{v}_h, \boldsymbol{q}_h) \in
    \boldsymbol{X}_h.
  \end{equation}
\end{definition}

\section{Uniform well-posedness}
\label{s:uniform-wp}

\subsection{Inner products and norms}
\label{ss:innerprodsnorms}

We start by defining the following parameter-dependent inner products
for $\boldsymbol{u}_h,\boldsymbol{v}_h \in \boldsymbol{V}_h$ and
$\boldsymbol{p}_h,\boldsymbol{q}_h \in \boldsymbol{Q}_h$:
\begin{subequations}
  \label{eq:inner-p}
  \begin{align}
    \label{eq:inner-p-c}
    (\boldsymbol{u}_h, \boldsymbol{v}_h)_{v,1}
    &:= (\nabla u_h, \nabla v_h)_{\mathcal{T}_h}
      + \eta \langle h_K^{-1}(u_h - \bar{u}_h), v_h - \bar{v}_h \rangle_{\partial \mathcal{T}_h},
    \\
    \label{eq:inner-p-d}
    (\boldsymbol{p}_h, \boldsymbol{q}_h)_{q,0}
    &:= (p_h, q_h)_{\mathcal{T}_h}
      + \eta^{-1} \langle h_K \bar{p}_h, \bar{q}_h\rangle_{\partial \mathcal{T}_h},
    \\
    \label{eq:inner-p-d0}
    (\boldsymbol{p}_h, \boldsymbol{q}_h)_{q,0*}
    &:= (p_h, q_h)_{\mathcal{T}_h}
      + \eta^{-1} \langle h_K (p_h-\bar{p}_h), (q_h-\bar{q}_h)\rangle_{\partial \mathcal{T}_h},
    \\
    \label{eq:inner-p-e}
    (\boldsymbol{p}_h, \boldsymbol{q}_h)_{q,1}
    &:= (\nabla p_h, \nabla q_h)_{\mathcal{T}_h}
      + \eta \langle h_K^{-1}(p_h - \bar{p}_h), q_h - \bar{q}_h \rangle_{\partial \mathcal{T}_h},
    \\
    \label{eq:inner-p-a}
    (\boldsymbol{u}_h, \boldsymbol{v}_h)_v
    &:= \tau(u_h, v_h)_{\mathcal{T}_h}
      + \nu (\boldsymbol{u}_h, \boldsymbol{v}_h)_{v,1},
    \\
    \label{eq:inner-p-b}
    (\boldsymbol{p}_h, \boldsymbol{q}_h)_q
    &:= \inf_{\substack{\boldsymbol{p}_h = \boldsymbol{p}_{1,h} + \boldsymbol{p}_{2,h}
    \\ \boldsymbol{q}_h = \boldsymbol{q}_{1,h} + \boldsymbol{q}_{2,h}}}
    \del[1]{ \nu^{-1} (\boldsymbol{p}_{1,h}, \boldsymbol{q}_{1,h})_{q,0}
    + \tau^{-1} (\boldsymbol{p}_{2,h}, \boldsymbol{q}_{2,h})_{q,1}},
    \\
    \label{eq:inner-p-bstar}
    (\boldsymbol{p}_h, \boldsymbol{q}_h)_{q*}
    &:= \inf_{\substack{\boldsymbol{p}_h = \boldsymbol{p}_{1,h} + \boldsymbol{p}_{2,h}
    \\ \boldsymbol{q}_h = \boldsymbol{q}_{1,h} + \boldsymbol{q}_{2,h}}}
    \del[1]{ \nu^{-1} (\boldsymbol{p}_{1,h}, \boldsymbol{q}_{1,h})_{q,0*}
    + \tau^{-1} (\boldsymbol{p}_{2,h}, \boldsymbol{q}_{2,h})_{q,1}}.
  \end{align}
\end{subequations}
These inner products induce norms which are denoted by
$\tnorm{\boldsymbol{v}_h}_{v,1}$, $\tnorm{\boldsymbol{q}_h}_{q,0}$,
$\tnorm{\boldsymbol{q}_h}_{q,0*}$, $\tnorm{\boldsymbol{q}_h}_{q,1}$,
$\tnorm{\boldsymbol{v}_h}_v$, $\tnorm{\boldsymbol{q}_h}_q$, and
$\tnorm{\boldsymbol{q}_h}_{q*}$ respectively. Observe that
$\tnorm{\boldsymbol{v}_h}_v$ and $\tnorm{\boldsymbol{q}_h}_q$ are
discrete versions of norms on
$\tau [L^2(\Omega)]^d \cap \nu [H_0^1(\Omega)]^d$ and
$\tau^{-1} (H^1(\Omega)\cap L_0^2(\Omega)) + \nu^{-1} L_0^2(\Omega)$,
see \cite{mardal2004uniform}. Observe also that
$\tnorm{\boldsymbol{q}_h}_{q,0}$ and $\tnorm{\boldsymbol{q}_h}_{q,0*}$
are equivalent norms up to mesh- and parameter-independent
constants. Likewise, $\tnorm{\boldsymbol{q}_h}_{q}$ and
$\tnorm{\boldsymbol{q}_h}_{q*}$ are equivalent norms up to uniform
constants. We furthermore define the following inner product on
$\boldsymbol{X}_h$:
\begin{equation}
  \label{eq:inner-p-X}
  ((\boldsymbol{u}_h, \boldsymbol{p}_h), (\boldsymbol{v}_h, \boldsymbol{q}_h))_{\boldsymbol{X}_h}
  := (\boldsymbol{u}_h, \boldsymbol{v}_h)_v
  + (\boldsymbol{p}_h, \boldsymbol{q}_h)_q.
\end{equation}
Its induced norm is defined as
$\tnorm{\boldsymbol{x}_h}_{\boldsymbol{X}_h}^2 :=
\tnorm{\boldsymbol{v}_h}_v^2 + \tnorm{\boldsymbol{q}_h}_q^2$.

Let $F$ be an interior face shared by cells $K^+$ and $K^-$ and denote
by $w^{\pm}$ the traces of $w$ on $F$ taken from the interior of
$K^{\pm}$. The usual jump operator is defined as $\jump{w}=w^+ - w^-$
on interior faces and as $\jump{w} = w$ on boundary faces. For
$v_h \in V_h$ we define the usual DG norm
\begin{equation*}
  \norm[0]{v_h}_{dg}^2 
  := \norm[0]{\nabla v_h}_{\mathcal{T}_h}^2
  + \sum_{F \in \mathcal{F}_h} h_F^{-1}\norm[0]{\jump{v_h}}_F^2,
\end{equation*}
with $h_F$ a measure of a face $F \in \mathcal{F}_h$. The average
operator is defined as $\av{w} = \tfrac{1}{2}(w^++w^-)$ on interior
faces and as $\av{w}=0$ on boundary faces. We then note that, since
$\mathcal{T}_h$ is quasi-uniform, there exist uniform constants
$c_{dg,1}, c_{dg,2} > 0$ such that
\begin{equation}
  \label{eq:dgequiv}
  c_{dg,1} \norm[0]{v_h}_{dg} \le \tnorm{(v_h,\av{v_h})}_{v,1} \le c_{dg,2} \norm[0]{v_h}_{dg}
  \qquad \forall v_h \in V_h.
\end{equation}
The first inequality was shown in \cite[eq. (5.8)]{wells2011analysis}
while the second inequality follows from
\begin{align*}
  \norm[0]{\av{v_h} - v_h}_{\partial \mathcal{T}_h}
  &= \sum_{F \in \mathcal{F}_h^{int}} \norm[0]{\tfrac{1}{2}v_h^- - \tfrac{1}{2}v_h^+}_{F}
    + \sum_{F \in \mathcal{F}_h^{int}} \norm[0]{\tfrac{1}{2}v_h^+ - \tfrac{1}{2}v_h^-}_{F}
    + \sum_{F \in \mathcal{F}_h^{bnd}} \norm[0]{v_h}_F
  \lesssim \sum_{F \in \mathcal{F}_h} \norm[0]{\jump{v_h}}_F,
\end{align*}
where $x \lesssim y$ denotes that there exists a uniform constant
$c > 0$ such that $x \le c y$. We will also use $x \gtrsim y$ to
denote $x \ge c y$.

\subsection{Uniform inf-sup condition for $b_h$}
\label{ss:infsupbh}

In this section we prove the following theorem.

\begin{theorem}[inf-sup stability of $b_h$]
  \label{thm:infsupbh}
  There exists a uniform constant $c_3 > 0$, that depends on $\eta$, such that
  \begin{equation}
    \label{eq:infsup3}
    \sup_{0 \neq \boldsymbol{v}_h \in \boldsymbol{V}_h}
    \frac{b_h(v_h, \boldsymbol{q}_h)}{\tnorm{\boldsymbol{v}_h}_v} 
    \geq c_{3} \tnorm{\boldsymbol{q}_h}_q \quad \forall \boldsymbol{q}_h \in \boldsymbol{Q}_h.
  \end{equation}
\end{theorem}
\begin{remark}
  We point out that parameter-independent inf-sup constants of $b_h$
  using the norms
  $(\tau^{1/2}\norm[0]{\cdot}_{\mathcal{T}_h},
  \tau^{-1/2}\tnorm{\cdot}_{q,1})$ and
  $(\nu^{1/2}\tnorm{\cdot}_{v,1}, \nu^{-1/2}\tnorm{\cdot}_{q,0})$,
  proven in \cite[Lemma 4]{kraus2021uniformly} and \cite[Lemma
  1]{rhebergen2018preconditioning} respectively, do not imply
  Theorem~\ref{thm:infsupbh} by a purely functional analytic approach
  of interpolation spaces. This is because the linear map from
  $\boldsymbol{V}_h$ to $\boldsymbol{Q}_h^*$ induced by $b_h$ has a
  kernel space $Z_h \not =\{0\}$, and taking interpolation spaces and
  taking quotient spaces do not commute topologically.
\end{remark}

Before proving \cref{thm:infsupbh} we first present a few useful
results. The local degrees of freedom of the Brezzi--Douglas--Marini
space $BDM_k(K)$ implies the following two lemmas
(cf. \cite[Proposition~2.3.2]{boffi2013mixed} and \cite[Proposition
2.10]{du2019invitation}).
\begin{lemma}[Interior Local Basis]
  \label{lem:interior-local}
  There exists a subspace $B_K([P_{k}(K)]^d) \subset [P_{k}(K)]^d$
  that consists of functions with zero normal trace on $\partial
  K$. Then, for any given $q_h \in P_{k-1}(K)$, there exists a unique
  $v_{int} \in B_K([P_{k}(K)]^d)$ such that
  $(v_{int}, w_h)_K = (\nabla q_h, w_h)_K$ for all
  $w_h \in [P_{k-2}(K)]^d$ and all other local degrees of freedom are
  zero.
\end{lemma}

\begin{lemma}[Orthogonal Lifting]
  \label{lem:bdm-lifting}
  Let $\bar{r}_h \in P_k(F)$ be a given polynomial on face
  $F \subset \partial K$. There exists a unique local lifting operator
  $L_F(\bar{r}_h) \in [P_k(K)]^d$ such that its normal trace on $F$ is
  exactly $\bar{r}_h$, its normal trace on $\partial K \backslash F$
  is zero, and all interior degrees of freedom are zero. This lifting
  operator has the following properties:
  \begin{itemize}
  \item[{\rm (i)}] $(L_F(\bar{r}_h), \nabla q_h)_K \equiv 0$ for any
    $q_h \in P_{k-1}(K)$ since
    $\nabla q_h \subset [P_{k-2}(K)]^d \subset \mathcal{N}_{k-2}(K)$,
    where $\mathcal{N}_{k-2}$ is the N\'ed\'elec space;
  \item[{\rm (ii)}]
    $\norm[0]{L_F(\bar{r}_h)}_K \lesssim h_K^{\frac{1}{2}}
    \norm[0]{\bar{r}_h}_{F}$ and
    $\norm[0]{\nabla L_F(\bar{r}_h)}_K \lesssim h_K^{-\frac{1}{2}}
    \norm[0]{\bar{r}_h}_{F}$.
  \end{itemize}
\end{lemma}

\begin{lemma}[Lifting jump]
  \label{lem:liftjump}
  Let $\bar{r}_h \in P_k(F)$ be a given polynomial on face
  $F \subset \partial K$ and let $L_F(\bar{r}_h) \in [P_k(K)]^d$ be
  the local lifting operator defined in \cref{lem:bdm-lifting}. Then
  \begin{align*}
    \norm[0]{\jump{L_F(\bar{r}_h)}}_F^2
    &\lesssim \norm[0]{\bar{r}_h^+}_F^2 + \norm[0]{\bar{r}_h^-}_F^2
    && \text{if $F \in \mathcal{F}_h^{int}$,}
    \\
    \norm[0]{\jump{L_F(\bar{r}_h)}}_F^2
    &\lesssim \norm[0]{\bar{r}_h}_F^2
    && \text{if $F \in \mathcal{F}_h^{bnd}$.}
  \end{align*}
\end{lemma}
\begin{proof}
  Since the lifting $L_F(\bar{r}_h)$ is determined solely by its
  normal component on the face $F$
  ($L_F(\bar{r}_h) \cdot n = \bar{r}_h$, see \cref{lem:bdm-lifting}),
  the normal component of $L_F(\bar{r}_h)$ dominates
  $(L_F(\bar{r}_h))^t$, the tangential component of $L_F(\bar{r}_h)$,
  on the face $F$. Therefore, on an interior face,
  \begin{align*}
    \norm[0]{\jump{L_F(\bar{r}_h)}}_F^2
    \lesssim & \norm[0]{L_F^+(\bar{r}_h^+)}_F^2 + \norm[0]{L_F^-(\bar{r}_h^-)}_F^2
    \\
    \lesssim & \norm[0]{(L_F^+(\bar{r}_h^+)\cdot n^+)n^+}_F^2 + \norm[0]{(L_F^+(\bar{r}_h^-))^t}_F^2
               + \norm[0]{(L_F^-(\bar{r}_h^+)\cdot n^-)n^-}_F^2 + \norm[0]{(L_F^-(\bar{r}_h^+))^t}_F^2
    \\
    \lesssim & \norm[0]{\bar{r}_h^+}_F^2 + \norm[0]{\bar{r}_h^-}_F^2.
  \end{align*}
  Similar arguments hold on a boundary face.
\end{proof}

The following lemma considers a splitting of
$\boldsymbol{q}_h \in \boldsymbol{Q}_h$.

\begin{lemma}
  \label{lem:infimum}
  For $\boldsymbol{q}_h \in \boldsymbol{Q}_h$ there exist
  $\boldsymbol{q}_{0}^*, \boldsymbol{q}_{1}^* \in \boldsymbol{Q}_h$
  such that
  $\boldsymbol{q}_{0}^* + \boldsymbol{q}_{1}^* = \boldsymbol{q}_h$ and
  \begin{subequations}
    \begin{align}
      \label{eq:infimum}
      \tnorm{\boldsymbol{q}_h}_{q*}^2
      &= \nu^{-1} \tnorm{\boldsymbol{q}_0^*}_{q,0*}^2
      + \tau^{-1} \tnorm{\boldsymbol{q}_1^*}_{q,1}^2,
      \\
      \label{eq:optimality}
      \nu^{-1} (\boldsymbol{q}_{0}^*, \boldsymbol{r}_h )_{q,0*}
      &= \tau^{-1} (\boldsymbol{q}_{1}^*, \boldsymbol{r}_h )_{q,1}
        \qquad \forall \boldsymbol{r}_h \in \boldsymbol{Q}_h.
    \end{align}    
  \end{subequations}
\end{lemma}
\begin{proof}
  Let
  $F(\boldsymbol{\rho}_h):= \nu^{-1}
  \tnorm{\boldsymbol{q}_h-\boldsymbol{\rho}_h }_{q,0*}^2 + \tau^{-1}
  \tnorm{\boldsymbol{\rho}_h}_{q,1}^2$. The existence of a pair attaining
  the infimum of $F(\cdot)$ is because $\boldsymbol{Q}_h$ is a finite
  dimensional vector space.

  Next, by the optimality condition $DF(\boldsymbol{q}_{1}^*) = 0$ at
  $\boldsymbol{\rho}_h = \boldsymbol{q}_{1}^*$, where $DF$ denotes the
  G\^{a}teaux derivative of $F$. Writing out the norms, we have for
  all $\boldsymbol{r}_h \in \boldsymbol{Q}_h$
  \begin{equation*}
    0
    = DF(\boldsymbol{q}_1^*)      
    = \lim_{\varepsilon \to 0} \frac{F(\boldsymbol{q}_1^*
      + \varepsilon \boldsymbol{r}_h) - F(\boldsymbol{q}_1^*)}{\varepsilon}
    = - 2\nu^{-1} ( \boldsymbol{r}_h, \boldsymbol{q}_0^* )_{q,0*}
    + 2\tau^{-1} ( \boldsymbol{r}_h, \boldsymbol{q}_1^* )_{q,1},
  \end{equation*}
  proving \cref{eq:optimality}.
\end{proof}

\begin{lemma}
  \label{lem:q0starHmin1}
  Let $\boldsymbol{q}_0^*, \boldsymbol{q}_1^* \in \boldsymbol{Q}_h$ be
  as defined in \cref{lem:infimum}. Then
  \begin{equation*}
    \norm[0]{q_0^*}_{H^{-1}}
    := \sup_{r \in H_0^1(\Omega)} \frac{(q_0^*,r)_{\mathcal{T}_h}}{\norm[0]{r}_{H^1_0(\Omega)}}
    \lesssim h \tnorm{\boldsymbol{q}_0^*}_{q,0*} + \nu \tau^{-1} \eta^{1/2} \tnorm{\boldsymbol{q}_1^*}_{q,1}.
  \end{equation*}
\end{lemma}
\begin{proof}
  Let $\Pi_Q$ be the cell-wise $L^2$ projection into $Q_h$. Then,
  using \cref{eq:optimality},
  \begin{align*}
    (q_0^*, r)_{\mathcal{T}_h}
    =& (q_0^*, \Pi_{Q} r)_{\mathcal{T}_h}
       = (q_0^*, \Pi_{Q} r - (\Pi_{Q} r)_{\text{avg}} )_{\mathcal{T}_h}
    \\
    =& \del[1]{ (q_0^*, \bar{q}_0^*), (\Pi_{Q} r - (\Pi_{Q} r)_{\text{avg}}, \av{\Pi_{Q} r - (\Pi_{Q} r)_{\text{avg}}}) }_{q,0*} 
    \\
     & - \eta^{-1} \langle h_K (q_0^* - \bar{q}_0^*), \Pi_Q r - \av{\Pi_Q r} \rangle_{\partial \mathcal{T}_h}
    \\
    =& \frac{\nu}{\tau} \del[1]{ (q_1^*, \bar{q}_1^*), (\Pi_{Q} r - (\Pi_{Q} r)_{\text{avg}}, \av{\Pi_{Q} r - (\Pi_{Q} r)_{\text{avg}}}) }_{q,1} 
    \\
     & - \eta^{-1} \langle h_K( q_0^* - \bar{q}_0^*), \Pi_Q r - \av{\Pi_Q r} \rangle_{\partial \mathcal{T}_h}.
  \end{align*}
  Using the Cauchy--Schwarz inequality,
  \begin{equation}
    \label{eq:q0starrn-n}
    \begin{split}
      (q_0^*, r)_{\mathcal{T}_h}
      \le
      & \frac{\nu}{\tau} \tnorm{\boldsymbol{q}_1^*}_{q,1}
        \tnorm{(\Pi_{Q} r - (\Pi_{Q} r)_{\text{avg}}, \av{\Pi_{Q} r - (\Pi_{Q} r)_{\text{avg}}}) }_{q,1} 
      \\
      & + \eta^{-1}\langle h_K( q_0^* - \bar{q}_0^*), \Pi_Q r - \av{\Pi_Q r} \rangle_{\partial \mathcal{T}_h}.      
    \end{split}
  \end{equation}
  Consider the first term on the right hand side of
  \cref{eq:q0starrn-n} and note that
  \begin{align*}
    \tnorm{(\Pi_{Q} r - (\Pi_{Q} r)_{\text{avg}}, \av{\Pi_{Q} r - (\Pi_{Q} r)_{\text{avg}}})}_{q,1}^2
    =& \norm[0]{\nabla (\Pi_Qr)}_{\mathcal{T}_h}^2
    \\
     &+ \eta \norm[0]{ h_K^{-1/2}(\Pi_{Q} r - \av{\Pi_{Q} r} )}_{\partial\mathcal{T}_h}^2.
  \end{align*}
  Note that
  $\norm[0]{\nabla (\Pi_Qr)}_{\mathcal{T}_h} \lesssim |r|_{H^1(K)}$
  (c.f. \cite[Lemma 1.58]{di2011mathematical}). Furthermore, on an
  interior face $F$ we have
  \begin{equation*}
    (\Pi_Q r - \av{\Pi_Q r})|_{\partial K^+ \cap F} 
    = \tfrac{1}{2} ( (\Pi_Q r)^+ - (\Pi_Q r)^-)
    = \tfrac{1}{2} ( (\Pi_Q r)^+ -r + (r - (\Pi_Q r)^-) ),
  \end{equation*}
  so that, using quasi-uniformity of the mesh and \cite[Lemma
  1.59]{di2011mathematical},
  \begin{equation*}
    \begin{split}
      h_{K^+}^{-1/2}\norm[0]{\Pi_Q r - \av{\Pi_Q r}}_{\partial K^+ \cap F}
      &\le \tfrac{1}{2} h_{K^+}^{-1/2} \del[1]{
        \norm[0]{(\Pi_Q r)^+ - r}_{\partial K^+ \cap F}
        + \norm[0]{(\Pi_Q r)^- - r}_{\partial K^- \cap F} }
      \\
      &\lesssim |r|_{H^1(K^+)} + |r|_{H^1(K^-)}.
    \end{split}
  \end{equation*}
  A similar result holds on boundary faces. The first term on the
  right hand side of \cref{eq:q0starrn-n} is therefore bounded as
  \begin{equation}
    \label{eq:q0starr-rhs1n-n}
    \frac{\nu}{\tau} \tnorm{\boldsymbol{q}_1^*}_{q,1}
    \tnorm{(\Pi_{Q} r - (\Pi_{Q} r)_{\text{avg}}, \av{\Pi_{Q} r - (\Pi_{Q} r)_{\text{avg}}}) }_{q,1}
    \lesssim \frac{\nu}{\tau} \tnorm{\boldsymbol{q}_1^*}_{q,1} \eta^{1/2} \sum_{K \in \mathcal{T}_h}|r|_{H^1(K)}.
  \end{equation}
  For the second term on the right hand side of \cref{eq:q0starrn-n}
  we find:
  \begin{equation}
    \label{eq:q0starr-rhs2n-n}
    \begin{split}
      \eta^{-1}\langle & h_K( q_0^* - \bar{q}_0^*), \Pi_Q r - \av{\Pi_Q r} \rangle_{\partial \mathcal{T}_h}
      \\
                       &\le \eta^{-1}\norm[0]{h_K^{1/2}(q_0^* - \bar{q}_0^*)}_{\partial \mathcal{T}_h}
                         \norm[0]{h_K^{1/2} (\Pi_Q r - \av{\Pi_Q r})}_{\partial \mathcal{T}_h}
      \\
                       &\lesssim h \eta^{-1}\norm[0]{h_K^{1/2}(q_0^* - \bar{q}_0^*)}_{\partial \mathcal{T}_h} \sum_{K \in \mathcal{T}_h} |r|_{H^1(K)}.
    \end{split}
  \end{equation}
  Combining \cref{eq:q0starrn-n,eq:q0starr-rhs1n-n,eq:q0starr-rhs2n-n}
  with the definition of $\norm[0]{q_0^*}_{H^{-1}}$ and using that
  $\eta > 1$, the result follows.
\end{proof}

It is known that there exists a bounded linear operator
$\Pi_{\text{div}} : H(\text{div}, \Omega) \to V_h \cap
H(\text{div},\Omega)$ which is a projection on
$V_h \cap H(\text{div},\Omega)$ such that $\Pi_{\text{div}}$ is
locally $L^2$-bounded and it gives a bounded cochain projection in the
last two spaces of the discrete de Rham complex
(cf. \cite{Arnold-Guzman:2021,Ern-Gudi:2022,Gawlik:2021}). In
particular, if $w \in H(\text{div}, \Omega)$ and
$\nabla \cdot w \in P_{k-1}(\mathcal{T}_h)$, then
$\nabla \cdot \Pi_{\text{div}}w = \nabla \cdot w$. In addition to the
immediate consequence that
$\norm[0]{\Pi_{\text{div}} w}_{L^2(\Omega)} \lesssim
\norm[0]{w}_{L^2(\Omega)}$ we also have the following result.

\begin{lemma}
  \label{lem:pidivwH1}
  Let $w \in [H^1(\Omega)]^d$. Then
  $\norm[0]{\Pi_{\text{div}} w}_{dg} \lesssim
  \norm[0]{w}_{H^1(\Omega)}$.
\end{lemma}
\begin{proof}
  Let $\Pi_0 w$ be the $L^2$ projection of $w$ to
  $[\mathbb{P}_0(K)]^d$ and note that
  $\norm[0]{\nabla \Pi_{\text{div}} w}_K = \norm[0]{\nabla
    (\Pi_{\text{div}} w - \Pi_0 w)}_K \lesssim h_K^{-1} \norm[0]{
    (\Pi_{\text{div}} w - \Pi_0 w)}_K \lesssim h_K^{-1} (\norm[0]{
    \Pi_{\text{div}} w - w}_K + \norm[0]{ w - \Pi_0 w}_K) \lesssim
  \norm[0]{\nabla w}_{N(K)}$, where $N(K)$ is the union of $K$ and its
  adjacent elements. Then,
  $\norm[0]{\nabla \Pi_{\text{div}} w}_{\mathcal{T}_h}^2 =
  \sum_{K\in\mathcal{T}_h} \norm[0]{\nabla \Pi_{\text{div}} w }_K^2
  \lesssim \norm[0]{\nabla w}_{L^2(\Omega)}^2$. We next estimate
  $h_F^{-1} \norm[0]{\jump{\Pi_{\text{div}} w}}_F^2$. Since $w$ is
  single-valued on an interior face $F$, we note that
  $h_F^{-1} \norm[0]{\jump{\Pi_{\text{div}} w}}_F^2 = h_F^{-1}
  \norm[0]{\jump{\Pi_{\text{div}} w - w}}_F^2 \le 2 h_F^{-1}
  (\norm[0]{\Pi_{\text{div}} w^+ - w}_F^2 + \norm[0]{\Pi_{\text{div}}
    w^- - w}_F^2)$. By \cite[(1.19)]{di2011mathematical}, we find that
  $h_F^{-1} \norm[0]{\Pi_{\text{div}} w^+ - w}_F^2 \lesssim h_F^{-1}
  \norm[0]{\Pi_{\text{div}} w - w}_K \norm[0]{\nabla(\Pi_{\text{div}}
    w^+ - w)}_K \lesssim \norm[0]{\nabla w}_{N(K)}^2$. Similar
  arguments hold on a boundary face.
\end{proof}

We next recall that the Bogovski\u{i} operator
$\mathcal{B}: L_0^2(\Omega) \to [H_0^1(\Omega)]^d$ is a bounded linear
right inverse of $\nabla \cdot$ (cf. \cite[Chapter~3]{Galdi:NSE-book})
and its continuous extension is well-defined as a bounded linear map
from $H^{-1}(\Omega)$ to $[L^2(\Omega)]^d$
(cf. \cite{Geissert:2006}). 

We are now ready to prove \cref{thm:infsupbh}.

\begin{proof}[Proof of \cref{thm:infsupbh}]
  To prove \cref{eq:infsup3} we prove the equivalent result that for
  any $\boldsymbol{q}_h \in \boldsymbol{Q}_h$ there exists
  $\boldsymbol{v}_h \in \boldsymbol{V}_h$ such that
  $\tnorm{\boldsymbol{v}_h}_{v} \lesssim
  \tnorm{\boldsymbol{q}_h}_{q*}$ and
  $b_h(v_h, \boldsymbol{q}_h) = \tnorm{\boldsymbol{q}_h}_{q*}^2$.
  
  Given $\boldsymbol{q}_{h}$ let
  $\boldsymbol{q}_{h} = \boldsymbol{q}_{0}^* + \boldsymbol{q}_{1}^*$
  be the optimal splitting that realizes the infimum as stated in
  \cref{lem:infimum}. We denote the partial energies as
  \begin{equation}
    \label{eq:E0-E1-definition}
    E_0 := \nu^{-1} \tnorm{\boldsymbol{q}_{0}^*}_{q,0*}^2,
    \qquad
    E_1 := \tau^{-1} \tnorm{\boldsymbol{q}_{1}^*}_{q,1}^2. 
  \end{equation}
  By the optimality condition \cref{eq:optimality} with
  $\boldsymbol{r}_h = \boldsymbol{q}_{0}^*$ and
  $\boldsymbol{r}_h = \boldsymbol{q}_{1}^*$ we note that
  \begin{equation}
    \label{eq:optimality-identity}
    E_0 = \tau^{-1} ( \boldsymbol{q}_{1}^*, \boldsymbol{q}_{0}^* )_{q,1},
    \qquad 
    E_1 = \nu^{-1} ( \boldsymbol{q}_{1}^*, \boldsymbol{q}_{0}^* )_{q,0*}.
  \end{equation}
  We now prove the theorem by considering the cases
  $\tau h^2 \lesssim \nu$ and $\nu \lesssim \tau h^2$ separately.

  \smallskip
  \noindent \textbf{Case 1 ($\tau h^2 \lesssim \nu$).}  Let
  $v_0 = v_B + v_l$ where
  $v_B := -\nu^{-1} \Pi_{\text{div}} \mathcal{B}(q_0^*)$ and
  $v_l|_K := \sum_{F\subset \partial K} L_F\del[1]{-\nu^{-1} \eta^{-1}
    h_K(q_0^* - \bar{q}_0^*)}$.
  \\
  \underline{Step 1.} We first show that
  $b_h(v_0, \boldsymbol{q}_h) = \tnorm{\boldsymbol{q}_h}_{q*}^2$. By
  \cref{lem:bdm-lifting}(i), $({v}_{l}, \nabla q_0^*)_K = 0$ and
  $({v}_{l}, \nabla q_1^*)_K = 0$. Since
  $\nabla \cdot \Pi_{\text{div}} \mathcal{B}(q_0^*) = q_0^*$ we find
  from \cref{eq:bilinear-forms-b} and \cref{eq:optimality-identity}:
  \begin{align*}
    b_h(v_0, \boldsymbol{q}_0^*)
    & = \nu^{-1} \norm{q_0^*}_{\mathcal{T}_h}^2
      + \nu^{-1} \eta^{-1} \norm[0]{h_K^{1/2}(q_0^* - \bar{q}_0^*)}_{\partial \mathcal{T}_h}^2
      = E_0,
    \\
    b_h(v_0, \boldsymbol{q}_1^*)
    & = \nu^{-1} (q_0^*, q_1^*)_{\mathcal{T}_h}
      + \nu^{-1} \eta^{-1} \langle h_K (q_0^* - \bar{q}_0^*), q_1^* - \bar{q}_1^* \rangle_{\partial \mathcal{T}_h}
      = \nu^{-1} (\boldsymbol{q}_0^*, \boldsymbol{q}_1^*)_{q,0*}
      = E_1,
  \end{align*}
  so that
  $b_h(v_0, \boldsymbol{q}_h) = E_0 + E_1 =
  \tnorm{\boldsymbol{q}_h}_{q*}^2$.
  \\
  \underline{Step 2.} We now show that
  $\nu \norm{v_0}_{dg}^2 + \tau \norm{v_0}_{\mathcal{T}_h}^2 \lesssim
  \eta \tnorm{\boldsymbol{q}_h}_{q*}^2$. First note that by
  \cref{lem:pidivwH1} and properties of $\mathcal{B}$,
  \begin{equation*}
    \norm[0]{v_B}_{dg} =
    \norm[0]{\nu^{-1}\Pi_{\text{div}} \mathcal{B}(q_0^*)}_{dg} \lesssim
    \nu^{-1}\norm[0]{\mathcal{B}(q_0^*)}_{H^1(\Omega)} \lesssim
    \nu^{-1}\norm{q_0^*}_{\mathcal{T}_h}.
  \end{equation*}
  Next, using \cref{lem:bdm-lifting}(ii):
  \begin{equation}
    \label{eq:vl-dgnorm}
    \begin{split}
      \norm[0]{ v_l }_{dg}^2
      =& \norm[0]{\nabla v_l}_{\mathcal{T}_h}^2
         + \sum_{F \in \mathcal{F}_h} h_F^{-1}\norm[0]{\jump{v_l}}_F^2
      \\
      =& \nu^{-2}\eta^{-2} \norm[3]{\nabla (\sum_{F\subset \partial K} L_F\del[1]{-h_K(q_0^* - \bar{q}_0^*)})}_{\mathcal{T}_h}^2
         + \nu^{-2}\eta^{-2} \sum_{F \in \mathcal{F}_h} h_F^{-1}\norm[3]{\jump{\sum_{F\subset \partial K} L_F\del[1]{-h_K(q_0^* - \bar{q}_0^*)}}}_F^2.
    \end{split}
  \end{equation}
  For the first term on the right hand side we note that
  \begin{equation}
    \label{eq:vl-dgnorm-1term}
    \norm[3]{\nabla (\sum_{F\subset \partial K} L_F\del[1]{- h_K(q_0^* - \bar{q}_0^*)})}_{\mathcal{T}_h}^2
    = \norm[0]{h_K^{1/2}(q_0^* - \bar{q}_0^*)}_{\partial \mathcal{T}_h}^2.
  \end{equation}
  For the second term on the right hand side of \cref{eq:vl-dgnorm} we
  note that by \cref{lem:liftjump},
  $\norm[0]{\jump{L_F\del[1]{-h_K(q_0^* - \bar{q}_0^*)}}}_F^2 \lesssim
  \norm[0]{h_{K^+} (q_0^* - \bar{q}_0^*)^+}_F^2 + \norm[0]{h_{K^-}
    (q_0^* - \bar{q}_0^*)^-}_F^2$ on an interior face and
  $\norm[0]{\jump{L_F\del[1]{-h_K(q_0^* - \bar{q}_0^*)}}}_F^2 \lesssim
  \norm[0]{h_{K} (q_0^* - \bar{q}_0^*)}_F^2$ on a boundary
  face. Therefore, combined with
  \cref{eq:vl-dgnorm,eq:vl-dgnorm-1term} we find that
  \begin{equation*}
    \norm[0]{ v_l }_{dg}^2 \lesssim
    \nu^{-2}\eta^{-2}\norm[0]{h_K^{1/2}(q_0^* - \bar{q}_0^*)}_{\partial \mathcal{T}_h}^2.
  \end{equation*}
  Therefore, and using that $\eta > 1$,  
  \begin{equation*}
    \nu \norm[0]{ v_0 }_{dg}^2
    \le
    2 \nu \del[1]{
      \norm[0]{ v_B }_{dg}^2
      + \norm[0]{ v_l }_{dg}^2
    }
    \lesssim
    \nu^{-1} \tnorm{\boldsymbol{q}_0^*}_{q,0*}^2  = E_0.
  \end{equation*}
  To bound $\tau\norm{v_0}_{\mathcal{T}_h}^2$ we need to bound
  $\tau\norm{v_l}_{\mathcal{T}_h}^2$ and
  $\tau\norm{v_B}_{\mathcal{T}_h}^2$. By \cref{lem:bdm-lifting}(ii),
  using that $\eta > 1$, and since $\tau h^2 \lesssim \nu$,
  \begin{equation*}
    \tau\norm{v_l}_{\mathcal{T}_h}^2
    \lesssim \tau \nu^{-2} \eta^{-1} h^2 \tnorm{\boldsymbol{q}_0^*}_{q,0*}^2
    \lesssim \nu^{-1}\tnorm{\boldsymbol{q}_0^*}_{q,0*}^2 = E_0.
  \end{equation*}
  To bound $\tau\norm{v_B}_{\mathcal{T}_h}^2$ we have by properties of
  $\Pi_{\text{div}}$ and $\mathcal{B}$, and using
  \cref{lem:q0starHmin1} and that $\tau h^2 \lesssim \nu$,
  \begin{align*}
    \tau \norm{v_B}_{\mathcal{T}_h}^2
    \lesssim \tau \nu^{-2} \norm{\mathcal{B}(q_0^*)}_{\mathcal{T}_h}^2
    &\lesssim \tau \nu^{-2} \norm{q_0^*}_{H^{-1}}^2
    \\
    &\lesssim \tau \nu^{-2} ( \nu^2 \tau^{-2} \eta \tnorm{\boldsymbol{q}_1^*}_{q,1}^2
      + h^2 \tnorm{\boldsymbol{q}_0^*}_{q,0*}^2)
    \\
    &\lesssim \tau^{-1} \eta \tnorm{\boldsymbol{q}_1^*}_{q,1}^2
      + \nu^{-1} \tnorm{\boldsymbol{q}_0^*}_{q,0*}^2 
      = E_0 + \eta E_1,
  \end{align*}
  so that
  $\nu \norm{v_0}_{dg}^2 + \tau \norm{v_0}_{\mathcal{T}_h}^2 \lesssim
  \eta \tnorm{\boldsymbol{q}_h}_{q*}^2$.
  \\
  \underline{Step 3.} Define $\bar{v}_0 = \av{v_0}$. Then by
  \cref{eq:dgequiv} and Step 2 we have that
  $\tnorm{\boldsymbol{v}_0}_v^2 \lesssim \eta
  \tnorm{\boldsymbol{q}_h}_{q*}^2$.

  \smallskip
  \noindent \textbf{Case 2 ($\nu \lesssim \tau h^2$).} Let
  $v_1 := v_{int} + v_l$ where $v_{int}$ is such that
  $(v_{int}, w_h)_K = \tau^{-1}(\nabla q_1^*, w_h)_K$ for all
  $w_h \in [P_{k-2}(K)]^d$ and all other local degrees of freedom are
  zero (see \cref{lem:interior-local}), and where \\
  $v_l|_K = \sum_{F\subset \partial K} L_F \del[1]{-\tau^{-1} \eta
    h_K^{-1}(q_1^* - \bar{q}_1^*)}$.
  \\
  \underline{Step 4.} We first show that
  $b(v_1, \boldsymbol{q}_h) = \tnorm{\boldsymbol{q}_h}_{q*}^2$. By
  \cref{lem:bdm-lifting}(i) we note that $(v_l,\nabla q_0^*)_K = 0$
  and $(v_l,\nabla q_1^*)_K = 0$. Then
  \begin{equation*}
    b_h(v_1, \boldsymbol{q}_1^*)
    = \tau^{-1} \norm[0]{\nabla q_1^*}_{\mathcal{T}_h}^2
    + \tau^{-1} \eta \norm[0]{h_K^{-1/2}(q_1^* - \bar{q}_1^*)}_{\partial \mathcal{T}_h}^2
    = E_1.
  \end{equation*}
  Evaluating against $\boldsymbol{q}_0^*$ and using
  \cref{eq:optimality-identity} gives
  \begin{align*}
    b_h(v_1, \boldsymbol{q}_0^*)
    =& (v_{int}, \nabla q_0^*)_{\mathcal{T}_h}
       - \langle v_l \cdot n, q_0^* - \bar{q}_0^* \rangle_{\partial \mathcal{T}_h}
    \\
    =& \tau^{-1}(\nabla q_1^*, \nabla q_0^*)_{\mathcal{T}_h}
       + \tau^{-1} \eta \langle h_K^{-1}(q_1^*-\bar{q}_1^*), q_0^*-\bar{q}_0^* \rangle_{\partial \mathcal{T}_h}
    \\
    =& \tau^{-1} \del[1]{ \boldsymbol{q}_1^*, \boldsymbol{q}_0^* }_{q,1} = E_0
  \end{align*}
  Therefore,
  $b(v_1, \boldsymbol{q}_h) = E_0 + E_1 =
  \tnorm{\boldsymbol{q}_h}_{q*}^2$.
  \\
  \underline{Step 5.} We now show that
  $\nu \norm{v_1}_{dg}^2 + \tau \norm[0]{v_1}_{\mathcal{T}_h}^2
  \lesssim \eta \tnorm{\boldsymbol{q}_h}_{q*}^2$. Observe that
  \begin{equation*}
    \norm[0]{v_{int}}_K^2
    = (v_{int},v_{int})_K
    = \tau^{-1} (\nabla q_1^*, v_{int})_K
    \le \tau^{-1} \norm[0]{\nabla q_1^*}_K \norm[0]{v_{int}}_K
  \end{equation*}
  so that
  $\norm[0]{v_{int}}_K \le \tau^{-1} \norm[0]{\nabla q_1^*}_K$. Then,
  using an inverse inequality \cite[Lemma 1.44]{di2011mathematical}
  and a discrete trace inequality \cite[Lemma
  1.46]{di2011mathematical}
  \begin{align*}
    \norm[0]{\nabla v_{int}}_K
    &\lesssim h_K^{-1}\norm[0]{v_{int}}_K \le \tau^{-1} h_K^{-1} \norm[0]{\nabla q_1^*}_K,
    \\
    h_K^{-1/2}\norm[0]{v_{int}}_F &\lesssim h_K^{-1} \norm[0]{v_{int}}_K \le \tau^{-1} h_K^{-1} \norm[0]{\nabla q_1^*}_K,
  \end{align*}
  and so,
  \begin{equation}
    \label{eq:vintdgbound}
    \norm[0]{v_{int}}_{dg}^2
    = \norm[0]{\nabla v_{int}}_{\mathcal{T}_h}^2
    + \sum_{F \in \mathcal{F}_h}h_F^{-1} \norm[0]{\jump{v_{int}}}_F^2
    \lesssim \tau^{-2} h^{-2} \norm[0]{\nabla q_1^*}_{\mathcal{T}_h}^2.
  \end{equation}
  Furthermore, by \cref{lem:liftjump}, we have on interior and
  boundary faces, respectively,
  \begin{align*}
    \norm[0]{\jump{L_F\del[1]{-\tau^{-1}\eta h_K^{-1}(q_1^* - \bar{q}_1^*)}}}_F^2
    &\lesssim \norm[0]{\tau^{-1}\eta h_{K^+}^{-1} (q_1^* - \bar{q}_1^*)^+}_F^2
      + \norm[0]{\tau^{-1} \eta h_{K^-}^{-1} (q_1^* - \bar{q}_1^*)^-}_F^2,
    \\
    \norm[0]{\jump{L_F\del[1]{-\tau^{-1} \eta h_K^{-1}(q_1^* - \bar{q}_1^*)}}}_F^2
    &\lesssim \norm[0]{\tau^{-1} \eta h_{K}^{-1} (q_1^* - \bar{q}_1^*)}_F^2.
  \end{align*}
  Then, also using \cref{lem:bdm-lifting}(ii),
  \begin{equation}
    \label{eq:vldgbound}
    \norm[0]{v_l}_{dg}^2
    = \norm[0]{\nabla v_l}_{\mathcal{T}_h}^2
    + \sum_{F \in \mathcal{F}_h}h_F^{-1} \norm[0]{\jump{v_l}}_F^2
    \lesssim h^{-1} \tau^{-2} \eta^2 \norm[0]{ h_K^{-1}(q_1^*-\bar{q}_1^*) }_{\partial \mathcal{T}_h}^2.
  \end{equation}
  By a triangle inequality, \cref{eq:vintdgbound,eq:vldgbound} and
  using that $\nu \lesssim \tau h^2$,
  \begin{equation*}
    \begin{split}
      \nu \norm[0]{v_1}_{dg}^2
      &\le 2\nu\del[2]{ \norm[0]{v_{int}}_{dg}^2 + \norm[0]{v_l}_{dg}^2 }
      \lesssim \nu \tau^{-2}h^{-2} \tnorm{\boldsymbol{q}_1^*}_{q,1}^2
        + \nu h^{-2}\tau^{-2}\eta \tnorm{\boldsymbol{q}_1^*}_{q,1}^2
      \lesssim \tau^{-1} \eta \tnorm{\boldsymbol{q}_1^*}_{q,1}^2.
    \end{split}
  \end{equation*}    
  Next, we note that by \cref{lem:bdm-lifting}(ii)
  \begin{equation*}
    \norm[0]{v_l}_{\mathcal{T}_h}
    \lesssim h^{1/2} \tau^{-1}\eta \norm[0]{h_K^{-1}(q_1^*-\bar{q}_1^*)}_{\partial \mathcal{T}_h}
    \le \tau^{-1}\eta^{1/2} \tnorm{\boldsymbol{q}_1^*}_{q,1},
  \end{equation*}
  and so, combined with
  $\norm[0]{v_{int}}_{\mathcal{T}_h} \le \tau^{-1}
  \tnorm{\boldsymbol{q}_1^*}_{q,1}$,
  \begin{align*}
    \tau \norm[0]{v_1}_{\mathcal{T}_h}^2
    \le 2\tau \del[2]{\norm[0]{v_l}_{\mathcal{T}_h}^2 + \norm[0]{v_{int}}_{\mathcal{T}_h}^2}
    \lesssim \tau^{-1}\eta \tnorm{\boldsymbol{q}_1^*}_{q,1}^2
    + \tau^{-1} \tnorm{\boldsymbol{q}_1^*}_{q,1}^2
    \lesssim \tau^{-1} \eta \tnorm{\boldsymbol{q}_1^*}_{q,1}^2.
  \end{align*}
  Collecting the aforementioned results we find that
  \begin{equation}
    \label{eq:sumv1dgl2}
    \nu \norm[0]{v_1}_{dg}^2 + \tau \norm[0]{v_1}_{\mathcal{T}_h}^2
    \lesssim \tau^{-1} \eta \tnorm{\boldsymbol{q}_1^*}_{q,1}^2
    \le \eta \tnorm{\boldsymbol{q}_h}_{q*}^2.
  \end{equation}
  \underline{Step 6.} Define $\bar{v}_1 = \av{v_1}$. Then by
  \cref{eq:dgequiv} and Step 5 we have that
  $\tnorm{\boldsymbol{v}_1}_v^2 \lesssim \eta
  \tnorm{\boldsymbol{q}_h}_{q*}^2$.

  \smallskip In Cases 1 and 2 we have therefore shown that given any
  $\boldsymbol{q}_h \in \boldsymbol{Q}_h$ there exists
  $\boldsymbol{v}_h \in \boldsymbol{V}_h$ such that
  $\tnorm{\boldsymbol{v}_h}_{v} \lesssim \eta
  \tnorm{\boldsymbol{q}_h}_{q*}$ and
  $b_h(v_h, \boldsymbol{q}_h) = \tnorm{\boldsymbol{q}_h}_{q*}^2$.
\end{proof}

\subsection{Uniform well-posedness of the HDG discretization}
\label{ss:uniformwphdg}

In this section we prove uniform well-posedness of the HDG
discretization by proving that the bilinear form $a_h(\cdot, \cdot)$
is uniformly bounded and inf-sup stable with respect to the norm
$\tnorm{\cdot}_{\boldsymbol{X}_h}$. We start by proving uniform
boundedness.

\begin{lemma}[Uniform boundedness of $a_h$]
  \label{lem:boundedness}
  There exists a uniform constant $c_b > 0$ such that
  \begin{equation}
    \label{eq:boundedness}
    |a_h(\boldsymbol{x}_h, \boldsymbol{y}_h)|
    \leq c_b \tnorm{\boldsymbol{x}_h}_{\boldsymbol{X}_h} \tnorm{\boldsymbol{y}_h}_{\boldsymbol{X}_h}
    \quad \forall \boldsymbol{x}_h, \boldsymbol{y}_h \in \boldsymbol{X}_h.
  \end{equation}
\end{lemma}
\begin{proof}
  By \cite[Lemma 4.3]{rhebergen2017analysis} and using the definition
  of $\tnorm{\cdot}_v$, we find that there exists a uniform constant
  $c_d > 0$ such that
  \begin{equation}
    \label{eq:boundedness-aux1}
    |\tau(u_h, v_h)_{\mathcal{T}_h}
    + d_h(\boldsymbol{u}_h, \boldsymbol{v}_h)| 
    \leq c_d \tnorm{\boldsymbol{u}_h}_v \tnorm{\boldsymbol{v}_h}_v
    \quad \forall \boldsymbol{u}_h, \boldsymbol{v}_h \in \boldsymbol{V}_h.
  \end{equation}
  Furthermore, let $\boldsymbol{v}_h \in \boldsymbol{V}_h$,
  $\boldsymbol{q}_h \in \boldsymbol{Q}_h$, and assume an arbitrary
  splitting
  $\boldsymbol{q}_h = \boldsymbol{q}_{1,h} + \boldsymbol{q}_{2,h}$.
  Then
  \begin{equation}
    \label{eq:boundedness-aux2}
    b_h(v_h, \boldsymbol{q}_h) = b_h(v_h, \boldsymbol{q}_{1,h})
    + b_h(v_h, \boldsymbol{q}_{2,h}).
  \end{equation}
  By \cite[Lemma 4.8 and eq. (102)]{rhebergen2017analysis}, there
  exists a uniform constant $c_1 > 0$, such that
  \begin{equation}
    \label{eq:boundedness-aux3}
    b_h(\boldsymbol{v}_h, \boldsymbol{q}_{1,h}) 
    \leq c_1 \nu^{1/2} \tnorm{\boldsymbol{v}_h}_{v,1} \nu^{-1/2} \tnorm{\boldsymbol{q}_{1,h}}_{q,0}
    \leq c_1 \tnorm{\boldsymbol{v}_h}_v \nu^{-1/2} \tnorm{\boldsymbol{q}_{1,h}}_{q,0}.
  \end{equation}
  Next, following the same steps as used to prove
  \cite[eq. (A.3)]{henriquez2025parameter}, we find that there exists
  a uniform constant $c_2 > 0$ such that
  \begin{equation}
    \label{eq:boundedness-aux4}
    b_h(v_h, \boldsymbol{q}_{2,h}) 
    \leq c_2 \tau^{1/2} \norm[0]{v_h}_{\mathcal{T}_h}
    \tau^{-1/2} \tnorm{\boldsymbol{q}_{2,h}}_{q,1}
    \leq c_2 \tnorm{\boldsymbol{v}_h}_v
    \tau^{-1/2} \tnorm{\boldsymbol{q}_{2,h}}_{q,1}.
  \end{equation}
  Combining
  \cref{eq:boundedness-aux2,eq:boundedness-aux3,eq:boundedness-aux4}
  we obtain
  \begin{equation*}
    b_h(v_h, \boldsymbol{q}_h) 
    \leq c \tnorm{\boldsymbol{v}_h}_{v}
    ( \nu^{-1} \tnorm{\boldsymbol{q}_{1,h}}_{q,0}^2 
    + \tau^{-1} \tnorm{\boldsymbol{q}_{2,h}}_{q,1}^2)^{1/2}.
  \end{equation*}
  Taking the infimum over all the splittings of $\boldsymbol{q}_h$ and
  combining with \cref{eq:boundedness-aux1}, we find the desired
  bound.
\end{proof}

Before proving that $a_h(\cdot, \cdot)$, note that by \cite[Lemma
4.2]{rhebergen2017analysis}, there exists a uniform constant $c_c > 0$
such that
\begin{equation}
  \label{eq:coercivity}
  d_h(\boldsymbol{v}_h, \boldsymbol{v}_h) 
  + \tau \norm[0]{v_h}_{\mathcal{T}_h}^2
  \geq c_c \tnorm{\boldsymbol{v}_h}_v^2 
  \quad \forall \boldsymbol{v}_h \in \boldsymbol{V}_h.
\end{equation}

\begin{lemma}[Uniform inf-sup stability of $a_h$]
  \label{lem:stability}
  There exists a uniform constant $c_s > 0$ such that
  \begin{equation}
    \label{eq:stability}
    \inf_{\boldsymbol{x}_h \in \boldsymbol{X}_h}
    \sup_{\boldsymbol{y}_h \in \boldsymbol{X}_h}
    \frac{a_h(\boldsymbol{x}_h, \boldsymbol{y}_h)}
    {\tnorm{\boldsymbol{x}_h}_{\boldsymbol{X}_h} \tnorm{\boldsymbol{y}_h}_{\boldsymbol{X}_h} }
    \geq c_s.
  \end{equation}
\end{lemma}
\begin{proof}
  Given $\boldsymbol{p}_h \in \boldsymbol{Q}_h$, by \cref{eq:infsup3}
  there exists a $\Tilde{\boldsymbol{v}}_h \in \boldsymbol{V}_h$ such
  that
  \begin{equation}
    \label{eq:stability-aux2}
    b_h(\Tilde{v}_h, \boldsymbol{p}_h) 
    = \tnorm{\boldsymbol{p}_h}_q^2
    \quad
    \text{ and }
    \quad
    \tnorm{\tilde{\boldsymbol{v}}_h}_v
    \leq c_3^{-1} \tnorm{\boldsymbol{p}_h}_q. 
  \end{equation}
  Given non-null
  $\boldsymbol{x}_h := (\boldsymbol{u}_h, \boldsymbol{p}_h) \in
  \boldsymbol{X}_h$, define
  $\boldsymbol{y}_h := (\boldsymbol{v}_h, \boldsymbol{q}_h) \in
  \boldsymbol{X}_h$ such that
  \begin{equation*}
    \boldsymbol{v}_h = \boldsymbol{u}_h + \delta \tilde{\boldsymbol{v}}_h
    \quad
    \text{ and }
    \quad
    \boldsymbol{q}_h = - \boldsymbol{p}_h,
  \end{equation*}
  where $\delta > 0$ is a positive constant to be determined.

  Using \cref{eq:stability-aux2} in combination with
  \cref{eq:coercivity,eq:boundedness-aux1} and Young's inequality, we
  find:
  \begin{align*}
    a_h(\boldsymbol{x}_h, \boldsymbol{y}_h)
    & = d_h(\boldsymbol{u}_h, \boldsymbol{u}_h)
      + \tau \norm[0]{u_h}_{\mathcal{T}_h}^2
      + \delta \del[2]{\tau (u_h, \tilde{v}_h)_{\mathcal{T}_h}
      + d_h(\boldsymbol{u}_h, \tilde{\boldsymbol{v}}_h)}
      + \delta b_h(\tilde{v}_h, \boldsymbol{p}_h)
    \\
    &\geq (c_c - \frac{\delta}{2}c_d^2 c_3^{-2})\tnorm{\boldsymbol{u}_h}_v^2
      + \frac{\delta}{2} \tnorm{\boldsymbol{p}_h}_q^2.
  \end{align*}
  Choosing $\delta = c_c/(c_d^2 c_3^{-2})$, we obtain
  \begin{equation}
    \label{eq:stability-aux1a}
    a_h(\boldsymbol{x}_h, \boldsymbol{y}_h)
    \geq c_{s1} \tnorm{\boldsymbol{x}_h}_{\boldsymbol{X}_h}^2,
  \end{equation}
  where $c_{s1} := (c_c/2) \min(1, 1/(c_d^2 c_3^{-2}))$.

  Next, we note that
  \begin{equation*}
    \tnorm{\boldsymbol{v}_h}_v^2
    \leq 2 \tnorm{\boldsymbol{u}_h}_v^2
    + 2 \delta^2 \tnorm{\tilde{\boldsymbol{v}}_h}_v^2
    \leq 2 \tnorm{\boldsymbol{u}_h}_v^2
    + 2 \delta^2 c_3^{-2} \tnorm{\boldsymbol{p}_h}_q^2.
  \end{equation*}
  Noting furthermore that
  $\tnorm{\boldsymbol{q}_h}_q = \tnorm{\boldsymbol{p}_h}_q$ we find
  that there exists a uniform constant $c_{s2} > 0$ such that
  \begin{equation}
    \label{eq:stability-aux1b}
    \tnorm{\boldsymbol{y}_h}_{\boldsymbol{X}_h}
    \leq c_{s2} \tnorm{\boldsymbol{x}_h}_{\boldsymbol{X}_h}.    
  \end{equation}
  \Cref{eq:stability} now follows as a result of
  \cref{eq:stability-aux1a,eq:stability-aux1b}.
\end{proof}

\section{Preconditioning}
\label{s:precon}

\subsection{Local solvers and reduced problem}
\label{ss:local-solvers}

In this section we present the reduced problem in a variational
setting, which is obtained after eliminating $u_h$ and $p_h$ from
\cref{eq:hdgmethod}. To obtain this reduced problem we require local
solvers (see also \cite[Section
2.4]{rhebergen2018preconditioning}). To set notation, let
$V(K) := [\mathbb{P}_k(K)]^d$ and $Q(K) := \mathbb{P}_{k-1}(K)$ for
$K \in \mathcal{T}_h$.

\begin{definition}[Local solvers]
  \label{def:localsolver-a0}
  Given $(\bar{m}_h, \bar{t}_h) \in \bar{V}_h \times \bar{Q}_h$ and
  $s \in [L^2(\Omega)]^d$, we define the functions
  $u_h^L(\bar{m}_h, \bar{t}_h, s) \in V_h$ and
  $p_h^L(\bar{m}_h, \bar{t}_h, s) \in Q_h$ such that when restricted
  to cell $K$ it satisfies
  \begin{equation}
    \label{eq:local-solver}
    a_h^K((u_h^L, p_h^L), (v_h, q_h)) = f_h^K(v_h) \quad \forall (v_h, q_h) 
    \in V(K) \times Q(K),
  \end{equation}
  where
  \begin{align*}
    a_h^K((u_h, p_h), (v_h, q_h)) 
    :=& \tau (u_h, v_h)_K
        + \nu (\nabla u_h, \nabla v_h)_K
        + \nu \eta h_K^{-1} \langle u_h, v_h\rangle_{\partial K}
    \\
      & - \nu \langle \nabla u_h n, v_h\rangle_{\partial K}
        - \nu \langle \nabla v_h n, u_h\rangle_{\partial K}
    \\
      & - (p_h, \nabla \cdot v_h)_K
        - (q_h, \nabla \cdot u_h)_K,
    \\
    f_h^K(v_h)
    :=& (s, v_h)_K 
        - \nu \langle \nabla v_h n, \bar{m}_h \rangle_{\partial K}
        + \nu \eta h_K^{-1} \langle \bar{m}_h, v_h \rangle_{\partial K}
    \\
      & - \langle \bar{t}_h, v_h \cdot n\rangle_{\partial K}.
  \end{align*}
\end{definition}

The following lemma defines the reduced formulation of
\cref{eq:hdgmethod}. Its proof is identical to that of \cite[Lemma
4]{rhebergen2018preconditioning} and so it is omitted.

\begin{lemma}[Reduced problem]
  \label{lem:condensed-formulation}
  Given $f \in [L^2(\Omega)]^d$, define $u_h^f := u_h^L(0, 0, f)$
  and $p_h^f := p_h^L(0, 0, f)$. Furthermore, for all
  $\bar{y}_h := (\bar{v}_h, \bar{q}_h) \in \bar{V}_h \times \bar{Q}_h
  $, define $l_u(\bar{y}_h) := u_h^L(\bar{v}_h, \bar{q}_h, 0)$ and
  $l_p(\bar{y}_h) := p_h^L(\bar{v}_h, \bar{q}_h, 0)$. Let
  $\bar{x}_h := (\bar{u}_h, \bar{p}_h) \in \bar{V}_h \times \bar{Q}_h$
  be the solution to
  \begin{equation}
    \label{eq:condensed-formulation}
    \bar{a}_h(\bar{x}_h, \bar{y}_h) 
    = (f, l_u(\bar{y}_h))_{\mathcal{T}_h}
    \quad \bar{y}_h \in \bar{V}_h \times \bar{Q}_h,
  \end{equation}
  where
  \begin{align*}
    \bar{a}_h(\bar{x}_h, \bar{y}_h)
    := a_h((l_u(\bar{x}_h), \bar{u}_h, l_p(\bar{x}_h), \bar{p}_h),(l_u(\bar{y}_h), \bar{v}_h, l_p(\bar{y}_h), \bar{q}_h)).
  \end{align*}
  Then $(u_h, \bar{u}_h, p_h, \bar{p}_h)$, in which
  $u_h = u_h^f + l_u(\bar{x}_h)$ and $p_h = p_h^f + l_p(\bar{x}_h)$,
  solves \cref{eq:hdgmethod}.
\end{lemma}

\subsection{Preconditioning the reduced problem}
\label{ss:precon-condensed}

Let $\norm[0]{\cdot}_{\boldsymbol{X}_h}$ be the norm induced by the
inner product defined in \cref{eq:inner-p-X}. By
\cref{thm:cond-precon}, an operator
$\bar{P} : \bar{X}_h \to \bar{X}_h^*$ that defines an inner product on
$\bar{X}_h := \bar{V}_h \times \bar{Q}_h$ is a parameter-robust
preconditioner for the reduced problem \cref{eq:condensed-formulation}
provided there exist uniform constants $c_l, c_u > 0$ such that
\cref{eq:cond-precon-1} holds. In this section we determine an inner
product defining operator $\bar{P}$. In \cref{ss:barP-robust} we will
prove that this operator satisfies the remaining conditions of
\cref{thm:cond-precon}.

We will determine an operator $\bar{P}$ with block diagonal structure
\begin{equation}
  \label{eq:barPP11P22}
  \bar{P} =
  \begin{bmatrix}
    \bar{P}_{11} & 0 \\ 0 & \bar{P}_{22}
  \end{bmatrix},
\end{equation}
where $\bar{P}_{11} : \bar{V}_h \to \bar{V}_h^*$ and
$\bar{P}_{22} : \bar{Q}_h \to \bar{Q}_h^*$. We consider the
$\bar{P}_{11}$ and $\bar{P}_{22}$ blocks separately.

\subsubsection*{The $\bar{P}_{11}$ operator} We introduce the operator
$P^u : \boldsymbol{V}_h \to \boldsymbol{V}_h^*$ such that
\begin{equation*}
  \langle P^u \boldsymbol{u}_h, \boldsymbol{v}_h\rangle_{\boldsymbol{V}_h^*, \boldsymbol{V}_h}
  := (\boldsymbol{u}_h, \boldsymbol{v}_h)_v
  \quad \forall \boldsymbol{u}_h, \boldsymbol{v}_h \in \boldsymbol{V}_h,
\end{equation*}
where $(\cdot, \cdot)_v$ is defined in \cref{eq:inner-p-a}. This
operator has the block form
\begin{equation*}
  P^u  
  = \begin{bmatrix}
    P_{11}^u & (P_{21}^u)^T\\
    P_{21}^u & P_{22}^u
  \end{bmatrix},
\end{equation*}
where $P_{11}^u : V_h \to V_h^*$, $P_{21}^u : V_h \to \bar{V}_h^*$,
and $P_{22}^u : \bar{V}_h \to \bar{V}_h^*$. The Schur complement
operator $S_{P^u} : \bar{V}_h \to \bar{V}_h^*$ of $P^u$ on $\bar{V}_h$
is defined as
\begin{equation*}
  \langle S_{P^u} \bar{u}_h, \bar{v}_h\rangle_{\bar{V}_h^*, \bar{V}_h}
  := \langle (P_{22}^u - P_{21}^u P_{11}^{-1} (P_{21}^u)^T) \bar{u}_h, \bar{v}_h \rangle_{\bar{V}_h^*, \bar{V}_h} 
  \quad \forall \bar{u}_h, \bar{v}_h \in \bar{V}_h,
\end{equation*}
which defines an inner product on $\bar{V}_h$:
\begin{equation}
  \label{eq:innerp-barV}
  (\bar{u}_h, \bar{v}_h)_{\bar{V}_h}
  := \langle S_{P^u} \bar{u}_h, \bar{v}_h\rangle_{\bar{V}_h^*, \bar{V}_h}
  \quad \forall \bar{u}_h, \bar{v}_h \in \bar{V}_h.
\end{equation}
The norm induced by this inner product is denoted by
$\norm[0]{\cdot}_{\bar{V}_h}$. We set $\bar{P}_{11} = S_{P^u}$.

\subsubsection*{The $\bar{P}_{22}$ operator} To determine an operator
$\bar{P}_{22}$ it is possible to follow the same approach as used to
determine $\bar{P}_{11}$, but instead using the inner product
$(\cdot, \cdot)_q$ defined in \cref{eq:inner-p-b}. However,
$(\cdot, \cdot)_q$ is defined as an infimum over the sum of two inner
products on $\boldsymbol{Q}_h$, namely $\nu^{-1}(\cdot, \cdot)_{q,0}$
and $\tau^{-1}(\cdot, \cdot)_{q,1}$. Unfortunately, it is not clear
how to characterize the inverse of the Schur complement on $\bar{Q}_h$
of the operator associated with $(\cdot, \cdot)_q$. This has practical
consequences as the inverse characterization is important for
implementing the preconditioner. We therefore present here an
alternative operator. We first determine two inner products on
$\bar{Q}_h$ associated with the Schur complements of the operators
associated with $\nu^{-1}(\cdot, \cdot)_{q,0}$ and
$\tau^{-1}(\cdot, \cdot)_{q,1}$. We then take the infimum over the sum
of these two inner products. The inverse of the operator associated
with this inner product is characterized in
\cref{ss:characterization-preconds}.

Consider the operators $P^s: \boldsymbol{Q}_h \to \boldsymbol{Q}_h^*$
and $P^d: \boldsymbol{Q}_h \to \boldsymbol{Q}_h^*$ defined as
\begin{equation*}
  \langle P^s \boldsymbol{p}_h, \boldsymbol{q}_h \rangle_{\boldsymbol{Q}_h^*, \boldsymbol{Q}_h}
  := (\boldsymbol{p}_h, \boldsymbol{q}_h)_{q,0}
  \quad \text{ and }
  \quad
  \langle P^d \boldsymbol{p}_h, \boldsymbol{q}_h\rangle_{\boldsymbol{Q}_h^*, \boldsymbol{Q}_h}
  := (\boldsymbol{p}_h, \boldsymbol{q}_h)_{q,1},
\end{equation*}
for all $\boldsymbol{p}_h, \boldsymbol{q}_h \in
\boldsymbol{Q}_h$. These operators have the following block forms:
\begin{equation*}
  P^s
  = \begin{bmatrix}
    P_{11}^s & 0\\
    0 & P_{22}^s
  \end{bmatrix}
  \quad \text{ and } \quad
  P^d
  = \begin{bmatrix}
    P_{11}^d & (P_{21}^d)^T\\
    P_{21}^d & P_{22}^d
  \end{bmatrix},
\end{equation*}
where $P_{11}^s : Q_h \to Q_h^*$,
$P_{22}^s : \bar{Q}_h \to \bar{Q}_h^*$, $P_{11}^d: Q_h \to Q_h^*$,
$P_{21}^d: Q_h \to \bar{Q}_h^*$, and
$P_{22}^d : \bar{Q}_h \to \bar{Q}_h^*$. The Schur complements of $P^s$
and $P^d$ are the operators $S_{P^s}: \bar{Q}_h \to \bar{Q}_h^*$ and
$S_{P^d} : \bar{Q}_h \to \bar{Q}_h^*$ which satisfy
\begin{equation*}
  \begin{split}
    \langle S_{P^s} \bar{p}_h, \bar{q}_h\rangle_{\bar{Q}_h^*, \bar{Q}_h}
    &:= \langle P_{22}^s \bar{p}_h, \bar{q}_h\rangle_{\bar{Q}_h^*, \bar{Q}_h},
    \\
    \langle S_{P^d} \bar{p}_h, \bar{q}_h\rangle_{\bar{Q}_h^*, \bar{Q}_h}
    &:= \langle (P_{22}^d - P_{21}^d (P_{11}^d)^{-1} (P_{21})^T) \bar{p}_h, \bar{q}_h\rangle_{\bar{Q}_h^*, \bar{Q}_h},    
  \end{split}
\end{equation*}
for all $\bar{p}_h, \bar{q}_h \in \bar{Q}_h$. We then define the
following inner product on $\bar{Q}_h$:
\begin{equation}
  \label{eq:innerp-barQ}
  (\bar{p}_h, \bar{q}_h)_{\bar{Q}_h} 
  := \inf_{\substack{\bar{p}_h = \bar{p}_{1,h} + \bar{p}_{2,h}
      \\
      \bar{q}_h = \bar{q}_{1,h} + \bar{q}_{2,h}}}
  \del[1]{ \nu^{-1} \langle S_{P^s} \bar{p}_{h,1}, \bar{q}_{h,1}\rangle_{\bar{Q}_h^*, \bar{Q}_h}
  + \tau^{-1} \langle S_{P^d} \bar{p}_{h,2}, \bar{q}_{h,2}\rangle_{\bar{Q}_h^*, \bar{Q}_h} },
\end{equation}
for all $\bar{p}_h, \bar{q}_h \in \bar{Q}_h$. The norm induced by this
inner product is denoted by $\norm[0]{\cdot}_{\bar{Q}_h}$. We use this
inner product to define the operator
$\bar{P}^p : \bar{Q}_h \to \bar{Q}_h^*$:
\begin{equation*}
  \langle \bar{P}^p \bar{p}_h, \bar{q}_h\rangle_{\bar{Q}_h^*, \bar{Q}_h}
  := (\bar{p}_h, \bar{q}_h)_{\bar{Q}_h} 
  \quad\forall \bar{p}_h, \bar{q}_h \in \bar{Q}_h.
\end{equation*}
We set $\bar{P}_{22} = \bar{P}^p$.

\subsubsection*{The reduced preconditioner $\bar{P}$} Consider the
following inner product on $\bar{X}_h := \bar{V}_h \times \bar{Q}_h$:
\begin{equation}
  \label{eq:inner-barX}
  (\bar{x}_h, \bar{y}_h)_{\bar{X}_h}
  := (\bar{u}_h, \bar{v}_h)_{\bar{V}_h}
  + (\bar{p}_h, \bar{q}_h)_{\bar{Q}_h},
\end{equation}
for all $\bar{x}_h:=(\bar{u}_h,\bar{p}_h) \in \bar{X}_h$ and
$\bar{y}_h:=(\bar{v}_h,\bar{q}_h) \in \bar{X}_h$, with
$(\cdot, \cdot)_{\bar{V}_h}$ defined in \cref{eq:innerp-barV} and
$(\cdot, \cdot)_{\bar{Q}_h}$ defined in \cref{eq:innerp-barQ}. The
norm induced by this inner product is denoted by
$\norm[0]{\cdot}_{\bar{X}_h}$. The reduced preconditioner $\bar{P}$
for \cref{eq:condensed-formulation} is defined by
\begin{equation*}
  \langle \bar{P} \bar{x}_h, \bar{y}_h \rangle_{\bar{X}_h^*,\bar{X}_h}
  =
  (\bar{x}_h, \bar{y}_h)_{\bar{X}_h}
  \quad \forall \bar{x}_h, \bar{y}_h \in \bar{X}_h,
\end{equation*}
which has block structure
\begin{equation}
  \label{eq:barP}
  \bar{P}  
  := \begin{bmatrix}
    S_{P^u} & 0\\
    0 & \bar{P}^p
  \end{bmatrix}.
\end{equation}
In \cref{ss:barP-robust} we show that $\bar{P}$ is a parameter-robust
preconditioner. A characterization of $\bar{P}^{-1}$ is given in
\cref{ss:characterization-preconds}.

\subsection{$\bar{P}$ is a parameter-robust preconditioner}
\label{ss:barP-robust}

We use \cref{thm:cond-precon} to show that $\bar{P}$ is a
parameter-robust preconditioner. This requires showing that
\cref{eq:cond-precon-1a,eq:cond-precon-1b} hold. The following lemma
shows \cref{eq:cond-precon-1a}.

\begin{lemma}
  \label{lem:cond1}
  Let $\norm[0]{\cdot}_{\boldsymbol{X}_h}$ and
  $\norm[0]{\cdot}_{\bar{X}_h}$ be the norms induced by the inner
  products defined in \cref{eq:inner-p-X,eq:inner-barX},
  respectively. Then
  \begin{equation}
    \label{eq:cond1}
    \norm[0]{\boldsymbol{y}_h}_{\boldsymbol{X}_h}
    \geq \norm[0]{\bar{y}_h}_{\bar{X}_h}
    \quad \forall \boldsymbol{y}_h \in \boldsymbol{X}_h.
  \end{equation}
\end{lemma}
\begin{proof}
  Let
  $\boldsymbol{y}_h := (\boldsymbol{v}_h, \boldsymbol{q}_h) \in
  \boldsymbol{X}_h$. Then
  \begin{equation}
    \label{eq:cond1-aux1}
    \begin{split}      
      \norm[0]{\boldsymbol{v}_h}_{v}^2
      =& \langle P_{11}^u(v_h 
         + (P_{11}^u)^{-1} (P_{21}^u)^T \bar{v}_h),  v_h 
         + (P_{11}^u)^{-1} (P_{21}^u)^T \bar{v}_h\rangle_{V_h^*, V_h}
         + \langle S_{P^u} \bar{v}_h, \bar{v}_h \rangle_{\bar{V}_h^*, \bar{V}_h}
      \\
      \geq& \langle S_{P^u} \bar{v}_h, \bar{v}_h \rangle_{\bar{V}_h^*, \bar{V}_h} 
            = \norm[0]{\bar{v}_h}_{\bar{X}_h}^2,      
    \end{split}
  \end{equation}
  where we used that $P_{11}^u$ is a symmetric positive
  operator. Similarly,
  \begin{align*}
    \langle P^s \boldsymbol{q}_{h}, \boldsymbol{q}_{h}\rangle_{\boldsymbol{Q}_h^*, \boldsymbol{Q}_h}
    =& \langle P_{11}^sq_{h},  q_h \rangle_{Q_h^*, Q_h}
       + \langle S_{P^s} \bar{q}_h, \bar{q}_h \rangle_{\bar{Q}_h^*, \bar{Q}_h}
    \\
    \ge & \langle S_{P^s} \bar{q}_h, \bar{q}_h \rangle_{\bar{Q}_h^*, \bar{Q}_h},
    \\
    \langle P^d \boldsymbol{q}_{h}, \boldsymbol{q}_{h}\rangle_{\boldsymbol{Q}_h^*, \boldsymbol{Q}_h}
    =& \langle P_{11}^d(q_h 
         + (P_{11}^d)^{-1} (P_{21}^d)^T \bar{q}_h),  q_h 
         + (P_{11}^d)^{-1} (P_{21}^d)^T \bar{q}_h\rangle_{Q_h^*, Q_h}
       + \langle S_{P^d} \bar{q}_h, \bar{q}_h \rangle_{\bar{Q}_h^*, \bar{Q}_h}
    \\
    \ge & \langle S_{P^d} \bar{q}_h, \bar{q}_h \rangle_{\bar{Q}_h^*, \bar{Q}_h},
  \end{align*}
  so that
  \begin{align}
    \nonumber
    \norm[0]{\boldsymbol{q}_h}_q^2
    &= \inf_{\boldsymbol{q}_h = \boldsymbol{q}_{1,h} + \boldsymbol{q}_{2,h}}
      \big( \nu^{-1} \langle P^s \boldsymbol{q}_{1,h}, \boldsymbol{q}_{1,h}\rangle_{\boldsymbol{Q}_h^*, \boldsymbol{Q}_h} 
      + \tau^{-1} \langle P^d \boldsymbol{q}_{2,h}, \boldsymbol{q}_{2,h}\rangle_{\boldsymbol{Q}_h^*, \boldsymbol{Q}_h} 
      \big)
    \\
    \nonumber
    & \geq \inf_{\bar{q}_h = \bar{q}_{1,h} + \bar{q}_{2,h}}
      \big( \nu^{-1} \langle S_{P^s}\bar{q}_{1,h}, \bar{q}_{1,h} \rangle_{\bar{Q}_h^*, \bar{Q}_h}
      + \tau^{-1} \langle S_{P^d}\bar{q}_{2,h}, \bar{q}_{2,h} \rangle_{\bar{Q}_h^*, \bar{Q}_h}\big)
    \\
    \label{eq:cond1-aux2}
    & = \norm[0]{\bar{q}_h}_{\bar{Q}_h}^2.
  \end{align}
  The result follows by combining \cref{eq:cond1-aux1,eq:cond1-aux2}.
\end{proof}

The remainder of this section is devoted to proving
\cref{eq:cond-precon-1b}. To this end, we first present some
preliminary results after which \cref{eq:cond-precon-1b} is proven in
\cref{lem:cond2-th1}. It will be useful to define the following norm
on $\bar{V}_h$:
\begin{equation}
  \label{eq:facet-norm}
  \tnorm{\bar{v}_h}_{v,h}^2 
  := \norm[0]{h_K^{-1/2}(\bar{v}_h - m_K(\bar{v}_h))}_{\partial \mathcal{T}_h}^2
  \qquad \forall \bar{v}_h \in \bar{V}_h,  
\end{equation}
where
$m_K(\bar{v}_h) := |\partial K|^{-1}\int_{\partial K} \bar{v}_h \dif
s$ for $K \in \mathcal{T}_h$.

\begin{lemma}
  \label{lem:bound-qnorm}
  There exists a uniform constant $c > 0$ such that
  \begin{equation}
    \label{eq:bound-qnorm-1}
    \tnorm{(l_p(\bar{v}_h, \bar{q}_h), \bar{q}_h)}_q^2
    \leq c \eta^2 \big( \norm[0]{\bar{q}_h}_{\bar{Q}_h}^2
    + \tnorm{(l_u(\bar{v}_h, \bar{q}_h), \bar{v}_h)}_v^2 \big)
    \quad \forall (\bar{v}_h, \bar{q}_h) \in \bar{V}_h \times \bar{Q}_h.
  \end{equation}
\end{lemma}
\begin{proof}
  For sake of notation, we will write
  $l_p := l_p(\bar{v}_h, \bar{q}_h) $ and
  $l_u := l_u(\bar{v}_h, \bar{q}_h) $. By \cref{eq:infsup3}, given
  $(l_p, \bar{q}_h) \in \boldsymbol{Q}_h$, there exists
  $\tilde{\boldsymbol{v}}_h \in \boldsymbol{V}_h$ such that
  \begin{equation}
    \label{eq:bound-qnorm-aux1}
    b_h(\tilde{v}_h, (l_p, \bar{q}_h))
    = \tnorm{(l_p, \bar{q}_h)}_q^2
    \quad \text{ and } \quad
    \tnorm{\tilde{\boldsymbol{v}}_h}_v 
    \leq c_3^{-1} \tnorm{(l_p, \bar{q}_h)}_q.
  \end{equation}
  \textbf{Step 1.} Choose
  $v_h = \tilde{v}_h - m_K(\tilde{\bar{v}}_h)$, $q_h = 0$, and $s = 0$
  in \cref{eq:local-solver}, then
  \begin{align*}
    &\tau (l_u, \tilde{v}_h - m_K(\tilde{\bar{v}}_h))_K
      + \nu (\nabla l_u, \nabla \tilde{v}_h)_K
      + \nu \eta h_K^{-1} \langle l_u, \tilde{v}_h - m_K(\tilde{\bar{v}}_h) \rangle_{\partial K}
    \\
    & - \nu \langle \nabla l_u n, \tilde{v}_h - m_K(\tilde{\bar{v}}_h) \rangle_{\partial K}
      - \nu \langle \nabla \tilde{v}_h n, l_u \rangle_{\partial K}
      - (l_p, \nabla \cdot \tilde{v}_h)_{K}
    \\
    & = 
      - \nu \langle \nabla \tilde{v}_h n, \bar{v}_h \rangle_{\partial K}
      + \nu \eta h_K^{-1} \langle \bar{v}_h, \tilde{v}_h - m_K(\tilde{\bar{v}}_h) \rangle_{\partial K}
      - \langle \bar{q}_h, (\tilde{v}_h - m_K(\tilde{\bar{v}}_h))\cdot n \rangle_{\partial K}.
  \end{align*}
  Rearranging, summing over all elements, and using the equality in
  \cref{eq:bound-qnorm-aux1}, we obtain
  \begin{equation}
    \label{eq:bound-qnorm-aux2}
    \begin{split}
      \tnorm{(l_p, \bar{q}_h)}_q^2 
      =& \langle \bar{q}_h, m_K(\tilde{\bar{v}}_h) \cdot n \rangle_{\partial \mathcal{T}_h}
         - \tau ( l_u, \tilde{v}_h - m_K(\tilde{\bar{v}}_h) )_{\mathcal{T}_h}
         - \nu (\nabla l_u, \nabla \tilde{v}_h)_{\mathcal{T}_h}
      \\
       & - \nu \eta \langle h_K^{-1} (l_u - \bar{v}_h), \tilde{v}_h 
         - \tilde{\bar{v}}_h \rangle_{\partial \mathcal{T}_h}
         - \nu \eta \langle h_K^{-1} (l_u - \bar{v}_h), \tilde{\bar{v}}_h
         - m_K(\tilde{\bar{v}}_h) \rangle_{\partial \mathcal{T}_h}
      \\
       & + \nu \langle \nabla l_u n, \tilde{v}_h - \tilde{\bar{v}}_h \rangle_{\partial \mathcal{T}_h}
         + \nu \langle \nabla l_u n, \tilde{\bar{v}}_h - m_K(\tilde{\bar{v}}_h) \rangle_{\partial \mathcal{T}_h}
      \\
       & + \nu \langle \nabla \tilde{v}_h n, l_u - \bar{v}_h \rangle_{\partial \mathcal{T}_h}.      
    \end{split}
  \end{equation}
  Consider the first term on the right hand side of
  \cref{eq:bound-qnorm-aux2}. Let $\tilde{l}_p(\cdot)$ be defined in
  \cref{lem:redauxproblem}. Taking an arbitrary splitting
  $\bar{q}_h = \bar{q}_{1,h} + \bar{q}_{2,h} \in \bar{Q}_h$, we find
  \begin{equation*}
    \begin{split}
      \langle \bar{q}_h, m_K(\tilde{\bar{v}}_h) \cdot n \rangle_{\partial \mathcal{T}_h}
      =& \langle \bar{q}_{1,h}, m_K(\tilde{\bar{v}}_h) \cdot n\rangle_{\partial \mathcal{T}_h}
         + \langle \bar{q}_{2,h}, m_K(\tilde{\bar{v}}_h) \cdot n\rangle_{\partial \mathcal{T}_h}
      \\
      =& \langle \bar{q}_{1,h}, (m_K(\tilde{\bar{v}}_h) - \tilde{\bar{v}}_h) \cdot n\rangle_{\partial \mathcal{T}_h}
         + \langle \bar{q}_{2,h} - \tilde{l}_p(\bar{q}_{2,h}), m_K(\tilde{\bar{v}}_h) \cdot n\rangle_{\partial \mathcal{T}_h}
      \\
       & + (\nabla \tilde{l}_p(\bar{q}_{2,h}), m_K(\tilde{\bar{v}}_h))_{\mathcal{T}_h},      
    \end{split}
  \end{equation*}
  where the second equality follows because $\tilde{\bar{v}}_h$ is
  single-valued on interior faces, $\tilde{\bar{v}}_h = 0$ on
  $\partial \Omega$,
  $ - (\tilde{l}_p (\bar{q}_{2,h}), \nabla \cdot
  m_K(\tilde{\bar{v}}_h))_{\mathcal{T}_h} = 0$, and integration by
  parts. Using the Cauchy--Schwarz and Young's inequalities and a
  discrete trace inequality with uniform constant $c_{tr} > 0$, we
  find
  \begin{equation}
    \label{eq:bound-qnorm-aux3}
    \begin{split}
      \langle \bar{q}_h, m_K(\tilde{\bar{v}}_h) \cdot n \rangle_{\partial \mathcal{T}_h}
      \leq &
             \frac{\varepsilon_1 \nu^{-1} }{2}  \norm[0]{h_K^{1/2} \bar{q}_{1,h}}_{\partial \mathcal{T}_h}^2
             + \frac{\nu}{2 \varepsilon_1} \norm[0]{h_K^{-1/2} (m_K(\tilde{\bar{v}}_h) - \tilde{\bar{v}}_h)}_{\partial \mathcal{T}_h}^2
      \\
           & + \frac{\varepsilon_2 c_{tr}^2 \tau^{-1}}{2} \norm[0]{h_K^{-1/2} (\bar{q}_{2,h} - \tilde{l}_p(\bar{q}_{2,h}))}_{\partial \mathcal{T}_h}^2
             + \frac{\tau}{2 \varepsilon_2} \norm[0]{m_K(\tilde{\bar{v}}_h)}_{\mathcal{T}_h}^2
      \\
           & + \frac{\varepsilon_3 \tau^{-1}}{2} \norm[0]{\nabla \tilde{l}_p (\bar{q}_{2,h})}_{\mathcal{T}_h}^2
             + \frac{\tau}{2 \varepsilon_3} \norm[0]{m_K(\tilde{\bar{v}}_h)}_{\mathcal{T}_h}^2,
    \end{split}
  \end{equation}
  where $\varepsilon_1,\varepsilon_2$ and $\varepsilon_3$ are positive
  constants that will be chosen later. Combining
  \cref{eq:bound-qnorm-aux2,eq:bound-qnorm-aux3}, again using the
  Cauchy--Schwarz and Young's inequalities and a discrete trace
  inequality, we obtain
  \begin{align*}
    \tnorm{(l_p, \bar{q}_h)}_q^2
    \leq& \frac{\varepsilon_1 \nu^{-1}}{2} \norm[0]{h_K^{1/2} \bar{q}_{1,h}}_{\partial \mathcal{T}_h}^2
          + \frac{\nu}{2 \varepsilon_1} \norm[0]{h_K^{-1/2} (m_K(\tilde{\bar{v}}_h) - \tilde{\bar{v}}_h)}_{\partial \mathcal{T}_h}^2
    \\
        & + \frac{\varepsilon_2 c_{tr}^2 \tau^{-1}}{2} \norm[0]{h_K^{-1/2} (\bar{q}_{2,h} - \tilde{l}_p(\bar{q}_{2,h}))}_{\partial \mathcal{T}_h}^2
          + \frac{\tau}{2 \varepsilon_2} \norm[0]{m_K(\tilde{\bar{v}}_h)}_{\mathcal{T}_h}^2
    \\
        & + \frac{\varepsilon_3 \tau^{-1}}{2} \norm[0]{\nabla \tilde{l}_p(\bar{q}_{2,h})}_{\mathcal{T}_h}^2
          + \frac{\tau}{2 \varepsilon_3} \norm[0]{m_K(\tilde{\bar{v}}_h)}_{\mathcal{T}_h}^2
          + \frac{\varepsilon_4 \tau}{2} \norm[0]{l_u}_{\mathcal{T}_h}^2
          + \frac{\tau}{2 \varepsilon_4} \norm[0]{\tilde{v}_h}_{\mathcal{T}_h}^2
    \\
        & + \frac{\varepsilon_5 \tau}{2} \norm[0]{l_u}_{\mathcal{T}_h}^2
          + \frac{\tau}{2 \varepsilon_5} \norm[0]{m_K(\tilde{\bar{v}}_h)}_{\mathcal{T}_h}^2
          + \frac{\varepsilon_6 \nu}{2} \norm[0]{\nabla l_u}_{\mathcal{T}_h}^2
          + \frac{\nu}{2 \varepsilon_6} \norm[0]{\nabla \tilde{v}_h}_{\mathcal{T}_h}^2
    \\
        & + \frac{\varepsilon_7 \nu \eta}{2} \norm[0]{h_K^{-1/2} (l_u - \bar{v}_h)}_{\partial \mathcal{T}_h}^2
          + \frac{\nu \eta}{2 \varepsilon_7} \norm[0]{h_K^{-1/2} (\tilde{v}_h - \tilde{\bar{v}}_h)}_{\partial \mathcal{T}_h}^2
    \\
        & + \frac{\varepsilon_8 \nu \eta^2}{2} \norm[0]{h_K^{-1/2} (l_u - \bar{v}_h)}_{\partial \mathcal{T}_h}^2
          + \frac{\nu}{2 \varepsilon_8} \norm[0]{h_K^{-1/2}(\tilde{\bar{v}}_h - m_K(\tilde{\bar{v}}_h))}_{\partial \mathcal{T}_h}^2
    \\
        & + \frac{\varepsilon_9 c_{tr}^2 \nu}{2} \norm[0]{\nabla l_u}_{\mathcal{T}_h}^2
          + \frac{\nu}{2 \varepsilon_9} \norm[0]{h_K^{-1/2} (\tilde{v}_h - \tilde{\bar{v}}_h)}_{\partial \mathcal{T}_h}^2
    \\
        & + \frac{\varepsilon_{10} c_{tr}^2 \nu}{2} \norm[0]{\nabla l_u}_{\mathcal{T}_h}^2
          + \frac{\nu}{2 \varepsilon_{10}} \norm[0]{h_K^{-1/2} (\tilde{\bar{v}}_h - m_K(\tilde{\bar{v}}_h))}_{\partial \mathcal{T}_h}^2
    \\
        & + \frac{\varepsilon_{11} c_{tr}^2 \nu}{2} \norm[0]{h_K^{-1/2} (l_u - \bar{v}_h)}_{\partial \mathcal{T}_h}^2 
          + \frac{\nu}{2 \varepsilon_{11}} \norm[0]{\nabla \tilde{v}_h}_{\mathcal{T}_h}^2,
  \end{align*}
  where $\varepsilon_4, \ldots, \varepsilon_{11}$ are positive
  constants that will be chosen later. Applying
  \cref{eq:bound-qnorm-aux1,eq:facet-norm-bound} together with the
  definitions of the norms in \cref{eq:inner-p,eq:facet-norm}, we find
  \begin{align*}
    \tnorm{(l_p, \bar{q}_h)}_q^2
    \leq& \frac{\varepsilon_1 \nu^{-1}}{2} \norm[0]{h_K^{1/2} \bar{q}_{1,h}}_{\partial \mathcal{T}_h}^2
          + \del[2]{\frac{\varepsilon_3}{2} 
          + \frac{\varepsilon_2 c_{tr}^2 }{2}} \tau^{-1} \tnorm{(\tilde{l}_p(\bar{q}_{2,h}),\bar{q}_{2,h})}_{q,1}^2
    \\
        & + \del[2]{\frac{\bar{c}^2}{2 \varepsilon_1}
          + \frac{1}{2 \varepsilon_4}
          + \frac{1}{2 \varepsilon_6}
          + \frac{1}{2 \varepsilon_7}
          + \frac{\bar{c}^2}{2 \varepsilon_8}
          + \frac{1}{2 \varepsilon_9}
          + \frac{\bar{c}^2}{2 \varepsilon_{10}}
          + \frac{1}{2 \varepsilon_{11}} } c_3^{-2} \tnorm{(l_p, \bar{q}_h)}_q^2
    \\
        & + \del[2]{\frac{\varepsilon_4}{2}
          + \frac{\varepsilon_5}{2}
          + \frac{\varepsilon_6}{2}
          + \frac{\varepsilon_7}{2}
          + \frac{\varepsilon_8 \eta}{2}
          + \frac{\varepsilon_9 c_{tr}^2}{2}
          + \frac{\varepsilon_{10} c_{tr}^2}{2}
          + \frac{\varepsilon_{11} c_{tr}^2}{2}} \tnorm{(l_u,\bar{v}_h)}_v^2
    \\
        & + \del[2]{\frac{1}{2 \varepsilon_2} 
          + \frac{1}{2 \varepsilon_3} 
          + \frac{1}{2 \varepsilon_5} } \tau \norm[0]{m_K(\tilde{\bar{v}}_h)}_{\mathcal{T}_h}^2.
  \end{align*}
  Then, choosing
  $\varepsilon_1 = \varepsilon_8 = \varepsilon_{10} = 5 \bar{c}^2
  c_3^{-2}$, $\varepsilon_2 = \varepsilon_3 = \varepsilon_5$, and
  $ \varepsilon_4 = \varepsilon_6 = \varepsilon_7 = \varepsilon_9 =
  \varepsilon_{11} = 5 c_3^{-2}$, and using that $\eta > 1$, we find
  that there exists a positive constant $c_1'$ such that
  \begin{multline}
    \label{eq:bound-qnorm-aux5}
    \tnorm{(l_p, \bar{q}_h)}_q^2
    \leq c_1'\eta \big(\nu^{-1}\eta^{-1} \norm[0]{h_K^{1/2} \bar{q}_{1,h}}_{\partial \mathcal{T}_h}^2
    + \varepsilon_2 \tau^{-1} \tnorm{(\tilde{l}_p(\bar{q}_{2,h}),\bar{q}_{2,h})}_{q,1}^2
    \\
    + (\varepsilon_2 + 1) \tnorm{(l_u,\bar{v}_h)}_v^2
    + \varepsilon_2^{-1} \tau \norm[0]{m_K(\tilde{\bar{v}}_h)}_{\mathcal{T}_h}^2\big).
  \end{multline}
  \textbf{Step 2.} Now choose $v_h = \tilde{v}_h$, $q_h = 0$, and
  $s = 0$ in \cref{eq:local-solver}. We obtain
  \begin{align*}
    & \tau (l_u, \tilde{v}_h)_K
      + \nu (\nabla l_u, \nabla \tilde{v}_h)_K
      + \nu \eta h_K^{-1} \langle l_u, \tilde{v}_h\rangle_{\partial K}
    \\
    & - \nu \langle \nabla l_u n, \tilde{v}_h\rangle_{\partial K}
      - \nu \langle \nabla \tilde{v}_h n, l_u \rangle_{\partial K}
      - (l_p, \nabla \cdot \tilde{v}_h)_K
    \\
    =& - \nu \langle \nabla \tilde{v}_h n, \bar{v}_h \rangle_{\partial K}
       + \nu \eta h_K^{-1} \langle \bar{v}_h, \tilde{v}_h \rangle_{\partial K}
       - \langle \bar{q}_h, \tilde{v}_h \cdot n\rangle_{\partial K}.
  \end{align*}
  Summing over all elements and performing some algebraic
  rearrangements, we obtain
  \begin{align*}
    & \tnorm{(l_p, \bar{q}_h)}_q^2
    \\
    =& - \tau (l_u, \tilde{v}_h)_{\mathcal{T}_h} - \nu (\nabla l_u, \nabla \tilde{v}_h)_{\mathcal{T}_h}
       + \nu \langle \nabla l_u n, \tilde{v}_h  - \tilde{\bar{v}}_h \rangle_{\partial \mathcal{T}_h}
       + \nu \langle \nabla l_u n, \tilde{\bar{v}}_h - m_K(\tilde{\bar{v}}_h) \rangle_{\partial \mathcal{T}_h}
    \\
    & + \nu \langle \nabla l_u n, m_K(\tilde{\bar{v}}_h) \rangle_{\partial \mathcal{T}_h}
      - \nu \eta \langle h_K^{-1} (l_u - \bar{v}_h), \tilde{v}_h - \tilde{\bar{v}}_h\rangle_{\partial \mathcal{T}_h}
    \\
    & - \nu \eta \langle h_K^{-1} (l_u - \bar{v}_h), \tilde{\bar{v}}_h-m_K(\tilde{\bar{v}}_h)\rangle_{\partial \mathcal{T}_h}
      - \nu \eta \langle h_K^{-1} (l_u - \bar{v}_h), m_K(\tilde{\bar{v}}_h)\rangle_{\partial \mathcal{T}_h}
    \\
    & + \nu \langle \nabla \tilde{v}_h n, l_u - \bar{v}_h \rangle_{\partial \mathcal{T}_h}.
  \end{align*}
  Using the Cauchy--Schwarz and Young's inequalities, a discrete trace
  inequality, \cref{eq:bound-qnorm-aux1,eq:facet-norm-bound} and the
  definitions of the norms in \cref{eq:inner-p}, we find
  \begin{align*}
    & \tnorm{(l_p, \bar{q}_h)}_q^2
    \\
    \le&
         \del[2]{\tfrac{\delta_1}{2} + \tfrac{\delta_2}{2} + \tfrac{\delta_3 c_{tr}^2}{2}
         + \tfrac{\delta_4c_{tr}^2\bar{c}^2}{2} + \tfrac{\delta_5 c_{tr}^2}{2}
         + \tfrac{\delta_6}{2} + \tfrac{\delta_7 \bar{c}^2 \eta}{2}
         + \tfrac{\delta_8 c_{tr}^2 \eta}{2} + \tfrac{\delta_9 c_{tr}^2}{2}
         }\tnorm{(l_u, \bar{v}_h)}_v^2
    \\
    & + \del[2]{\tfrac{1}{2\delta_1} + \tfrac{1}{2\delta_2} + \tfrac{1}{2\delta_3}
      + \tfrac{1}{2\delta_4} + \tfrac{1}{2\delta_6}
      + \tfrac{1}{2\delta_7} + \tfrac{1}{2\delta_9}
      }c_3^{-2} \tnorm{(l_p,\bar{q}_h)}_q^2
    \\
    & + \del[2]{\tfrac{1}{2\delta_5} + \tfrac{1}{2\delta_8}
      }\nu \norm[0]{h_K^{-1}m_K(\tilde{\bar{v}}_h)}_{\mathcal{T}_h}^2.
  \end{align*}
  Choosing
  $\delta_1 = \delta_2 = \delta_3 = \delta_4 = \delta_6 = \delta_7 =
  \delta_9 = 5 c_3^{-2}$ and $\delta_5 = \delta_8$, we find that there
  exists a uniform constant $c'_2 > 0$ such that
  \begin{equation}
    \label{eq:bound-qnorm-aux6}
    \tnorm{(l_p, \bar{q}_h)}_q^2
    \leq c'_2 \eta \big((\delta_5 + 1) \tnorm{(l_u, \bar{v}_h)}_v^2
    + \delta_5^{-1} \nu \norm[0]{h_K^{-1} m_K(\tilde{\bar{v}}_h)}_{\mathcal{T}_h}^2 \big).
  \end{equation}
  \textbf{Step 3.} Combining
  \cref{eq:bound-qnorm-aux5,eq:bound-qnorm-aux6}, and choosing
  $\delta_5 = \varepsilon_2$, we find that there exists a uniform
  constant $c_3' > 0$ such that
  \begin{multline}
    \label{eq:bound-qnorm-aux7}
    \tnorm{(l_p, \bar{q}_h)}_q^2
    \leq c_3' \eta \bigg[ \nu^{-1} \eta^{-1} \norm[0]{h_K^{1/2} \bar{q}_{1,h}}_{\partial \mathcal{T}_h}^2
    + \varepsilon_2 \tau^{-1} \tnorm{(\tilde{l}_p(\bar{q}_{2,h}), \bar{q}_{2,h})}_{q,1}^2
    \\
    + (\varepsilon_2 + 1) \tnorm{(l_u, \bar{v}_h)}_v^2
    + \varepsilon_2^{-1} \sum_{K \in \mathcal{T}_h} \min\{ \tau, h_K^{-2} \nu \} \norm[0]{m_K(\tilde{\bar{v}}_h)}_K^2\bigg].
  \end{multline}
  For the last term on the right hand side of
  \cref{eq:bound-qnorm-aux7} we have that there exists a uniform
  constant $c_m > 0$ such that (see \cite[Lemma
  A.2]{henriquez2025parameter}):
  \begin{equation*}    
    \min\{ \tau, h_K^{-2} \nu \} \norm[0]{m_K(\tilde{\bar
        v}_h)}_K^2
    \leq c_m \big( \tau \norm[0]{\tilde{v}_h}_K^2 
    + \nu \eta h_K^{-1} \norm[0]{\tilde{v}_h - \tilde{\bar{v}}_h}_{\partial K}^2 \big).
  \end{equation*}
  Combining the above with the definition of $\tnorm{\cdot}_v$,
  \cref{eq:bound-qnorm-aux1,eq:bound-qnorm-aux7}, and choosing
  $\varepsilon_2 = 2 c_3' c_m c_3^{-2} \eta $, we find that there
  exists a uniform constant $c' > 0$ such that
  \begin{equation}
    \label{eq:bound-qnorm-aux8}
    \tnorm{(l_p, \bar{q}_h)}_q^2
    \leq c' \eta^2 \big( \nu^{-1} \eta^{-1} \norm[0]{h_K^{1/2} \bar{q}_{1,h}}_{\partial \mathcal{T}_h}
    + \tau^{-1}\tnorm{(\tilde{l}_p(\bar{q}_{2,h}), \bar{q}_{2,h})}_{q,1}^2
    + \tnorm{(l_u, \bar{v}_h)}_v^2\big).
  \end{equation}
  Next we note that by combining
  \cref{eq:bili-form-tildea-equivalence} and
  \Cref{lem:schur-comp-equiv} we obtain:
  \begin{equation}
    \label{eq:bound-qnorm-aux9}
    \begin{split}
      \tau^{-1}\tnorm{\tilde{l}_p(\bar{q}_{2,h}, \bar{q}_{2,h})}_{q,1}^2
      &\leq \tilde{c}_1^{-1} \tilde{a}_h((\tilde{l}_p(\bar{q}_{2,h}), \bar{q}_{2,h}), (\tilde{l}_p(\bar{q}_{2,h}), \bar{q}_{2,h})))
      \\
      &\leq \tilde{c}_1^{-1} \tilde{c}_2 \tau^{-1} \langle S_{P^d} \bar{q}_{2,h}, \bar{q}_{2,h} \rangle_{\bar{Q}_h^*, \bar{Q}_h}.      
    \end{split}
  \end{equation}
  Therefore, by combining
  \cref{eq:bound-qnorm-aux8,eq:bound-qnorm-aux9} and taking the
  infimum over all splittings
  $\bar{q}_h = \bar{q}_{1,h} + \bar{q}_{2,h}$, we conclude that
  \cref{eq:bound-qnorm-1} holds.
\end{proof}

\begin{lemma}
  \label{lem:estimate-normv}
  There exists a positive uniform constant $c$ such that for all
  $(\bar{v}_h, \bar{q}_h) \in \bar{V}_h \times \bar{Q}_h$
  \begin{equation}
    \label{eq:estimate-normv}
    \tnorm{(l_u(\bar{v}_h, \bar{q}_h), \bar{v}_h)}_v^2
    \leq c \eta \del[1]{ \nu \tnorm{\bar{v}_h}_{v,h}^2
    + \norm[0]{\bar{v}_h}_{\bar{V}_h}^2 
    + \norm[0]{\bar{q}_h}_{\bar{Q}_h}^2 }.
  \end{equation}
\end{lemma}
\begin{proof}
  To simplify the notation in what follows, we write
  $l_p := l_p(\bar{v}_h, \bar{q}_h) $ and
  $l_u := l_u(\bar{v}_h, \bar{q}_h) $.

  \textbf{Step 1.} Choosing $v_h = l_u - m_K(\bar{v}_h)$,
  $q_h = - l_p$, and $s = 0$ in \cref{eq:local-solver}, and after
  reordering, we obtain
  \begin{equation}
    \label{eq:estimate-normv-aux1}
    \begin{split}
      \tau \norm[0]{l_u}_K^2
      &+ \nu \norm[0]{\nabla l_u}_K^2      
      + 2 \nu \langle \nabla l_u n, \bar{v}_h - l_u \rangle_{\partial K}
      + \nu \eta h_K^{-1} \norm[0]{l_u - \bar{v}_h}_{\partial K}^2      
    \\
    =&
      \tau (l_u, m_K(\bar{v}_h))_K
      + \nu \langle \nabla l_u n, \bar{v}_h - m_K(\bar{v}_h) \rangle_{\partial K}
    \\
    & + \nu \eta h_K^{-1} \langle \bar{v}_h - m_K(\bar{v}_h), \bar{v}_h - l_u\rangle_{\partial K}
      - \langle \bar{q}_h, (l_u - m_K(\bar{v}_h)) \cdot n\rangle_{\partial K}.      
    \end{split}
  \end{equation}
  Consider an arbitrary splitting
  $\bar{q}_h = \bar{q}_{1,h} + \bar{q}_{2,h}$. Let
  $\tilde{l}_p(\cdot)$ be as defined in \cref{lem:redauxproblem} and
  note that
  $0 = (\tilde{l}_p(\bar{q}_{2,h}), \nabla \cdot m_K(\bar{v}_h))_K = -
  (\nabla \tilde{l}_p(\bar{q}_{2,h}), m_K(\bar{v}_h))_K + \langle
  \tilde{l}_p(\bar{q}_{2,h}) , m_K(\bar{v}_h) \cdot n
  \rangle_{\partial K}$. For the last term on the right hand side of
  \cref{eq:estimate-normv-aux1} we then find:
  \begin{equation}
    \label{eq:estimate-normv-aux2}
    \begin{split}
      \langle \bar{q}_h, (l_u - m_K(\bar{v}_h)) \cdot n \rangle_{\partial K}
      =& \langle \bar{q}_{1,h}, (l_u - m_K(\bar{v}_h)) \cdot n \rangle_{\partial K}
         + \langle \bar{q}_{2,h}, (l_u - m_K(\bar{v}_h)) \cdot n \rangle_{\partial K}
      \\
      =& \langle \bar{q}_{1,h}, (l_u - \bar{v}_h) \cdot n \rangle_{\partial K}
         + \langle \bar{q}_{1,h}, (\bar{v}_h - m_K(\bar{v}_h)) \cdot n \rangle_{\partial K}
      \\
       & + \langle \bar{q}_{2,h}, l_u \cdot n \rangle_{\partial K}
         - \langle \bar{q}_{2,h} - \tilde{l}_p(\bar{q}_{2,h}), m_K(\bar{v}_h) \cdot n \rangle_{\partial K}
      \\
       & - (\nabla \tilde{l}_p(\bar{q}_{2,h}), m_K(\bar{v}_h))_K.     
    \end{split}
  \end{equation}
  Choose $v_h = 0$, $q_h = \tilde{l}_p(\bar{q}_{2,h})$, and $s = 0$ in
  \cref{eq:local-solver}. We find that
  $0 = - (\tilde{l}_p(\bar{q}_{2,h}), \nabla \cdot l_u)_K = (\nabla
  \tilde{l}_p(\bar{q}_{2,h}), l_u)_K - \langle
  \tilde{l}_p(\bar{q}_{2,h}) , l_u \cdot n \rangle_{\partial K}$.
  Combining with \cref{eq:estimate-normv-aux2} we obtain:
  \begin{equation}
    \label{eq:estimate-normv-aux2'}
    \begin{split}
      \langle \bar{q}_h, (l_u - m_K(\bar{v}_h)) \cdot n \rangle_{\partial K}
      =& \langle \bar{q}_{1,h}, (l_u - \bar{v}_h) \cdot n \rangle_{\partial K}
         + \langle \bar{q}_{1,h}, (\bar{v}_h - m_K(\bar{v}_h)) \cdot n \rangle_{\partial K}
      \\
       & + \langle \bar{q}_{2,h} - \tilde{l}_p(\bar{q}_{2,h}), l_u \cdot n \rangle_{\partial K}
         + (\nabla \tilde{l}_p(\bar{q}_{2,h}), l_u)_K
      \\
       & - \langle \bar{q}_{2,h} - \tilde{l}_p(\bar{q}_{2,h}), m_K(\bar{v}_h) \cdot n \rangle_{\partial K}
      \\
       & - (\nabla \tilde{l}_p(\bar{q}_{2,h}), m_K(\bar{v}_h))_K.     
    \end{split}
  \end{equation}
  Combining \cref{eq:estimate-normv-aux1,eq:estimate-normv-aux2'},
  summing over all elements, and using \cref{eq:coercivity}, we obtain
  \begin{align*}
    c_c \tnorm{(l_u, \bar{v}_h)}_v^2
    \le&
       \tau (l_u, m_K(\bar{v}_h))_{\mathcal{T}_h}
       + \nu \langle \nabla l_u n, \bar{v}_h - m_K(\bar{v}_h) \rangle_{\partial \mathcal{T}_h}
    \\
    & + \nu \eta \langle h_K^{-1} (\bar{v}_h - m_K(\bar{v}_h)), \bar{v}_h - l_u\rangle_{\partial \mathcal{T}_h}
    \\
    & - \langle \bar{q}_{1,h}, (l_u - \bar{v}_h) \cdot n \rangle_{\partial \mathcal{T}_h}
      - \langle \bar{q}_{1,h}, (\bar{v}_h - m_K(\bar{v}_h)) \cdot n \rangle_{\partial \mathcal{T}_h}
    \\
    & - \langle \bar{q}_{2,h} - \tilde{l}_p(\bar{q}_{2,h}), l_u \cdot n \rangle_{\partial \mathcal{T}_h}
      - (\nabla \tilde{l}_p(\bar{q}_{2,h}), l_u)_{\mathcal{T}_h}
    \\
    & + \langle \bar{q}_{2,h} - \tilde{l}_p(\bar{q}_{2,h}), m_K(\bar{v}_h) \cdot n \rangle_{\partial \mathcal{T}_h}
     + (\nabla \tilde{l}_p(\bar{q}_{2,h}), m_K(\bar{v}_h))_{\mathcal{T}_h}.
  \end{align*}
  Applying the Cauchy--Schwarz and Young's inequalities, a discrete
  trace inequality, and using the definitions of the norms in
  \cref{eq:inner-p}, we obtain
  \begin{align*}
    c_c \tnorm{(l_u, \bar{v}_h)}_v^2
    \le & \del[2]{\tfrac{\varepsilon_1}{2} + \tfrac{\varepsilon_2}{2} + \tfrac{\varepsilon_3}{2} + \tfrac{\varepsilon_4}{2} + \tfrac{\varepsilon_5}{2} + \tfrac{\varepsilon_6}{2} }\tnorm{(l_u, \bar{v}_h)}_v^2
    \\
    & + \del[2]{\tfrac{c_{tr}^2}{2\varepsilon_2} + \tfrac{\eta}{2\varepsilon_3} + \tfrac{\eta}{2}}\nu \norm[0]{h_K^{-1/2}(\bar{v}_h - m_K(\bar{v}_h))}_{\partial \mathcal{T}_h}^2
    \\
         & + \del[2]{\tfrac{1}{2\varepsilon_4} + \tfrac{1}{2} } \nu^{-1} \eta^{-1} \norm[0]{h_K^{1/2} \bar{q}_{1,h}}_{\partial\mathcal{T}_h}^2
    \\
         & + \del[2]{\tfrac{c_{tr}^2}{2\varepsilon_5} + \tfrac{1}{2\varepsilon_6} + \tfrac{c_{tr}^2+1}{2} } \tau^{-1} \tnorm{(\tilde{l}_p(\bar{q}_{2,h}), \bar{q}_{2,h})}_{q,1}^2
    \\
    & + \del[2]{\tfrac{1}{2\varepsilon_1} + 1 }\tau \norm[0]{m_K(\bar{v}_h)}_{\mathcal{T}_h}^2,
  \end{align*}
  with $\varepsilon_i$, $i=1,\hdots,6$, positive constants that are
  free to choose. We choose
  $\varepsilon_1 = \varepsilon_2 = \varepsilon_3 = \varepsilon_4 =
  \varepsilon_5 = \varepsilon_6 = c_c / 6$ and find that there exists
  a uniform constant $c_1' > 0$ such that
  \begin{multline}
    \label{eq:estimate-normv-aux3}
    \tnorm{(l_u, \bar{v}_h)}_v^2
    \leq c_1' \eta \big( \nu \norm[0]{h_K^{-1/2}(\bar{v}_h - m_K(\bar{v}_h))}_{\partial \mathcal{T}_h}^2
    + \nu^{-1} \eta^{-1} \norm[0]{h_K^{1/2} \bar{q}_{1,h}}_{\partial \mathcal{T}_h}^2
    \\
    + \tau^{-1}\tnorm{(\tilde{l}_p(\bar{q}_{2,h}), \bar{q}_{2,h})}_{q,1}^2
    + \tau \norm[0]{m_K(\bar{v}_h)}_{\mathcal{T}_h}^2 \big).
  \end{multline}
  \textbf{Step 2.} We now choose $v_h = l_u$, $q_h = -l_p$, and
  $s = 0$ in \cref{eq:local-solver}. Then, after rearranging terms, we
  find
  \begin{equation}
    \label{eq:estimate-normv-aux4}
    \begin{split}
      \tau \norm[0]{l_u}_K^2
      &+ \nu \norm[0]{\nabla l_u}_K^2
        + \nu \eta \norm[0]{h_K^{-1/2}(l_u - \bar{v}_h)}_{\partial K}^2
        - 2\nu \langle \nabla l_u n, l_u - \bar{v}_h\rangle_{\partial K}
      \\
      =& - \nu \eta h_K^{-1} \langle l_u - \bar{v}_h, \bar{v}_h - m_K(\bar{v}_h) \rangle_{\partial K}
         - \nu \eta h_K^{-1} \langle l_u - \bar{v}_h, m_K(\bar{v}_h) \rangle_{\partial K}
      \\
      & + \nu \langle \nabla l_u n, \bar{v}_h - m_K(\bar{v}_h)\rangle_{\partial K}
        + \nu \langle \nabla l_u n, m_K(\bar{v}_h)\rangle_{\partial K}       
        - \langle \bar{q}_h, l_u \cdot n\rangle_{\partial K}.
    \end{split}
  \end{equation}
  Consider an arbitrary splitting
  $\bar{q}_h = \bar{q}_{1,h} + \bar{q}_{2,h} \in \bar{Q}_h$. Note also
  that
  $0 = - (\tilde{l}_p(\bar{q}_{2,h}), \nabla \cdot l_u)_K = (\nabla
  \tilde{l}_p(\bar{q}_{2,h}), l_u)_K - \langle
  \tilde{l}_p(\bar{q}_{2,h}), l_u \cdot n \rangle_{\partial K}$. For
  the last term on the right hand side of
  \cref{eq:estimate-normv-aux4} we then find:
  \begin{equation}
    \label{eq:estimate-normv-aux5}
    \begin{split}
      \langle \bar{q}_h, l_u \cdot n\rangle_{\partial K}
      =&
         \langle \bar{q}_{1,h}, l_u \cdot n\rangle_{\partial K}
         +
         \langle \bar{q}_{2,h}, l_u \cdot n\rangle_{\partial K}
      \\
      =&
         \langle \bar{q}_{1,h}, (l_u - \bar{v}_h) \cdot n\rangle_{\partial K}
         + \langle \bar{q}_{1,h}, (\bar{v}_h - m_K(\bar{v}_h)) \cdot n\rangle_{\partial K}
      \\
       & + \langle \bar{q}_{1,h}, m_K(\bar{v}_h) \cdot n\rangle_{\partial K}
         + \langle \bar{q}_{2,h} - \tilde{l}_p(\bar{q}_{2,h}), l_u \cdot n\rangle_{\partial K}
      \\
       & + (\nabla \tilde{l}_p(\bar{q}_{2,h}), l_u)_K.
    \end{split}
  \end{equation}
  Combining \cref{eq:estimate-normv-aux4,eq:estimate-normv-aux5},
  summing over all elements, and using \cref{eq:coercivity}, we obtain
  \begin{align*}
    c_c \tnorm{(l_u, \bar{v}_h)}_v^2
    \le& - \nu \eta \langle h_K^{-1} (l_u - \bar{v}_h), \bar{v}_h - m_K(\bar{v}_h) \rangle_{\partial \mathcal{T}_h}
         - \nu \eta \langle h_K^{-1} (l_u - \bar{v}_h), m_K(\bar{v}_h) \rangle_{\partial \mathcal{T}_h}
    \\
       & + \nu \langle \nabla l_u n, \bar{v}_h - m_K(\bar{v}_h)\rangle_{\partial \mathcal{T}_h}
         + \nu \langle \nabla l_u n, m_K(\bar{v}_h)\rangle_{\partial \mathcal{T}_h}
    \\
       & - \langle \bar{q}_{1,h}, (l_u - \bar{v}_h) \cdot n\rangle_{\partial \mathcal{T}_h}
         - \langle \bar{q}_{1,h}, (\bar{v}_h - m_K(\bar{v}_h)) \cdot n\rangle_{\partial \mathcal{T}_h}
    \\
       & - \langle \bar{q}_{1,h}, m_K(\bar{v}_h) \cdot n\rangle_{\partial \mathcal{T}_h}
         - \langle \bar{q}_{2,h} - \tilde{l}_p(\bar{q}_{2,h}), l_u \cdot n\rangle_{\partial \mathcal{T}_h}
    \\
       & - (\nabla \tilde{l}_p(\bar{q}_{2,h}), l_u)_{\mathcal{T}_h}.
  \end{align*}
  Using the Cauchy--Schwarz and Young's inequalities and a discrete
  trace inequality, and the definitions of the norms in
  \cref{eq:inner-p}, we find
  \begin{align*}
    c_c \tnorm{(l_u, \bar{v}_h)}_v^2
    \le &
          \del[2]{\tfrac{\delta_1}{2} + \tfrac{\delta_2}{2} + \tfrac{\delta_3}{2} + \tfrac{\delta_4}{2} + \tfrac{\delta_5}{2} + \tfrac{\delta_6}{2} + \tfrac{\delta_7}{2} } \tnorm{(l_u, \bar{v}_h)}_v^2
    \\
         & + \del[2]{\tfrac{\eta}{2\delta_1} + \tfrac{c_{tr}^2}{2\delta_3} + \tfrac{\eta}{2} } \nu \norm[0]{h_K^{-1/2} (\bar{v}_h - m_K(\bar{v}_h))}_{\partial \mathcal{T}_h}^2
    \\
         & + \del[2]{\tfrac{1}{2\delta_5} + \tfrac{1}{2} + \tfrac{c_{tr}}{2} } \nu^{-1} \eta^{-1} \norm[0]{h_K^{1/2}\bar{q}_{1,h}}_{\partial \mathcal{T}_h}^2
    \\
         & + \del[2]{\tfrac{c_{tr}^2}{2\delta_6} + \tfrac{1}{2\delta_7} } \tau^{-1} \tnorm{(\tilde{l}_p(\bar{q}_{2,h}), \bar{q}_{2,h})}_{q,1}^2
    \\
         & + \del[2]{\tfrac{c_{tr}^2\eta}{2\delta_2} + \tfrac{c_{tr}^4}{2\delta_4} + \tfrac{c_{tr}\eta}{2} } \nu \norm[0]{h_K^{-1}m_K(\bar{v}_h)}_{\mathcal{T}_h}^2,
  \end{align*}
  with $\delta_i$, $i=1,\hdots,7$, positive constants that are free to
  choose. We choose
  $\delta_1 = \delta_2 = \delta_3 = \delta_4 = \delta_5 = \delta_6 =
  \delta_7 = c_c/7$ and find that there exists a uniform constant
  $c_2' > 0$ such that
  \begin{multline}
    \label{eq:estimate-normv-aux6}
    \tnorm{(l_u, \bar{v}_h)}_v^2
    \leq c_2' \eta \big( \nu \norm[0]{h_K^{-1/2} (\bar{v}_h - m_K(\bar{v}_h))}_{\partial \mathcal{T}_h}^2
    + \nu^{-1} \eta^{-1} \norm[0]{h_K^{1/2} \bar{q}_{1,h}}_{\partial \mathcal{T}_h}^2
    \\
    + \tau^{-1} \tnorm{(\tilde{l}_p(\bar{q}_{2,h}), \bar{q}_{2,h})}_{q,1}^2
    + \nu \norm[0]{h_K^{-1} m_K(\bar{v}_h)}_{\mathcal{T}_h}^2 \big).
  \end{multline}
  \textbf{Step 3.} Combining
  \cref{eq:estimate-normv-aux3,eq:estimate-normv-aux6}, we find that
  there exist a uniform constant $c' > 0$ such that
  \begin{multline}
    \label{eq:estimate-normv-aux7}
    \tnorm{(l_u, \bar{v}_h)}_v^2
    \leq c' \eta \big( \nu \norm[0]{h_K^{-1/2}(\bar{v}_h - m_K(\bar{v}_h))}_{\partial \mathcal{T}_h}^2
    + \nu^{-1} \eta^{-1} \norm[0]{h_K^{1/2} \bar{q}_{1,h}}_{\partial \mathcal{T}_h}^2
    \\
    + \tau^{-1} \tnorm{(\tilde{l}_p(\bar{q}_{2,h}), \bar{q}_{2,h})}_{q,1}^2
    + \sum_{K \in \mathcal{T}_h} \min\{\tau, h_K^{-2} \nu\} \norm[0]{m_K(\bar{v}_h)}_{\mathcal{T}_h}^2\big).
  \end{multline}
  Let $\tilde{l}_u$ be defined as in
  \cref{lem:reduced-reac-dif-u}. Combining \cite[Lemma
  A.2]{henriquez2025parameter}, \cref{eq:spec-equiv-tilde-d}, and
  \Cref{lem:schur-comp-equiv}, we find that
  \begin{align*}
    &\sum_{K \in \mathcal{T}_h} \min\cbr[0]{\tau, h_K^{-2} \nu} \norm[0]{m_K(\bar{v}_h)}_{\mathcal{T}_h}^2
    \\
    & \leq c \del[1]{ \tau \norm[0]{\tilde{l}_u(\bar{v}_h)}_{\mathcal{T}_h}^2 
      + \nu \eta \norm[0]{h_K^{-1/2} (\tilde{l}_u(\bar{v}_h) - \bar{v}_h)}_{\partial \mathcal{T}_h}^2 }
    \\
    & \leq c \tnorm{(\tilde{l}_u(\bar{v}_h), \bar{v}_h)}_v^2
    \\
    & \leq c c_c^{-1} \tilde{d}_h((\tilde{l}_u(\bar{v}_h), \bar{v}_h), (\tilde{l}_u(\bar{v}_h), \bar{v}_h))
    \\ 
    & \leq c c_c^{-1} c_d \langle S_{P^u} \bar{v}_h, \bar{v}_h \rangle_{\bar{V}_h^*, \bar{V}_h}.
  \end{align*}
  Therefore, combining the above estimate with
  \cref{eq:estimate-normv-aux7} and taking the infimum over all
  splittings $\bar{q}_h = \bar{q}_{1,h} + \bar{q}_{2,h}$, we conclude
  that \cref{eq:estimate-normv} holds.
\end{proof}

We now show that \cref{eq:cond-precon-1b} holds.

\begin{lemma}
  \label{lem:cond2-th1}
  There exists a uniform constant $c > 0$ such that
  \begin{equation}
    \label{eq:cond2-th1}
    \tnorm{(l_u(\bar{v}_h, \bar{q}_h), \bar{v}_h, l_p(\bar{v}_h, \bar{q}_h), \bar{q}_h)}_{\boldsymbol{X}_h}^2
    \leq c \eta^3 \big(\norm[0]{\bar{v}_h}_{\bar{V}_h}^2
    + \norm[0]{\bar{q}_h}_{\bar{Q}_h}^2 \big) 
    \quad \forall (\bar{v}_h, \bar{q}_h) \in \bar{V}_h \times \bar{Q}_h.
  \end{equation}
\end{lemma}
\begin{proof}
  Combining \cref{eq:bound-qnorm-1,eq:estimate-normv} we find that for
  all $(\bar{v}_h, \bar{q}_h) \in \bar{V}_h \times \bar{Q}_h$
  \begin{equation}
    \label{eq:luvbqbxnormc1p}
    \tnorm{(l_u(\bar{v}_h, \bar{q}_h), \bar{v}_h, l_p(\bar{v}_h, \bar{q}_h), \bar{q}_h)}_{\boldsymbol{X}_h}^2
    \leq c_1' \eta^3 \big( \nu \tnorm{\bar{v}_h}_{v,h}^2 
    + \norm[0]{\bar{v}_h}_{\bar{V}_h}^2
    + \norm[0]{\bar{q}_h}_{\bar{Q}_h}^2\big).
  \end{equation}
  Let $\tilde{l}_u$ be defined as in
  \cref{lem:reduced-reac-dif-u}. Then by \cref{eq:facet-norm-bound}
  and \cref{eq:spec-equiv-tilde-d}
  \begin{equation*}
    \nu\tnorm{\bar{v}_h}_{v,h}^2
    \le
    \bar{c}^2 \tnorm{(\tilde{l}_u(\bar{v}_h), \bar{v}_h)}_v^2
    \le \bar{c}^2 c_c^{-1} \tilde{d}_h((\tilde{l}_u(\bar{v}_h), \bar{v}_h), (\tilde{l}_u(\bar{v}_h), \bar{v}_h)).
  \end{equation*}
  Furthermore, since \cref{eq:spec-equiv-tilde-d} holds, we have by
  \cref{lem:schur-comp-equiv} that
  \begin{equation*}
    \tilde{d}_h((\tilde{l}_u(\bar{v}_h), \bar{v}_h), (\tilde{l}_u(\bar{v}_h), \bar{v}_h))
    \le c_2 \norm[0]{\bar{v}_h}_{\bar{V}_h}^2 \qquad \forall \bar{v}_h \in \bar{V}_h,
  \end{equation*}
  and so
  $\nu \tnorm{\bar{v}_h}_{v,h}^2 \le \bar{c}^2c_c^{-1}c_2
  \norm[0]{\bar{v}_h}_{\bar{V}_h}^2$. Combined with
  \cref{eq:luvbqbxnormc1p} the result follows.
\end{proof}

We conclude this section by remarking that
\cref{lem:cond1,lem:cond2-th1,thm:cond-precon} imply that $\bar{P}$
(cf. \cref{eq:barP}) is a parameter-robust preconditioner for the
reduced problem \cref{eq:condensed-formulation}.

\subsection{Characterization of $\bar{P}^{-1}$}
\label{ss:characterization-preconds}

The inverse of the reduced preconditioner $\bar{P}$ for
\cref{eq:condensed-formulation} is given by
\begin{equation}
  \label{eq:barP-1}
  \bar{P}^{-1} =
  \begin{bmatrix}
    S_{P^u}^{-1} & 0 \\ 0 & (\bar{P}^p)^{-1}
  \end{bmatrix}.
\end{equation}
To determine an expression for $(\bar{P}^p)^{-1}$ which is suitable
for implementation we follow the ideas of \cite[proof of Corollary
1]{baerland2020observation} and \cite[proof of Theorem
3.3]{fu2023uniform}.

For any $\bar{q}_h \in \bar{Q}_h$ we can write
\begin{equation}
  \label{eq:rewrite-hatQnorm}
  \norm[0]{\bar{q}_h}_{\bar{Q}_h}^2
  = \inf_{\bar{\psi}_h \in \bar{Q}_h} 
  \big( \nu^{-1} \langle S_{P^s} (\bar{q}_h - \bar{\psi}_h), \bar{q}_h - \bar{\psi}_h \rangle_{\bar{Q}_h^*, \bar{Q}_h}
  + \tau^{-1} \langle S_{P^d} \bar{\psi}_h, \bar{\psi}_h \rangle_{\bar{Q}_h^*, \bar{Q}_h} \big).
\end{equation}
The infimum is attained by the unique $\bar{\psi}_h \in \bar{Q}_h$
that satisfies
\begin{equation*}
  \nu^{-1} \langle S_{P^s}\bar{\psi}_h, \bar{\phi}_h \rangle_{\bar{Q}_h^*, \bar{Q}_h}
  + \tau^{-1} \langle S_{P^d} \bar{\psi}_h, \bar{\phi}_h \rangle_{\bar{Q}_h^*, \bar{Q}_h}
  = \nu^{-1} \langle S_{P^s}\bar{q}_h, \bar{\phi}_h \rangle_{\bar{Q}_h^*, \bar{Q}_h}
  \quad \forall \bar{\phi}_h \in \bar{Q}_h.
\end{equation*}
Since this holds for all $\bar{\phi}_h \in \bar{Q}_h$ we must have
\begin{equation}
  \label{eq:useful-eq}
  \tau^{-1} S_{P^d} \bar{\psi}_h = \nu^{-1} S_{P^s}(\bar{q}_h - \bar{\psi}_h)
  \quad \text{ and } \quad
  \bar{\psi}_h = (\tau^{-1} S_{P^d} + \nu^{-1} S_{P^s})^{-1} \nu^{-1} S_{P^s} \bar{q}_h.
\end{equation}
Substituting this $\bar{\psi}_h$ in \cref{eq:rewrite-hatQnorm} we find
that
\begin{align*}
  \norm[0]{\bar{q}_h}_{\bar{Q}_h}^2
  & = \nu^{-1} \langle S_{P^s} (\bar{q}_h - \bar{\psi}_h), \bar{q}_h - \bar{\psi}_h \rangle_{\bar{Q}_h^*, \bar{Q}_h}
    + \tau^{-1} \langle S_{P^d} \bar{\psi}_h, \bar{\psi}_h \rangle_{\bar{Q}_h^*, \bar{Q}_h}
  \\
  & = \langle \tau^{-1} S_{P^d} \bar{\psi}_h, \bar{q}_h - \bar{\psi}_h \rangle_{\bar{Q}_h^*, \bar{Q}_h}
    + \tau^{-1} \langle S_{P^d} \bar{\psi}_h, \bar{\psi}_h \rangle_{\bar{Q}_h^*, \bar{Q}_h}
  \\
  & = \langle \tau^{-1} S_{P^d} \bar{\psi}_h, \bar{q}_h \rangle_{\bar{Q}_h^*, \bar{Q}_h}
  \\
  & = \langle \tau^{-1} S_{P^d} (\tau^{-1} S_{P^d} + \nu^{-1} S_{P^s})^{-1} \nu^{-1} S_{P^s} \bar{q}_h, \bar{q}_h \rangle_{\bar{Q}_h^*, \bar{Q}_h}.
\end{align*}
Interpreting the operators as matrices we can note that
\begin{equation}
  \label{eq:inter-step}
  \begin{split}
    \bar{P}^p = \tau^{-1}
    & S_{P^d} (\tau^{-1} S_{P^d} + \nu^{-1} S_{P^s})^{-1} \nu^{-1} S_{P^s}
    \\
    &= \tau^{-1} S_{P^d} (\tau^{-1} S_{P^d} + \nu^{-1} S_{P^s})^{-1}
      ((\tau^{-1} S_{P^d} + \nu^{-1} S_{P^s}) - \tau^{-1} S_{P^d})
    \\  
    &= \tau^{-1} S_{P^d} - \tau^{-1} S_{P^d} (\tau^{-1} S_{P^d} + \nu^{-1} S_{P^s})^{-1} \tau^{-1} S_{P^d}.
  \end{split}
\end{equation}
Using the Sherman--Morrison--Woodbury formula (see, e.g., \cite[page
258]{higham2002accuracy}) it follows that
\begin{equation*}
  (\bar{P}^p)^{-1}
  = \tau S_{P^d}^{-1} + \nu S_{P^s}^{-1},
\end{equation*}
and so we can write \cref{eq:barP-1} as
\begin{equation*}
  \bar{P}^{-1}  
  = \begin{bmatrix}
    S_{P^u}^{-1} & 0\\
    0 & \tau S_{P^d}^{-1} + \nu S_{P^s}^{-1}
  \end{bmatrix}.
\end{equation*}

\subsection{An alternative parameter-robust preconditioner  $\widehat{P}^{-1}$}
\label{ss:characterization-preconds-alt}

Consider a block diagonal operator $\widehat{P}$ with the same block
structure as $\bar{P}$ (see \cref{eq:barPP11P22}):
\begin{equation*}
  \widehat{P} =
  \begin{bmatrix}
    \widehat{P}_{11} & 0 \\ 0 & \widehat{P}_{22}
  \end{bmatrix}.
\end{equation*}
If $\widehat{P}_{11}$ and $\widehat{P}_{22}$ are norm equivalent to
$\bar{P}_{11}$ and $\bar{P}_{22}$, respectively, then $\widehat{P}$ is
also a parameter-robust preconditioner for the reduced problem
\cref{eq:condensed-formulation}. Choosing
$\widehat{P}_{22} = \bar{P}_{22}$, we present here an alternative for
$\bar{P}_{11}$.

Consider the operator $\widehat{S}_{P^u} : \bar{V}_h \to \bar{V}_h^*$
which is defined by
\begin{equation*}
  \langle \widehat{S}_{P^u} \bar{u}_h, \bar{v}_h \rangle_{\bar{V}_h^*, \bar{V}_h}
  = \tilde{d}_h((\tilde{l}_u(\bar{u}_h), \bar{u}_h), (\tilde{l}_u(\bar{v}_h), \bar{v}_h)).
\end{equation*}
Since $\tilde{d}_h(\cdot, \cdot)$ is bounded and coercive we have that
$\widehat{S}_{P^u}$ is norm-equivalent to $S_{P^u}$.

An alternative parameter-robust preconditioner is therefore given by
\begin{equation*}
  \label{eq:barP-2}
  \widehat{P}^{-1}  
  = \begin{bmatrix}
    \widehat{S}_{P^u}^{-1} & 0\\
    0 & \tau S_{P^d}^{-1} + \nu S_{P^s}^{-1}
  \end{bmatrix}.
\end{equation*}

\section{Numerical examples}
\label{s:numex}

In this section we verify our analysis by demonstrating that
$\bar{P}^{-1}$ and $\widehat{P}^{-1}$ are parameter-robust
preconditioners for the reduced problem
\cref{eq:condensed-formulation}. We solve
\cref{eq:condensed-formulation}, preconditioned by $\bar{P}^{-1}$ and
$\widehat{P}^{-1}$, using MINRES with a relative preconditioned
residual tolerance of $10^{-8}$ for two dimensional simulations and
$10^{-6}$ for three dimensional simulations. We also consider the
performance of $\bar{P}^{-1}$ and $\widehat{P}^{-1}$ as
preconditioners for an EDG-HDG discretization of the time-dependent
Stokes equations \cite{rhebergen2020embedded}. An EDG-HDG
discretization is obtained by replacing $\bar{V}_h$ in
\cref{eq:hdgmethod} by $\bar{V}_h \cap C^0(\Gamma_0)$.

We consider both exact and inexact versions of the preconditioners. For
the inexact versions we apply a balancing domain decomposition by
constraints (BDDC) method \cite{schoberl2013domain} to approximate
$S_{P^u}^{-1}$, $\widehat{S}_{P^u}^{-1}$, and $\tau S_{P^d}^{-1}$. To
approximate $\nu S_{P^s}^{-1}$ we use a direct solver.

All simulations are performed with a polynomial degree of $k = 2$
and on unstructured simplicial meshes generated by Netgen
\cite{schoberl1997netgen}. NGSolve \cite{schoberl2014c++} was used to
implement the discretizations. 

\subsection{Example 1: manufactured solutions}
\label{ss:tc1}

We consider manufactured solutions in the domain $\Omega =
(0,1)^d$. In two dimensions the source and boundary terms are set such
that
\begin{equation*}
  u =
  \begin{bmatrix}
    \sin(\pi x_1) \sin(\pi x_2) \\
    \cos(\pi x_1) \cos(\pi x_2)    
  \end{bmatrix},
  \quad 
  p = \sin(\pi x_1) \cos(\pi x_2),
\end{equation*}
while in three dimensions these are set such that
\begin{equation*}
  u =
  \begin{bmatrix}
    \pi \sin(\pi x_1) \cos(\pi x_2) - \pi \sin(\pi x_1) \cos(\pi x_3)  \\
    \pi \sin(\pi x_2) \cos(\pi x_3) - \pi \sin(\pi x_2) \cos(\pi x_1)  \\
    \pi \sin(\pi x_3) \cos(\pi x_1) - \pi \sin(\pi x_3) \cos(\pi x_2)
  \end{bmatrix},
  \quad 
  p = \cos(\pi x_1) \sin(\pi x_2) \cos(\pi x_3).
\end{equation*}
For this example the penalization term is chosen as $\eta = 4 k^2$ for
the 2D case and $\eta = 6 k (k+1)$ for the 3D case.  In
\Cref{tab:pre1vpre2} we consider the $h$-robustness of the
preconditioners $\bar{P}^{-1}$ and $\widehat{P}^{-1}$. We set
$\nu = \tau = 1$ and observe that both preconditioners are $h$-robust
in two and three dimensions, for HDG and EDG-HDG discretizations, and
for their exact and inexact versions. The results furthermore indicate
that $\widehat{P}^{-1}$ outperforms $\bar{P}^{-1}$ in terms of the
number of iterations. Therefore, in the remaining experiments, we only
report results for $\widehat{P}^{-1}$.

In \Cref{tab:paramet-robust} we consider the parameter ($\nu$ and
$\tau$) robustness of the preconditioner $\widehat{P}^{-1}$. For this
we consider a fixed mesh with 32768 simplices for the two-dimensional
problem and 3072 simplices for the three-dimensional problem. Although
there is some variation in the iteration count, this variation is
relatively small considering the large variation in the parameter
ratio, $10^{-6} \le \nu/\tau \le 1$, with slightly smaller iteration
count when $\nu/\tau$ is small. This is observed for both the exact
and inexact versions of the preconditioner and in two and three
dimensions.

\begin{table}[tbp]
  \centering
  \resizebox{\textwidth}{!}{%
    \begin{tabular}{c|c|c|c|c|c||c|c|c|c|c}
      \hline
      \multicolumn{11}{c}{EDG-HDG}\\
      \hline
      & \multicolumn{5}{c||}{2d} & \multicolumn{5}{c}{3d} \\
      \hline
      Cells & 512 & 2048 & 8192 & 32768 & 131072 & 48 & 384 & 3072 & 24576 & 196608\\
      \hline
      $\bar{P}^{-1}$ & 80 (103) & 83 (106) & 82 (107) & 82 (108) & 82 (107) & 56 (79) & 70 (90) & 69 (87) & 65 (86) & 65 (81) \\
      \hline
      $\widehat{P}^{-1}$ & 71 (94) & 71 (96) & 71 (97) & 71 (97) & 71 (97) & 48 (73) & 60 (76) & 61 (75) & 55 (75) & 54 (68) \\
      \hline
      \multicolumn{11}{c}{HDG}\\
      \hline
      & \multicolumn{5}{c||}{2d} & \multicolumn{5}{c}{3d} \\
      \hline
      Cells & 512 & 2048 & 8192 & 32768 & 131072 & 48 & 384 & 3072 & 24576 & 196608\\
      \hline
      $\bar{P}^{-1}$ & 92 (118) & 91 (120) & 90 (122) & 90 (122) & 89 (121) & 79 (110) & 103 (138) & 107 (136) & 104 (136) & 102 (136)\\
      \hline
      $\widehat{P}^{-1}$ & 82 (110) & 82 (112) & 81 (111) & 79 (111) & 79 (111) & 74 (105) & 94 (119) & 98 (118) & 95 (118) & 93
                                                                                                                              (118)\\
      \hline
    \end{tabular}}
  \caption{$h$-robustness for the manufactured solution test case of
    \cref{ss:tc1}. The table reports the number of iterations required
    by preconditioned MINRES (preconditioners $\bar{P}^{-1}$ and
    $\widehat{P}^{-1}$ in both their exact and inexact forms) to
    achieve convergence for different mesh sizes $h$. The results
    corresponding to the inexact preconditioners are shown in
    parentheses.  }
  \label{tab:pre1vpre2}
\end{table}

\begin{table}[tbp]
  \centering
  \begin{tabular}{c|c|c|c||c|c|c}
    \hline
    \multicolumn{7}{c}{EDG-HDG}\\
    \hline
    & \multicolumn{3}{c||}{2d} & \multicolumn{3}{c}{3d} \\
    \hline
    & $\tau = 1$ & $\tau = 10^2$ & $\tau = 10^3$ & $\tau = 1$ & $\tau = 10^2$ & $\tau = 10^3$ \\
    \hline
    $\nu = 1$ & 68 (89) & 68 (89) & 68 (88) & 60 (78) & 55 (73) & 49 (59) \\
    \hline
    $\nu = 10^{-2}$ & 70 (91) & 65 (81) & 56 (66) & 40 (97) & 35 (47) & 33 (36) \\
    \hline
    $\nu = 10^{-3}$ & 69 (91) & 56 (66) & 46 (49) & 69 (81) & 42 (43) & 36 (36) \\
    \hline
    \multicolumn{7}{c}{HDG}\\
    \hline
    & \multicolumn{3}{c||}{2d} & \multicolumn{3}{c}{3d} \\
    \hline
    & $\tau = 1$ & $\tau = 10^{2}$ & $\tau = 10^{3}$ & $\tau = 1$ & $\tau = 10^{2}$ & $\tau = 10^{3}$ \\
    \hline
    $\nu = 1$ & 79 (105) & 79 (104) & 77 (103) & 99 (123) & 97 (98) & 85 (98) \\
    \hline
    $\nu = 10^{-2}$ & 79 (105) & 76 (97) & 64 (78) & 108 (128) & 60 (97) & 50 (103) \\
    \hline
    $\nu = 10^{-3}$ & 79 (104) & 64 (78) & 49 (68) & 98 (116) & 55 (112) & 54 (107) \\
    \hline
  \end{tabular}
  \caption{Parameter-robustness for the manufactured solution test
    case of \cref{ss:tc1}. The table reports the number of iterations
    required by preconditioned MINRES method (preconditioner
    $\widehat{P}^{-1}$ in both its exact and inexact forms) to
    achieve convergence for different values of $\nu$ and $\tau$. The
    results corresponding to the inexact form of $\widehat{P}^{-1}$
    are shown in parentheses.}
  \label{tab:paramet-robust}
\end{table}

\subsection{Example 2: lid-driven cavity problem}
\label{ss:tc2}

We now consider a lid-driven cavity problem in two and three
dimensions. In two dimensions we consider the domain
$\Omega = (-1, 1)^2$ and impose $u = (1 - x_1^4, 0)$ on the boundary
$x_2 = 1$ and $u = (0, 0)$ on the remaining boundaries. To evaluate
the performance of $\widehat{P}^{-1}$ we use a fixed mesh consisting
of 59392 simplices and we set $\eta = 6 k^2$. In three dimensions we
consider the domain $\Omega = (0, 1)^3$ and impose
$u = (1 - \tau_1^4, (1 - \tau_2^2)^4 / 10, 0)$, where
$\tau_i = 2x_i - 1$, on the boundary $x_3 = 1$ and $u = (0, 0, 0)$ on
the remaining boundaries. In three dimensions we consider a fixed mesh
with 3072 simplices and $\eta = 6 k (k+1)$.

In \Cref{tab:lid-paramet-robust} we consider the $\nu$ and $\tau$
parameter-robustness of the preconditioner $\widehat{P}^{-1}$. We
observe similar behaviour as in \cref{ss:tc1}, i.e., we observe some
variation in the iteration count, but this variation is acceptable
given the large variation in the ratio $\nu/\tau$, which ranges
between $10^{-6}$ and $1$.

\begin{table}[tbp]
  \centering
  \begin{tabular}{c|c|c|c||c|c|c}
    \hline
    \multicolumn{7}{c}{EDG-HDG}\\
    \hline
    & \multicolumn{3}{c||}{2d} & \multicolumn{3}{c}{3d} \\
    \hline
    & $\tau = 1$ & $\tau = 10^2$ & $\tau = 10^3$ & $\tau = 1$ & $\tau = 10^2$ & $\tau = 10^3$ \\
    \hline
    $\nu = 1$ & 92 (135) & 87 (125) & 79 (100) & 91 (118) & 85 (105) & 74 (82) \\
    \hline
    $\nu = 10^{-2}$ & 87 (125) & 71 (74) & 73 (74) & 85 (105) & 85 (64) & 45 (47) \\
    \hline
    $\nu = 10^{-3}$ & 79 (100) & 69 (73) & 78 (86) & 74 (82) & 48 (49) & 38 (38) \\
    \hline
    \multicolumn{7}{c}{HDG}\\
    \hline
    & \multicolumn{3}{c||}{2d} & \multicolumn{3}{c}{3d} \\
    \hline
    & $\tau = 1$ & $\tau = 10^{2}$ & $\tau = 10^{3}$ & $\tau = 1$ & $\tau = 10^{2}$ & $\tau = 10^{3}$ \\
    \hline
    $\nu = 1$ & 107 (158) & 102 (151) & 96 (122) & 122 (166) & 120 (155) & 108 (124) \\
    \hline
    $\nu = 10^{-2}$ & 102 (151) & 81 (114) & 84 (166) & 120 (155) & 79 (127) & 66 (138) \\
    \hline
    $\nu = 10^{-3}$ & 96 (122) & 80 (157) & 74 (159) & 108 (124) & 74 (152) & 59 (141) \\
    \hline
  \end{tabular}
\caption{Parameter-robustness for the lid-driven cavity test case of
  \cref{ss:tc2}. The table reports the number of iterations required
  by preconditioned MINRES (preconditioner $\widehat{P}^{-1}$ in both
  its exact and inexact forms) to achieve convergence for different
  values of $\nu$ and $\tau$. The results corresponding to the inexact
  form of $\widehat{P}^{-1}$ are shown in parentheses.}
  \label{tab:lid-paramet-robust}
\end{table}

\subsection{Example 3: Brinkman heterogeneous media case}
\label{ss:tc3}

In this section we consider the Brinkman model, i.e., we consider
$\tau$ in \cref{eq:tstokes-model-be} to be spatially varying. In two
dimensions we define $\Omega = (0, 1)^2$, $f = (1, 1)$, and
$u = (0, 0)$ on $\partial \Omega$. We use a fixed mesh consisting of
32768 simplices and we set $\eta = 4 k^2$. We consider the following
expression for $\tau$:
\begin{equation*}
  \tau(x_1, x_2) := 0.5 \cdot 10^{6} (1 + 10^{-6} 
  + \sin(8.3\pi x_1) \sin(6.2\pi x_2)).
\end{equation*}
In three dimensions we define $\Omega = (0, 1)^3$, $f = (1, 1, 1)$,
and $u = (0, 0, 0)$ on $\partial \Omega$. We use a fixed mesh
consisting of 3072 simplices and $\eta = 6 k^2$. We consider the 
following expression for $\tau$:
\begin{equation*}
  \tau(x_1, x_2, x_3) := 0.5 \cdot 10^{6} (1 + 10^{-6} 
  + \sin(8.3\pi x_1) \sin(6.2\pi x_2) \sin(5.1\pi x_3)).
\end{equation*}
In \Cref{tab:brinkman-paramet-robust} we list the number of iterations
required for preconditioned MINRES (with preconditioner
$\widehat{P}^{-1}$) to converge for different values of $\nu$. In two
dimensions we observe for both HDG and EDG-HDG discretizations that
the iteration count varies slightly with higher iteration count when
$\nu$ is small. In three dimensions we also observe small variations
in the iteration count, however, here we observe that iteration count
is slightly lower for small $\nu$. As in the previous sections, the
variation in iteration count is small given the large variation in
parameters.

\begin{table}[tbp]
  \centering
  \begin{tabular}{c|c|c|c||c|c|c}
    \hline
    & \multicolumn{3}{c||}{EDG-HDG} & \multicolumn{3}{c}{HDG} \\
    \hline
    & $\nu = 1$ & $\nu = 10^{-2}$ & $\nu = 10^{-3}$ & $\nu = 1$ & $\nu = 10^{-2}$ & $\nu = 10^{-3}$ \\
    \hline
    2d & 67 (87) & 95 (111) & 86 (101) & 73 (117) & 98 (193) & 91 (183) \\
    \hline
    3d & 57 (62) & 36 (36) & 33 (33) & 64 (165) & 50 (119) & 45 (108) \\
    \hline
  \end{tabular}
  \caption{Parameter-robustness for the Brinkman heterogeneous media
    test case of \cref{ss:tc3}. We list the number of iterations
    required for preconditioned MINRES (preconditioner
    $\widehat{P}^{-1}$ in both its exact and inexact form) to converge
    for different values of $\nu$. The results for the inexact form of
    $\widehat{P}^{-1}$ are shown in parentheses.}
  \label{tab:brinkman-paramet-robust}
\end{table}

\section{Conclusions}
\label{s:conclusions}

In this paper we presented parameter-robust preconditioners for the
reduced linear system arising from applying static condensation to an
HDG discretization of the time-dependent Stokes problem. In the
process we generalized the Schur complement approach that we presented
in our previous work \cite{henriquez2025parameter} and proved uniform
well-posedness of the non-condensed HDG discretization of the
time-dependent Stokes equations.

The Schur complement approach that we presented in our previous work
\cite{henriquez2025parameter} is difficult to use for this problem
because this would need the construction of the Schur complement
corresponding to the norms of function spaces that are intersections
or sums of Hilbert spaces. To overcome this difficulty we proposed new
face-norm conditions generalizing the condition introduced in
\cite{henriquez2025parameter}. As a consequence we devised new
theoretical tools for preconditioning of statically condensed systems
in which norms of function spaces that are intersections or sums of
Hilbert spaces are involved.

A key step to proving uniform well-posedness of the non-condensed HDG
discretization of the time-dependent Stokes equations is proving
uniform inf-sup stability of the velocity/pressure coupling term. We
presented a detailed proof of this result.

Finally, numerical test results for the time-dependent Stokes
equations and the Brinkman equations verify our theoretical results.

\section*{Acknowledgments}

SR acknowledges support from the Natural Sciences and Engineering
Research Council of Canada through the Discovery Grant program
(RGPIN-2023-03237).

\bibliographystyle{plain}
\bibliography{references}

@article{southworth2020fixed,
  title={On Fixed-Point, Krylov, and $2 \times 2$ Block Preconditioners for Nonsymmetric Problems},
  author={Southworth, Ben S and Sivas, Abdullah A and Rhebergen, Sander},
  journal={SIAM Journal on Matrix Analysis and Applications},
  volume={41},
  number={2},
  pages={871--900},
  year={2020},
  publisher={SIAM}
}

@phdthesis{sivas2021preconditioning,
  title={Preconditioning of hybridizable discontinuous Galerkin discretizations of the Navier-Stokes equations},
  author={Sivas, Abdullah Ali},
  year={2021},
  school={University of Waterloo}
}

@article{lindsay2025preconditioning,
  title={Preconditioning a hybridizable discontinuous {G}alerkin method for {N}avier-{S}tokes at high {R}eynolds number},
  author={Lindsay, Alexander D and Rhebergen, Sander and Southworth, Ben S},
  journal={arXiv preprint arXiv:2512.02971},
  year={2025}
}

@article{he2021local,
  title={Local {F}ourier analysis of multigrid for hybridized and embedded discontinuous {G}alerkin methods},
  author={He, Yunhui and Rhebergen, Sander and Sterck, Hans De},
  journal={SIAM Journal on Scientific Computing},
  volume={43},
  number={5},
  pages={S612--S636},
  year={2021},
  publisher={SIAM}
}

@article{rhebergen2018preconditioning,
  title={Preconditioning of a hybridized discontinuous {G}alerkin finite element method for the {S}tokes equations},
  author={Rhebergen, Sander and Wells, Garth N},
  journal={Journal of Scientific Computing},
  volume={77},
  number={3},
  pages={1936--1952},
  year={2018},
  publisher={Springer}
}

@article{rhebergen2018hybridizable,
  title={A hybridizable discontinuous {G}alerkin method for the {N}avier--{S}tokes equations with pointwise divergence-free velocity field},
  author={Rhebergen, Sander and Wells, Garth N},
  journal={Journal of Scientific Computing},
  volume={76},
  number={3},
  pages={1484--1501},
  year={2018},
  publisher={Springer}
}

@article{mardal2011preconditioning,
  title={Preconditioning discretizations of systems of partial differential equations},
  author={Mardal, Kent-Andre and Winther, Ragnar},
  journal={Numerical Linear Algebra with Applications},
  volume={18},
  number={1},
  pages={1--40},
  year={2011},
  publisher={Wiley Online Library}
}

@book{di2011mathematical,
  title={Mathematical aspects of discontinuous {G}alerkin methods},
  author={Di Pietro, Daniele Antonio and Ern, Alexandre},
  volume={69},
  year={2011},
  publisher={Springer Science \& Business Media}
}

@article{rhebergen2017analysis,
  title={Analysis of a hybridized/interface stabilized finite element method for the {S}tokes equations},
  author={Rhebergen, Sander and Wells, Garth N},
  journal={SIAM Journal on Numerical Analysis},
  volume={55},
  number={4},
  pages={1982--2003},
  year={2017},
  publisher={SIAM}
}

@article{rhebergen2020embedded,
  title={An embedded--hybridized discontinuous {G}alerkin finite element method for the {S}tokes equations},
  author={Rhebergen, Sander and Wells, Garth N},
  journal={Computer Methods in Applied Mechanics and Engineering},
  volume={358},
  pages={112619},
  year={2020},
  publisher={Elsevier}
}

@book{boffi2013mixed,
  title={Mixed finite element methods and applications},
  author={Boffi, Daniele and Brezzi, Franco and Fortin, Michel},
  volume={44},
  year={2013},
  publisher={Springer}
}

@article{kraus2021uniformly,
  title={Uniformly well-posed hybridized discontinuous {G}alerkin/hybrid mixed discretizations for {B}iot’s consolidation model},
  author={Kraus, Johannes and Lederer, Philip L and Lymbery, Maria and Sch{\"o}berl, Joachim},
  journal={Computer Methods in Applied Mechanics and Engineering},
  volume={384},
  pages={113991},
  year={2021},
  publisher={Elsevier}
}

@article{schoberl1997netgen,
  title={{NETGEN} An advancing front {2D/3D}-mesh generator based on abstract rules},
  author={Sch{\"o}berl, Joachim},
  journal={Computing and visualization in science},
  volume={1},
  number={1},
  pages={41--52},
  year={1997},
  publisher={Springer}
}

@article{schoberl2014c++,
  title={C++ 11 implementation of finite elements in {NGS}olve},
  author={Sch{\"o}berl, Joachim},
  journal={Institute for analysis and scientific computing, Vienna University of Technology},
  volume={30},
  year={2014},
  publisher={Citeseer}
}

@article{rhebergen2022preconditioning,
  title={Preconditioning for a pressure-robust {HDG} discretization of the {S}tokes equations},
  author={Rhebergen, Sander and Wells, Garth N},
  journal={SIAM Journal on Scientific Computing},
  volume={44},
  number={1},
  pages={A583--A604},
  year={2022},
  publisher={SIAM}
}

@book{du2019invitation,
  title={An {I}nvitation to the {T}heory of the {H}ybridizable {D}iscontinuous {G}alerkin {M}ethod: {P}rojections, {E}stimates, {T}ools},
  author={Du, Shukai and Sayas, Francisco-Javier},
  year={2019},
  publisher={Springer Nature}
}

@article{cockburn2009unified,
  title={Unified hybridization of discontinuous {G}alerkin, mixed, and continuous {G}alerkin methods for second order elliptic problems},
  author={Cockburn, Bernardo and Gopalakrishnan, Jayadeep and Lazarov, Raytcho},
  journal={SIAM Journal on Numerical Analysis},
  volume={47},
  number={2},
  pages={1319--1365},
  year={2009},
  publisher={SIAM}
}

@article{cockburn2014multigrid,
  title={Multigrid for an {HDG} method},
  author={Cockburn, Bernardo and Dubois, Olivier and Gopalakrishnan, Jay and Tan, Shuguang},
  journal={IMA Journal of Numerical Analysis},
  volume={34},
  number={4},
  pages={1386--1425},
  year={2014},
  publisher={OUP}
}

@article{fu2023uniform,
  title={Uniform block-diagonal preconditioners for divergence-conforming {HDG} Methods for the generalized {S}tokes equations and the linear elasticity equations},
  author={Fu, Guosheng and Kuang, Wenzheng},
  journal={IMA Journal of Numerical Analysis},
  volume={43},
  number={3},
  pages={1718--1741},
  year={2023},
  publisher={Oxford University Press}
}

@article{fu2021uniform,
  title={Uniform auxiliary space preconditioning for {HDG} methods for elliptic operators with a parameter dependent low order term},
  author={Fu, Guosheng},
  journal={SIAM Journal on Scientific Computing},
  volume={43},
  number={6},
  pages={A3912--A3937},
  year={2021},
  publisher={SIAM}
}

@article{henriquez2025parameter,
    author = {Henr\'{\i}quez, Esteban and Lee, Jeonghun J. and Rhebergen, Sander},
    title = {Parameter-Robust Preconditioning for Hybridizable Symmetric Discretizations},
    journal = {SIAM Journal on Scientific Computing},
    volume = {47},
    number = {6},
    pages = {A3212-A3238},
    year = {2025},
}

@incollection{schoberl2013domain,
  title={Domain decomposition preconditioning for high order hybrid discontinuous {G}alerkin methods on tetrahedral meshes},
  author={Sch{\"o}berl, Joachim and Lehrenfeld, Christoph},
  booktitle={Advanced finite element methods and applications},
  pages={27--56},
  year={2013},
  publisher={Springer}
}

@article{kraus2023hybridized,
  title={Hybridized discontinuous Galerkin/hybrid mixed methods for a multiple network poroelasticity model with application in biomechanics},
  author={Kraus, Johannes and Lederer, Philip L and Lymbery, Maria and Osthues, Kevin and Sch{\"o}berl, Joachim},
  journal={SIAM Journal on Scientific Computing},
  volume={45},
  number={6},
  pages={B802--B827},
  year={2023},
  publisher={SIAM}
}

@article{wells2011analysis,
  title={Analysis of an interface stabilized finite element method: the advection-diffusion-reaction equation},
  author={Wells, Garth N},
  journal={SIAM journal on numerical analysis},
  volume={49},
  number={1},
  pages={87--109},
  year={2011},
  publisher={SIAM}
}

@article{mardal2004uniform,
  title={Uniform preconditioners for the time dependent {S}tokes problem},
  author={Mardal, Kent-Andre and Winther, Ragnar},
  journal={Numerische Mathematik},
  volume={98},
  number={2},
  pages={305--327},
  year={2004},
  publisher={Springer}
}

@article{baerland2020observation,
  title={An observation on the uniform preconditioners for the mixed {D}arcy problem},
  author={B{\ae}rland, Trygve and Kuchta, Miroslav and Mardal, Kent-Andre and Thompson, Travis},
  journal={Numerical Methods for Partial Differential Equations},
  volume={36},
  number={6},
  pages={1718--1734},
  year={2020},
  publisher={Wiley Online Library}
}

@book{higham2002accuracy,
  title={Accuracy and stability of numerical algorithms},
  author={Higham, Nicholas J},
  year={2002},
  publisher={SIAM}
}

@article{cahouet1988some,
  title={Some fast {3D} finite element solvers for the generalized {S}tokes problem},
  author={Cahouet, J and Chabard, J-P},
  journal={International Journal for Numerical Methods in Fluids},
  volume={8},
  number={8},
  pages={869--895},
  year={1988},
  publisher={Wiley Online Library}
}

@article{bramble1997iterative,
  title={Iterative techniques for time dependent {S}tokes problems},
  author={Bramble, James H and Pasciak, Joseph E},
  journal={Computers \& Mathematics with Applications},
  volume={33},
  number={1-2},
  pages={13--30},
  year={1997},
  publisher={Elsevier}
}

@article{olshanskii2006uniform,
  title={Uniform preconditioners for a parameter dependent saddle point problem with application to generalized {S}tokes interface equations},
  author={Olshanskii, Maxim A and Peters, J{\"o}rg and Reusken, Arnold},
  journal={Numerische Mathematik},
  volume={105},
  number={1},
  pages={159--191},
  year={2006},
  publisher={Springer}
}

@article{lu2022homogeneous,
  title={Homogeneous multigrid for {HDG}},
  author={Lu, Peipei and Rupp, Andreas and Kanschat, Guido},
  journal={IMA Journal of Numerical Analysis},
  volume={42},
  number={4},
  pages={3135--3153},
  year={2022},
  publisher={Oxford University Press}
}

@article{lu2022analysis,
  title={Analysis of injection operators in geometric multigrid solvers for {HDG} methods},
  author={Lu, Peipei and Rupp, Andreas and Kanschat, Guido},
  journal={SIAM Journal on Numerical Analysis},
  volume={60},
  number={4},
  pages={2293--2317},
  year={2022},
  publisher={SIAM}
}

@article{lu2024homogeneous,
  title={Homogeneous multigrid for {HDG} applied to the {S}tokes equation},
  author={Lu, Peipei and Wang, Wei and Kanschat, Guido and Rupp, Andreas},
  journal={IMA Journal of Numerical Analysis},
  volume={44},
  number={5},
  pages={3124--3152},
  year={2024},
  publisher={Oxford University Press}
}

@article{tu2021bddc,
  title={{BDDC} algorithms for advection-diffusion problems with {HDG} discretizations},
  author={Tu, Xuemin and Zhang, Jinjin},
  journal={Computers \& Mathematics with Applications},
  volume={101},
  pages={74--106},
  year={2021},
  publisher={Elsevier}
}

@article{tu2020analysis,
  title={Analysis of {BDDC} algorithms for {S}tokes problems with hybridizable discontinuous {G}alerkin discretizations},
  author={Tu, Xuemin and Wang, Bin and Zhang, Jinjin},
  journal={Electron. Trans. Numer. Anal.},
  volume={52},
  pages={553--570},
  year={2020}
}

@article{zhang2022robust,
  title={Robust {BDDC} algorithms for the {B}rinkman problem with {HDG} discretizations},
  author={Zhang, Jinjin and Tu, Xuemin},
  journal={Computer Methods in Applied Mechanics and Engineering},
  volume={400},
  pages={115548},
  year={2022},
  publisher={Elsevier}
}

@article{lu2023two,
  title={Two-level {S}chwarz methods for hybridizable discontinuous {G}alerkin methods},
  author={Lu, Peipei and Rupp, Andreas and Kanschat, Guido},
  journal={Journal of Scientific Computing},
  volume={95},
  number={1},
  pages={9},
  year={2023},
  publisher={Springer}
}

@article{yu2024nonoverlapping,
  title={Nonoverlapping spectral additive {S}chwarz methods for hybrid discontinuous {G}alerkin discretizations},
  author={Yu, Yi and Dryja, Maksymilian and Sarkis, Marcus},
  journal={IMA Journal of Numerical Analysis},
  volume={44},
  number={1},
  pages={192--224},
  year={2024},
  publisher={Oxford University Press}
}

@article{henriquez2025preconditioning,
  title={Preconditioning of a hybridizable discontinuous {G}alerkin method for {B}iot's consolidation model},
  author={Henr{\'\i}quez, Esteban and Lee, Jeonghun J and Rhebergen, Sander},
  journal={arXiv preprint arXiv:2508.12991},
  year={2025}
}

@article{sivas2021air,
  title={{AIR} algebraic multigrid for a space-time hybridizable discontinuous {G}alerkin discretization of advection (-diffusion)},
  author={Sivas, Abdullah Ali and Southworth, Ben S and Rhebergen, Sander},
  journal={SIAM Journal on Scientific Computing},
  volume={43},
  number={5},
  pages={A3393--A3416},
  year={2021},
  publisher={SIAM}
}

@book {Galdi:NSE-book,
    AUTHOR = {Galdi, G. P.},
     TITLE = {An introduction to the mathematical theory of the
              {N}avier-{S}tokes equations},
    SERIES = {Springer Monographs in Mathematics},
   EDITION = {Second},
      NOTE = {Steady-state problems},
 PUBLISHER = {Springer, New York},
      YEAR = {2011},
     PAGES = {xiv+1018},
      ISBN = {978-0-387-09619-3},
   MRCLASS = {35Q30 (35-02 76D03 76D05 76D07)},
  MRNUMBER = {2808162},
}

@article {Arnold-Guzman:2021,
    AUTHOR = {Arnold, Douglas and Guzm\'an, Johnny},
     TITLE = {Local {$L^2$}-bounded commuting projections in {FEEC}},
   JOURNAL = {ESAIM Math. Model. Numer. Anal.},
  FJOURNAL = {ESAIM. Mathematical Modelling and Numerical Analysis},
    VOLUME = {55},
      YEAR = {2021},
    NUMBER = {5},
     PAGES = {2169--2184},
      ISSN = {2822-7840,2804-7214},
   MRCLASS = {65N30},
  MRNUMBER = {4323408},
MRREVIEWER = {Martin\ W.\ Licht},
}

@incollection {Geissert:2006,
    AUTHOR = {Gei\ss ert, Matthias and Heck, Horst and Hieber, Matthias},
     TITLE = {On the equation {${\rm div}\,u=g$} and {B}ogovski\u i's
              operator in {S}obolev spaces of negative order},
 BOOKTITLE = {Partial differential equations and functional analysis},
    SERIES = {Oper. Theory Adv. Appl.},
    VOLUME = {168},
     PAGES = {113--121},
 PUBLISHER = {Birkh\"auser, Basel},
      YEAR = {2006},
      ISBN = {978-3-7643-7600-4; 3-7643-7600-7},
   MRCLASS = {35F10 (35Q35)},
  MRNUMBER = {2240056},
}

@article {Ern-Gudi:2022,
    AUTHOR = {Ern, Alexandre and Gudi, Thirupathi and Smears, Iain and
              Vohral\'ik, Martin},
     TITLE = {Equivalence of local- and global-best approximations, a simple
              stable local commuting projector, and optimal {$hp$}
              approximation estimates in {$\bold H({\rm div})$}},
   JOURNAL = {IMA J. Numer. Anal.},
  FJOURNAL = {IMA Journal of Numerical Analysis},
    VOLUME = {42},
      YEAR = {2022},
    NUMBER = {2},
     PAGES = {1023--1049},
      ISSN = {0272-4979,1464-3642},
   MRCLASS = {65N30 (65N15)},
  MRNUMBER = {4410735},
}

@article {Gawlik:2021,
    AUTHOR = {Gawlik, Evan and Holst, Michael J. and Licht, Martin W.},
     TITLE = {Local finite element approximation of {S}obolev differential
              forms},
   JOURNAL = {ESAIM Math. Model. Numer. Anal.},
  FJOURNAL = {ESAIM. Mathematical Modelling and Numerical Analysis},
    VOLUME = {55},
      YEAR = {2021},
    NUMBER = {5},
     PAGES = {2075--2099},
      ISSN = {2822-7840,2804-7214},
   MRCLASS = {65N30},
  MRNUMBER = {4319601},
MRREVIEWER = {Marius\ Ghergu},
}
\appendix
\section{Useful results}
\label{sec:useful-ineq}

\begin{lemma}
  \label{lem:facet-norm-bound}
  There exists a positive uniform constant $\bar{c}$ such that
  \begin{equation}
    \label{eq:facet-norm-bound}
    \tnorm{\bar{v}_h}_{v,h} \leq \bar{c} \tnorm{\boldsymbol{v}_h}_{v,1}
    \quad \forall \boldsymbol{v}_h \in \boldsymbol{V}_h.
  \end{equation}
\end{lemma}
\begin{proof}
  The proof follows the same steps as the proof of \cite[Lemma
  5]{rhebergen2018preconditioning} and is therefore omitted.
\end{proof}

\begin{theorem}
  \label{lem:schur-comp-equiv}
  Let $A$ and $P$ be operators with block structures as defined in
  \cref{ss:preconframeworkscold}. Assume there exist uniform constants
  $c_1, c_2 > 0$ such that
  \begin{equation}
    \label{eq:schur-comp-equiv-1}
    c_1 \norm[0]{\boldsymbol{x}_h}_{\boldsymbol{X}_h}^2
    \leq a_h(\boldsymbol{x}_h , \boldsymbol{x}_h)
    \leq c_2 \norm[0]{\boldsymbol{x}_h}_{\boldsymbol{X}_h}^2
    \quad \forall \boldsymbol{x}_h \in \boldsymbol{X}_h.
  \end{equation}
  Then
  \begin{equation}
    \label{eq:schur-comp-equiv-2}
    c_1 \norm[0]{\bar{x}_h}_{\bar{X}_h}^2
    \leq \langle S_A \bar{x}_h, \bar{x}_h \rangle_{\bar{X}_h^*, \bar{X}_h}
    \leq c_2 \norm[0]{\bar{x}_h}_{\bar{X}_h}^2
    \quad \forall \bar{x}_h \in \bar{X}_h,
  \end{equation}
  where
  $\norm[0]{\bar{x}_h}^2 := \langle S_P \bar{x}_h,
  \bar{x}_h\rangle_{\bar{X}_h^*, \bar{X}_h}$ and
  $S_P := P_{22} - P_{21} P_{11}^{-1} P_{21}^T$.
\end{theorem}
\begin{proof}
  We begin by proving the lower bound in
  \cref{eq:schur-comp-equiv-2}. Writing out
  \cref{eq:schur-comp-equiv-1} in operator notation we note that
  \begin{align*}
    \langle A \boldsymbol{x}_h, \boldsymbol{x}_h \rangle_{\boldsymbol{X}_h^*, \boldsymbol{X}_h}
    & = \langle A_{11} (x_h + A_{11}^{-1} A_{21}^T \bar{x}_h), 
      x_h + A_{11}^{-1} A_{21}^T \bar{x}_h \rangle _{X_h^*, X_h}
      + \langle S_A \bar{x}_h, \bar{x}_h \rangle_{\bar{X}_h^*, \bar{X}_h}
    \\
    &\geq c_1 \big( \langle P_{11} (x_h + P_{11}^{-1} P_{21}^T \bar{x}_h), 
      x_h + P_{11}^{-1} P_{21}^T \bar{x}_h \rangle _{X_h^*, X_h}
      + \langle S_P \bar{x}_h, \bar{x}_h \rangle_{\bar{X}_h^*, \bar{X}_h}\big).
  \end{align*}
  Choosing $x_h = - A_{11}^{-1} A_{21}^T \bar{x}_h$, we find
  \begin{align*}
    \langle S_A \bar{x}_h, \bar{x}_h\rangle_{\bar{X}_h^*, \bar{X}_h}
    \geq& c_1 \big( \langle P_{11}( - A_{11}^{-1} A_{21} \bar{x}_h 
      + P_{11} P_{21}^T \bar{x}_h), 
      - A_{11}^{-1} A_{21} \bar{x}_h + P_{11} P_{21}^T \bar{x}_h\rangle_{X_h^*, \bar{X}_h}
    \\
    &
      + \langle S_P \bar{x}_h, \bar{x}_h\rangle_{\bar{X}_h^*, \bar{X}_h}\big)
    \\
    \geq& c_1 \langle S_P \bar{x}_h, \bar{x}_h\rangle_{\bar{X}_h^*, \bar{X}_h} 
      = c_1 \norm[0]{\bar{x}_h}_{\bar{X}_h}^2,
  \end{align*}
  where the last inequality holds because $P_{11}$ is a positive
  operator.

  We now prove the upper bound in
  \cref{eq:schur-comp-equiv-2}. Writing out
  \cref{eq:schur-comp-equiv-1} in operator notation we note that
  \begin{align*}
    c_2 \norm[0]{\boldsymbol{x}_h}_{\boldsymbol{X}_h}^2
    & = c_2 \big( \langle P_{11} (x_h + P_{11}^{-1} P_{21}^T \bar{x}_h),  
      x_h + P_{11}^{-1} P_{21}^T \bar{x}_h \rangle_{X_h^*, X_h}
      + \langle S_P \bar{x}_h, \bar{x}_h\rangle_{\bar{X}_h^*, \bar{X}_h}\big)
    \\
    & \geq \langle A_{11} (x_h + A_{11}^{-1} A_{21}^T \bar{x}_h),  
      x_h + A_{11}^{-1} A_{21}^T \bar{x}_h \rangle_{X_h^*, X_h}
      + \langle S_A \bar{x}_h, \bar{x}_h\rangle_{\bar{X}_h^*, \bar{X}_h}.
  \end{align*}
  Choosing $x_h = - P_{11}^{-1} P_{21}^T \bar{x}_h$, we find
  \begin{align*}
    c_2 \norm[0]{\bar{x}_h}_{\bar{X}_h}^2
    = & \langle S_P \bar{x}_h, \bar{x}_h\rangle_{\bar{X}_h^*, \bar{X}_h}
    \\
    \geq & \langle A_{11} (- P_{11}^{-1} P_{21}^T \bar{x}_h
           + A_{11}^{-1} A_{21}^T \bar{x}_h), - P_{11}^{-1} P_{21}^T \bar{x}_h
           + A_{11}^{-1} A_{21}^T \bar{x}_h \rangle_{X_h^*, X_h}
    \\
      & + \langle S_A \bar{x}_h, \bar{x}_h \rangle_{\bar{X}_h^*, \bar{X}_h}
    \\
    \geq & \langle S_A \bar{x}_h, \bar{x}_h \rangle_{\bar{X}_h^*, \bar{X}_h},
  \end{align*}
  where the last inequality holds because $A_{11}$ is a positive
  operator.
\end{proof}

\section{Auxiliary problems}

\subsection{Auxiliary problem for the pressure field}
\label{ap:auxprobpressure}

Consider the following diffusion problem for the pressure:
\begin{equation*}
  - \nabla \cdot (\tau^{-1} \nabla p) = g \text{ in }\Omega,
  \quad
  \nabla p \cdot n = 0 \text{ on } \partial \Omega, 
  \quad
  \int_{\Omega} p \dif x = 0,
\end{equation*}
for some source term $g$. The HDG discretization of this problem is
given by (see \cite{wells2011analysis}): Given $g \in L^2(\Omega)$,
find $\boldsymbol{p}_h \in \boldsymbol{Q}_h$ such that
\begin{equation}
  \label{eq:auxiliary-problem-p}
  \tilde{a}_h(\boldsymbol{p}_h, \boldsymbol{q}_h) 
  = (g, q_h)_{\mathcal{T}_h}
  \quad \forall \boldsymbol{q}_h \in \boldsymbol{Q}_h,
\end{equation}
where
\begin{equation}
  \label{eq:bili-form-tildea}
  \begin{split}
    \tilde{a}_h(\boldsymbol{p}_h, \boldsymbol{q}_h)
    :=& \tau^{-1} (\nabla p_h, \nabla q_h)_{\mathcal{T}_h}
        + \tau^{-1} \eta \langle h_K^{-1} (p_h - \bar{p}_h),
        q_h - \bar{q}_h \rangle_{\partial\mathcal{T}_h}
    \\
      & - \tau^{-1} \langle \nabla p_h \cdot n, q_h - \bar{q}_h\rangle_{\partial \mathcal{T}_h}
        - \tau^{-1} \langle \nabla q_h \cdot n, p_h - \bar{p}_h\rangle_{\partial \mathcal{T}_h}.    
  \end{split}
\end{equation}
By \cite[Lemmas 5.2 and 5.3]{wells2011analysis} this bilinear form is
such that there exist constants $\tilde{c}_1, \tilde{c}_2 > 0$ such
that
\begin{equation}
  \label{eq:bili-form-tildea-equivalence}
  \tilde{c}_1 \tau^{-1} \tnorm{\boldsymbol{q}_h}_{q,1}^2
  \leq \tilde{a}_h(\boldsymbol{q}_h, \boldsymbol{q}_h)
  \leq \tilde{c}_2 \tau^{-1} \tnorm{\boldsymbol{q}_h}_{q,1}^2 
  \quad \forall \boldsymbol{q}_h \in \boldsymbol{Q}_h.
\end{equation}

Local solvers associated with \cref{eq:auxiliary-problem-p} are similar
to those of \cref{def:localsolver-a0}: Given $\bar{t}_h \in \bar{Q}_h$
and $s \in L^2(\Omega)$, the function
$\tilde{p}_h^L(\bar{t}_h, s) \in Q_h$ satisfies the following problem
restricted to a cell $K$:
\begin{equation*}
  \tilde{a}_h^K(\tilde{p}_h^L, q_h) = \tilde{g}_h^K(q_h)
  \quad \forall q_h \in Q(K),
\end{equation*}
where
\begin{align*}
  \tilde{a}_h^K(p_h, q_h)
  :=& \tau^{-1} (\nabla p_h, \nabla q_h)_K
      + \tau^{-1} \eta h_K^{-1} \langle p_h, q_h \rangle_{\partial K}
  \\
    & - \tau^{-1} \langle \nabla p_h \cdot n, q_h \rangle_{\partial K}
      - \tau^{-1} \langle \nabla q_h \cdot n, p_h \rangle_{\partial K},
  \\
  \tilde{g}_h^K(q_h)
  :=& (s, q_h)_K
      + \tau^{-1} \eta h_K^{-1} \langle q_h, \bar{t}_h \rangle_{\partial K}
      - \tau^{-1} \langle \nabla q_h \cdot n, \bar{t}_h \rangle_{\partial K}.
\end{align*}
Furthermore, similar to \cref{lem:condensed-formulation}, we have the
following reduced formulation of \cref{eq:auxiliary-problem-p} from
which $p_h$ has been eliminated from the system.

\begin{lemma}[Reduced auxiliary pressure problem]
  \label{lem:redauxproblem}
  Given $g \in L^2(\Omega)$, define $\tilde{p}_h^g :=
  \tilde{p}_h(0,g)$. Furthermore, define
  $\tilde{l}_p(\bar{q}_h) := \tilde{p}_h(\bar{q}_h, 0)$ for all
  $\bar{q}_h \in \bar{Q}_h$. Let $\bar{p}_h \in \bar{Q}_h$ be the
  solution to
  \begin{equation}
    \label{eq:condensed-auxiliary-problem-p}
    \tilde{a}_h((\tilde{l}_p(\bar{p}_h), \bar{p}_h), (\tilde{l}_p(\bar{q}_h), \bar{q}_h))
    = (g, \tilde{l}_p(\bar{q}_h))_{\mathcal{T}_h}
    \quad \forall \bar{q}_h \in \bar{Q}_h.
  \end{equation}
  Then $(p_h, \bar{p}_h)$, in which
  $p_h = \tilde{p}_h^g + \tilde{l}_p(\bar{p}_h)$, solves
  \cref{eq:auxiliary-problem-p}.
\end{lemma}

\subsection{Auxiliary problem for the velocity field}
\label{ap:auxprobvelocity}

Consider the following vector reaction-diffusion problem:
\begin{equation*}
  - \nabla \cdot (\nu \nabla u) + \tau u = f \text{ in }\Omega,
  \quad
  u = 0 \text{ on } \partial \Omega.
\end{equation*}
Its HDG discretization is given by: Given $f \in [L^2(\Omega)]^d$
find $\boldsymbol{u}_h \in \boldsymbol{V}_h$ such that
\begin{equation}
  \label{eq:auxiliary-problem-u}
  \tilde{d}_h(\boldsymbol{u}_h, \boldsymbol{v}_h):=
  d_h(\boldsymbol{u}_h, \boldsymbol{v}_h)
  + \tau (u_h, v_h)_{\mathcal{T}_h}
  = (f, v_h)_{\mathcal{T}_h}
  \quad \forall \boldsymbol{v}_h \in \boldsymbol{V}_h,
\end{equation}
with $d_h(\cdot, \cdot)$ defined in \cref{eq:bilinear-forms-a}. Note
that by \cref{eq:coercivity,eq:boundedness-aux1} we have that
\begin{equation}
  \label{eq:spec-equiv-tilde-d}
  c_c \tnorm{\boldsymbol{v}_h}_v^2 
  \leq \tilde{d}_h(\boldsymbol{v}_h, \boldsymbol{v}_h)
  \leq c_d \tnorm{\boldsymbol{v}_h}_v^2 
  \quad \forall \boldsymbol{v}_h \in \boldsymbol{V}_h.
\end{equation}
The local solvers associated with the discretization
\cref{eq:auxiliary-problem-u} are defined as follows: Given
$\bar{m}_h \in \bar{V}_h$ and $s \in [L^2(\Omega)]^d$, we define the
function $\tilde{u}_h^L(\bar{m}_h, s) \in V_h$ such that when
restricted to a cell $K$ it satisfies
\begin{equation}
  \label{eq:local-solver-reac-diff-u}
  \tilde{d}_h^K(\tilde{u}_h^L, v_h)
  = \tilde{f}_h^K(v_h)
  \quad \forall v_h \in V(K),
\end{equation}
where $V(K) := [\mathbb{P}_k(K)]^d$ and
\begin{align*}
  \tilde{d}_h^K(u_h, v_h)
  :=& \tau (u_h, v_h)_K 
      + \nu (\nabla u_h, \nabla v_h)_K
      + \nu \eta h_K^{-1} \langle u_h, v_h \rangle_{\partial K}
  \\
    & - \nu \langle \nabla u_h n, v_h \rangle_{\partial K}
      - \nu \langle \nabla v_h n, u_h \rangle_{\partial K},
  \\
  \tilde{f}_h^K(q_h)
  :=& (s, v_h)_K
      - \nu \langle \nabla v_h n, \bar{m}_h \rangle_{\partial K}
      + \nu \eta h_K^{-1} \langle \bar{m}_h, v_h \rangle_{\partial K}.
\end{align*}
Similar to \cref{lem:condensed-formulation}, we have the following
reduced formulation of \cref{eq:auxiliary-problem-u} from which $u_h$
has been eliminated from the system.
\begin{lemma}[Reduced auxiliary velocity problem]
  \label{lem:reduced-reac-dif-u}
  Given $f \in [L^2(\Omega)]^d$, define
  $\tilde{u}_h^f := \tilde{u}_h^L(0, f)$ and
  $\tilde{l}_u(\bar{v}_h) := \tilde{u}_h^L(\bar{v}_h, 0)$ for all
  $\bar{v}_h \in \bar{V}_h$.  Let $\bar{u}_h$ be the solution to
  \begin{equation}
    \label{eq:reduced-reac-dif-u}
    \tilde{d}_h((\tilde{l}_u(\bar{u}_h), \bar{u}_h), (\tilde{l}_u(\bar{v}_h), \bar{v}_h))
    = (f, \tilde{l}_u(\bar{v}_h))_{\mathcal{T}_h}
    \quad \forall \bar{v}_h \in \bar{V}_h.
  \end{equation}
  Then $(u_h, \bar{u}_h)$, in which
  $u_h = \tilde{u}_h^f + \tilde{l}_u(\bar{u}_h)$, solves
  \cref{eq:auxiliary-problem-u}.
\end{lemma}

\end{document}